\documentclass[11pt]{article}
\usepackage{geometry}
\usepackage{soul}
\usepackage{amsfonts,amsmath,amssymb,amsthm}
\usepackage{dsfont,verbatim}
\usepackage[font=footnotesize,labelformat=simple]{subcaption}
\usepackage{url}
\usepackage[ruled,vlined]{algorithm2e}

\usepackage{pgfplots}
\usepackage{hyperref}

\newtheorem{theorem}{Theorem}
\newtheorem{corollary}[theorem]{Corollary}
\newtheorem{lemma}[theorem]{Lemma}
\newtheorem{prop}[theorem]{Proposition}

\theoremstyle{definition}
\newtheorem{definition}[theorem]{Definition}
\newtheorem{example}[theorem]{Example}

\newtheorem{rem}{Remark}

\newcommand{\EE}{\mathbb{E}}

\newcommand{\PP}{\mathbb{P}}
\newcommand{\RR}{\mathbb{R}}

\newcommand{\ee}{\varepsilon}
\newcommand{\dint}{\mathrm{d}}

\begin{document}

\title{Statistical Advantages of Oblique Randomized Decision \\ Trees and Forests}
\author{Eliza O'Reilly}

\date{}

%\subjclass[2000]{Primary  60D05; Secondary 62G07}

\maketitle

\begin{abstract}

This work studies the statistical implications of using features comprised of general linear combinations of covariates to partition the data in randomized decision tree and forest regression algorithms. Using random tessellation theory in stochastic geometry, we provide a theoretical analysis of a class of efficiently generated random tree and forest estimators that allow for oblique splits along such features. We call these estimators \emph{oblique Mondrian} trees and forests, as the trees are generated by first selecting a set of features from linear combinations of the covariates and then running a Mondrian process that hierarchically partitions the data along these features. Quadratic risk bounds and convergence rates are obtained for the flexible function class of multi-index models for dimension reduction, where the output is assumed to depend on a low-dimensional relevant feature subspace of the input domain. The results highlight how the risk of these estimators depends on the choice of features and quantify how robust the risk is with respect to error between the selected features along which the data is split and the true relevant feature subspace. The asymptotic analysis also provides conditions on the convergence rate a set of estimated relevant features must satisfy for oblique Mondrian estimators to obtain minimax optimal rates of convergence with respect to the dimension of the relevant feature subspace. Additionally, a lower bound on the risk of axis-aligned Mondrian trees (where features are restricted to the set of covariates) is obtained, proving that these estimators are suboptimal for general ridge functions, no matter how the distribution over the covariates used to divide the data at each tree node is weighted.
\end{abstract}

\section{Introduction}

Random forests are a widely used class of machine learning algorithms that achieve competitive performance for many tasks \cite{chen2012random, fernandez2014we}. The original algorithm popularized by Breiman \cite{breiman2001random}, and influenced by the work of Amit and Geman \cite{amit_shape_1997} and Ho \cite{ho_random_1998}, remains highly valued for its relative interpretability and ability to handle large datasets with high dimensionality. %Recently,
There has also been a recent surge in progress in understanding the statistical properties of Breiman's random forest including consistency rates in fixed and high dimensional settings \cite{ScornetConsistency2015, ChietalnewRates2021, Syrgkanis2020, KlusowskiTian2024}. The algorithm is an ensemble method, outputting predictions that average the predictions across a collection of randomized decision trees. Each tree recursively splits the training data using a set of features of the input and a prediction for a new input is determined by the labels of the training data lying in the same leaf of the tree, or equivalently, the same cell of the random hierarchical partition of the input space generated by the splits. 

Random forests most commonly used in practice are restricted to axis-aligned splits, where only one dimension, or covariate, of the input data is used to partition the data in a given node of the tree. This generates random partitions of the input space made up of cells that are axis-aligned boxes, producing step-wise decision boundaries. The geometry of axis-aligned partitions limits the model's ability to capture dependencies between dimensions of the input, and the corresponding theory and consistency rates have generally been limited to the assumption that the regression function comes from an additive model. Oblique random forests are variants of the algorithm where splits are allowed to depend on linear combinations of the covariates. There have been many approaches for choosing these split directions and the resulting estimators have shown improved empirical performance in a variety of settings over axis-aligned versions \cite{breiman2001random, blaser2016random, pmlr-v84-fan18b, Menzeetal2011, rainforth2015canonical, Tomita2020}. Some recent work \cite{Klusowski_oblique_2023, Zhan2022ConsistencyOO} has also obtained convergence rates for oblique random trees utilizing the CART methodology of Breiman's random forest under the assumption of additive single-index regression models. However, theoretical guarantees for these variants remain scarce and a complete understanding of the statistical advantages of oblique splits over axis-aligned versions is lacking.

There are many difficulties in analyzing Breiman's original random forest algorithm due to the complex dependence of the partitioning scheme on the inputs and labels of the training dataset. To overcome these challenges in the axis-aligned case, simplified versions of the algorithm where the splits do not use the labels of the data have also been studied, including centered random forests \cite{Biau2012} and median random forests \cite{DurouxScornet2018}. Both of these variants, however, have since been shown to be minimax suboptimal for input dimensions greater than one \cite{Klusowski2021}. The first random forest variant for which minimax optimal convergence rates were obtained in arbitrary dimension is the \emph{Mondrian random forest} \cite{mourtada2020minimax}, where component trees are generated by a Mondrian process \cite{roy2008mondrian}. Recent work \cite{KlusowskiMondrian2023} has also proved a central limit theorem for Mondrian forest point estimators and shown that a debiased variant of Mondrian forests can achieve minimax rates for general H\"{o}lder classes.

Given the amenability of the Mondrian partitioning mechanism to theoretical analysis, a natural direction for studying oblique random forests is to study variants of the Mondrian process that use linear combinations of covariates to make splits. Fortunately, the Mondrian process is a special case of the general class of \emph{stable under iteration} (STIT) processes in stochastic geometry introduced by Nagel and Weiss \cite{Nagel2005, Nagel2008}. STIT processes all satisfy properties such as spatial consistency and the Markov property that are attractive about the Mondrian process, but form a much more general class of stochastic hierarchical partitioning processes indexed by a probability measure on the unit sphere called a \emph{directional distribution} that governs the distribution of split directions. Utilizing STIT processes to generate randomized decision trees thus forms a rich and flexible class of oblique random forests. This class of algorithms, called \emph{random tessellation forests}, has been studied empirically in \cite{TehRTFs2019} and the theory of random tessellations in stochastic geometry has been used in \cite{OReillyTran2021, OReillyTran2021minimax} to provide a theoretical framework for the use of these STIT processes in machine learning applications.
In particular, the results of \cite{OReillyTran2021minimax} extend the minimax rates obtained for Mondrian forests to random tessellation forests for any fixed directional distribution. These were the first minimax optimality guarantees for random forest variants with oblique splits. However, these worst-case risk bounds for Lipschitz and $\mathcal{C}^2$ functions do not illustrate an advantage of random tessellation forests with oblique splits over Mondrian forests. The rates in \cite{OReillyTran2021minimax} also suffer from the curse of dimensionality when the input is not contained in a low-dimensional subspace, becoming very slow when the ambient dimension of the input is large. 

In this paper, we address these theoretical limitations by studying how this choice of directional distribution allows random tessellation trees and forests to adapt to a flexible class of dimension reduction models. 
This effort shows the potential of these models to overcome the curse of dimensionality with a data-driven choice of split directions and 
establishes a statistical advantage of employing oblique splits in random forest regression. 
Prior results on the adaptation of random forests to low dimensional structure have focused on the axis-aligned setting and adaptation to \emph{sparse} regression functions, where the output only depends on a small number of covariates relative to the ambient dimension. 
This work establishes that with a good choice of the directional distribution governing the directions of the hyperplane splits, random tessellation forests have the potential to adapt to the more general dimension reduction class of \emph{multi-index models}. Multi-index models are those for which the output only varies with respect to changes of the input in directions relative to a low dimensional subspace of $\RR^d$, called the \emph{relevant feature subspace}. These regression model classes are as general as those studied for two-layer neural networks \cite{Bach2017_curse}, laying additional groundwork for theoretical comparison of the statistical properties of random forests versus neural networks. 

Our specific contributions are the following. We first obtain a general upper bound for the risk of random tessellation trees and forests when the underlying regression function comes from a multi-index model. These bounds illuminate how the risk of the estimator is controlled by the geometry of convex bodies associated with the random tessellation model projected onto the relevant feature subspace (see Theorem \ref{t:Lipschitz_Rate} and Proposition \ref{t:C2_Rate}). 
Next, we restrict to studying random tessellation trees and forests generated by STIT processes where the directional distribution is discrete. We will call these estimators \emph{oblique Mondrian trees and forests} because they can be obtained by first applying a linear transformation to the data to obtain a new set of features from linear combinations of covariates, and then running a Mondrian process (see Section \ref{sec:Mondrian_Transform}). Our results include an upper bound on the risk of the estimators controlled by constants quantifying how close the linear transformation is to a projection onto the relevant feature subspace. These bounds quantify how robust the estimator is to a measure of the approximation error of relevant features (see Theorems \ref{t:C1_Transformed_Mondrain_Bound} and \ref{t:C2_Transformed_Mondrain_Bound}). We also establish rates of convergence of the estimators as the amount of data increases that depend on how this approximation error of relevant features converges to zero, including sufficient asymptotic conditions under which, with proper tuning of complexity parameters, minimax rates of convergence depending only on the dimension of the relevant feature subspace are obtained (see Corollaries \ref{cor:C1_Transformed_Mondrain_Rate} and \ref{cor:C2_Transformed_Mondrain_Rate}).

Finally, we obtain a suboptimality result for axis-aligned randomized decision trees. 
Indeed, while our first collection of results shows that oblique Mondrian trees have the potential to obtain improved rates of convergence for multi-index models over those for general Lipschitz functions with a well-chosen set of split directions and weights, we also obtain a risk \emph{lower bound} for axis-aligned Mondrian trees showing that for \emph{any} choice of weights over the covariates, the axis-aligned splits \emph{cannot} achieve such improved rates for general ridge functions (see Theorem \ref{thm:Mondrian_suboptimal}). 

\subsection{Outline}

The remainder of this paper is organized as follows. Section \ref{sec:background} covers the relevant definitions and background from stochastic and convex geometry needed to prove our results. Section \ref{sec:setting} describes the problem setting and notation followed by risk upper bounds for general random tessellation trees and forests when the underlying regression function comes from a multi-index model. Section \ref{sec:main_results} presents our main results on convergence rates for oblique Mondrian forests, and Section \ref{sec:Mondrian_forests} considers the special case of axis-aligned weighted Mondrian forests and sparse regression models. Section \ref{sec:suboptimality} presents our final main result on the suboptimality of weighted Mondrian forests for general ridge functions. Crucial to our main results is the observation that oblique Mondrian processes obtained through a linear transformation of the data and a Mondrian process is equivalent to partitioning with a STIT process with a particular discrete directional distribution and this is stated and proved in Section \ref{sec:Mondrian_Transform}. Finally, Section \ref{sec:conclusion} concludes with a discussion of the results and future work, and Section \ref{sec:proofs} collects some of the proofs of our main results. The remaining proofs are contained in the Appendix.

\section{Background}\label{sec:background}

In this section, we briefly describe the necessary definitions and other background from stochastic geometry and convex geometry needed for the statements and proofs of our results. In the following, we will denote by $\kappa_k$ the volume of the unit $\ell_2$ ball $B^k$ in $\RR^k$ for $k \in \mathbb{N}$. 

\subsection{Stable Under Iteration (STIT) Tessellations}

A random tessellation $\mathcal{P}$ of $\RR^d$ is a point process of compact convex polytopes $\{C_i\}_{i \in \mathbb{N}}$ in $\RR^d$ such that almost surely, $\cup_i C_i = \RR^d$ and $\mathrm{int}(C_i) \cap \mathrm{int}(C_j) = \emptyset$ for all $i \neq j$. These polytopes will be referred to as the \emph{cells} of the tessellation in the following. A random tessellation is \emph{stationary} if the distribution of $\mathcal{P}$ is invariant under translations in $\RR^d$. 

The \emph{iteration} of a random tessellation is the process of subdividing each cell of the tessellation by an independent copy of the random tessellation restricted to that cell. A random tessellation is \emph{stable under iteration} (STIT) if for all $n$, iterating $n$ times and scaling all the boundaries by $n$ recovers in distribution the original random tessellation. 

The distribution of a stationary STIT tessellation of $\RR^d$ is determined by a parameter $\lambda > 0$ called the \emph{lifetime} and 
an even probability measure $\phi$ on $\mathbb{S}^{d-1}$ called the \emph{directional distribution}, which governs the distribution of the normal directions of the hyperplane splits used to generate the tessellation. A probability measure $\phi$ on the sphere is even if $\phi(B) = \phi(-B)$ for all $B \in \mathcal{B}(\mathbb{S}^{d-1})$.
The following procedure describes the stochastic \emph{STIT process} on a compact window $W \subset \RR^d$, which constructs a STIT tessellation restricted to $W$ with lifetime $\lambda$ and directional distribution $\phi$:
\begin{enumerate}
    \item Sample an exponential clock $\delta$ with parameter
\begin{align*}
\int_{\mathbb{S}^{D-1}} \left(h_W(u) + h_W(-u)\right) \dint \phi(u), 
\end{align*}
where $h_W(u) := \sup_{x \in W} \langle u,x \rangle$ is the support function of $W$. 
\item If $\delta > \lambda$, stop. Else, at time $\delta$, generate a random hyperplane $$H(U,T) := \{x \in \RR^d : \langle x, U\rangle = T\},$$ 
where the unit normal direction $U$ is drawn from the distribution
\[\dint \Phi(u) := \frac{h_W(u) + h_W(-u)}{\int_{\mathbb{S}^{D-1}} \left(h_W(u) + h_W(-u)\right) \dint \phi(u)}\dint\phi(u), \quad u \in \mathbb{S}^{D-1},\]
and conditioned on $U$, $T$ is drawn uniformly on the interval from $-h_W(-U)$ to $h_W(U)$ defining the width of $W$ in direction $u$. 
Split $W$ into two cells $W_1$ and $W_2$ with $H \cap W$.
\item Repeat steps (1) and (2) in each sub-window $W_1$ and $W_2$ independently with new lifetime parameter $\lambda - \delta$ until lifetime expires.
\end{enumerate}
Note that the lifetime $\lambda$ governs the complexity of the resulting STIT tessellation; the larger $\lambda$ is, the longer the process will run and more cells will be generated. When $\phi$ is the uniform distribution over the standard (signed) basis vectors in $\RR^d$, the corresponding STIT process has the same distribution as the Mondrian process \cite{roy2008mondrian}. 

We refer to \cite{Nagel2005} for the proof of the existence of STIT tessellations on $\RR^d$ and some of their properties, one of which we recall here.  For a STIT tessellation $\mathcal{P}(\lambda)$ with lifetime $\lambda > 0$, let $\mathcal{Y}(\lambda)$ denote the union of boundaries of the polytopes. The STIT property implies the following useful \emph{scaling property} of STIT tessellations: \begin{align}\label{e:scaling}
\lambda\mathcal{Y}(\lambda) \overset{(d)}{=} \mathcal{Y}(1).    
\end{align}

\subsubsection{Cells of stationary random tessellations}

Let $Z_x^{\lambda}$ be the cell containing $x \in \RR^d$ of a stationary STIT tessellation with lifetime $\lambda > 0$. The cell $Z_0^{\lambda}$ containing the origin is called the \emph{zero cell}. By stationarity and the property \eqref{e:scaling},
\begin{align}\label{e:scaling_and_translate}
Z_x^{\lambda} \overset{(d)}{=} \frac{1}{\lambda}Z_0 + x, \text{ for all } x \in \RR^d,
\end{align}
where $Z_0 := Z_0^1$ denotes the zero cell of the STIT tessellation with unit lifetime. Another random polytope associated with a stationary STIT tessellation is called the \emph{typical cell}. To define this, first let $\mathcal{K}$ denote the space of compact and convex polytopes in $\RR^d$ and let $c: \mathcal{K} \to \RR^d$ be a function that assigns a ``center" to each polytope $K \in \mathcal{K}$ such that $c(K + x) = c(K) + x$ for all $x \in \RR^d$. Now let $\mathcal{K}_0 := \{ K \in \mathcal{K} : c(K) = 0\}$. The typical cell $Z$ of a stationary random tessellation $\mathcal{P}$ is the random polytope in $\mathcal{K}_0$ such that 
for any non-negative measurable function $f$ on $\mathcal{K}$,
\begin{align}\label{e:campbell}
    \EE\left[\sum_{C \in \mathcal{P}} f(C)\right] = \frac{1}{\EE[\mathrm{vol}_D(Z)]}\EE\left[\int_{\RR^d} f(Z + y) \dint y\right].%,
\end{align}
The above equality is a special case of Campbell's theorem applied the the stationary point process of convex polytopes that make up the cells of the random tessellation. We refer to \cite[Section~4.1]{weil} for further details.

\subsubsection{Associated zonoid}

There is a rich connection between STIT tessellations and the geometry of convex bodies. In particular, the class of STIT tessellations in $\RR^d$ have a one-to-one correspondence to a subset of convex bodies in $\RR^d$ called \emph{zonoids} \cite{Schneider1983Zonoids}. This class of convex bodies are those that can be approximated by finite Minkowksi sums of line segments with respect to Hausdorff distance. Recall the Minkowksi sum $K + L$ of two convex bodies $K$ and $L$ in $\RR^d$ is defined by
\[K + L := \{x + y : x \in K, y \in L\} \subseteq \RR^d.\]
A convex body $\Pi$ in $\RR^d$ is a zonoid if and only if it has support function of the form $h_{\Pi}(u) = \int_{\mathbb{S}^{d-1}} |\langle u, v \rangle| \dint \mu(v)$ for some finite positive measure $\mu$ on the unit sphere. We can thus define a particular zonoid for a STIT tessellation through its directional distribution.

\begin{definition}
The \emph{normalized associated zonoid} of a STIT process in $\RR^d$ with directional distribution $\phi$ is the zonoid with support function
\begin{align}\label{e:Pi_support_fxn}
    h_{\Pi}(u) := \frac{1}{2}\int_{\mathbb{S}^{d-1}}|\langle u,v \rangle| \dint \phi(v).
\end{align}    
\end{definition}
In the sequel we will use the the following known fact (see \cite{Nagel2003} and \cite[(10.4) and (10.44)]{weil}):
\begin{align}\label{e:EVZ}
\EE[\mathrm{vol}_d(Z)] = \frac{1}{\mathrm{vol}_d(\Pi)},
\end{align}
where $Z$ is the typical cell of a STIT process with lifetime $1$ and normalized associated zonoid $\Pi$.

\begin{example}
An \emph{isotropic} STIT process is obtained by taking the directional distribution to be $\phi \sim \text{Uniform}(\mathbb{S}^{d-1})$. In this case, the normalized associated zonoid $\Pi = c_dB^d$ is an $\ell_2$ ball radius $c_d := \frac{\Gamma(\frac{d}{2})}{2\sqrt{\pi}\Gamma(\frac{d+1}{2})}$.
\end{example}

\begin{example}\label{ex:PI_Mondrian}
The Mondrian process in $\RR^d$ is a special case of a STIT process when the directional distribution is given by $\phi = \frac{1}{2 d}\sum_{i=1}^d \left(\delta_{e_i} + \delta_{-e_i}\right)$, where $\{e_i\}_{i=1}^d$ is the standard orthonormal basis in $\RR^d$. The normalized associated zonoid is the $\ell^{\infty}$ ball
\begin{align*}
    \Pi = [-e_1/2d, e_1 /2d] + \cdots + [-e_d/2d, e_d /2d].
\end{align*}
When the unit basis directions are given more general weights, i.e. $\phi = \sum_{i=1}^d \frac{\omega_i}{2} \left(\delta_{e_i} + \delta_{-e_i}\right)$ where $\sum_{i=1}^d \omega_i = 1$ and $\omega_i > 0$ for all $i$, then the normalized associated zonoid is the hyperrectangle
\begin{align*}
    \Pi = [-\omega_1 e_1/2, \omega_1 e_1 /2] + \cdots + [-\omega_d e_d/2, \omega_d e_d /2],
\end{align*}
and we call the associated STIT process a \emph{weighted Mondrian} process.
\end{example}

\begin{example}\label{ex:PI_oblique_Mondrian}
A general discrete directional distribution on $\mathbb{S}^{d-1}$ has the form $\phi = \sum_{i=1}^m \frac{\omega_i}{2}\left(\delta_{u_i} + \delta_{-u_i}\right)$ for some $m \geq d$, where the weights $\{\omega_i\}_{i=1}^m$ satisfy $\omega_i > 0$ and $\sum_{i=1}^m \omega_i = 1$ and the directions $u_i \in \mathbb{S}^{d-1}$ for $i =1 , \ldots, m$ span all of $\RR^d$. Then, the normalized associated zonoid is given by
\[\Pi = [-\omega_1 u_1/2, \omega_1 u_1 /2] + \cdots + [-\omega_m u_m/2, \omega_m u_m /2],\]
i.e. it is the Minkowski sum of $m$ line segments. In this case, we refer to the corresponding STIT process as an \emph{oblique Mondrian} process.
\end{example}

\begin{figure}
    \centering
    \begin{subfigure}[b]{0.45\textwidth}
     \centering
    \includegraphics[width=.93\textwidth]{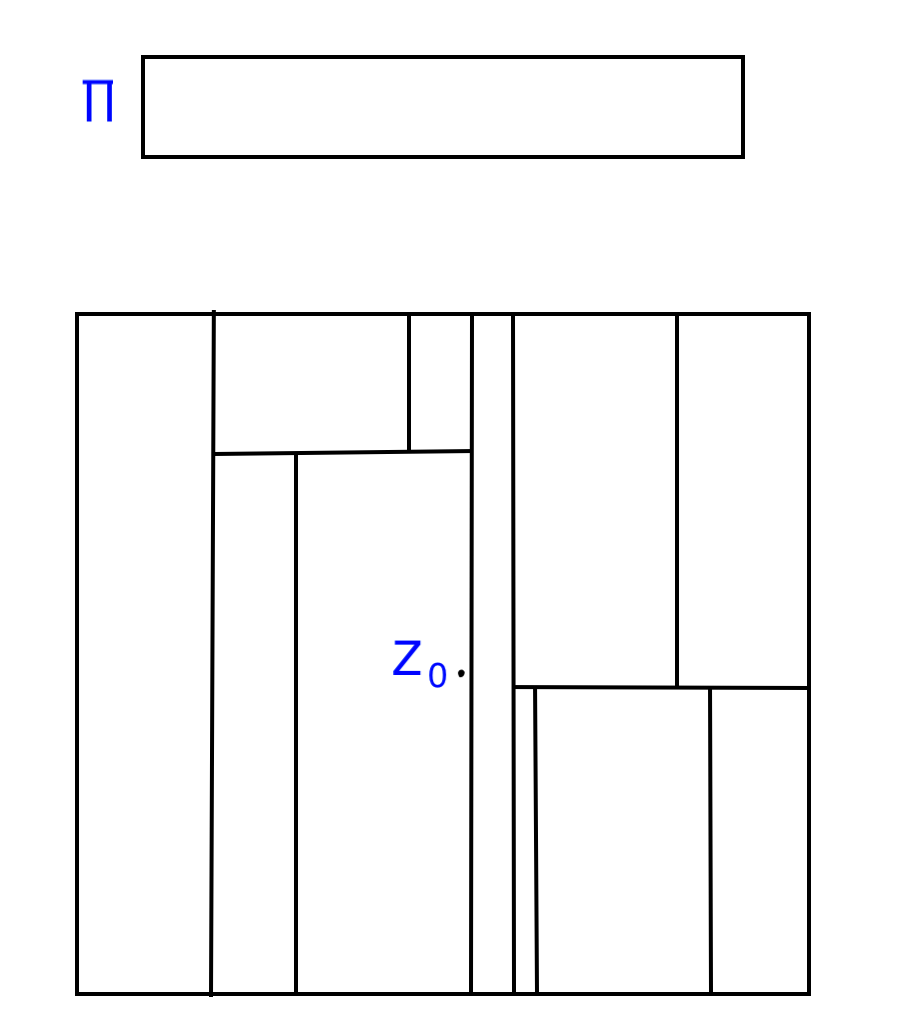}
    \caption{Weighted Mondrian}
    \end{subfigure}
    \begin{subfigure}[b]{0.45\textwidth}
     \centering
    \includegraphics[width=.8\textwidth]{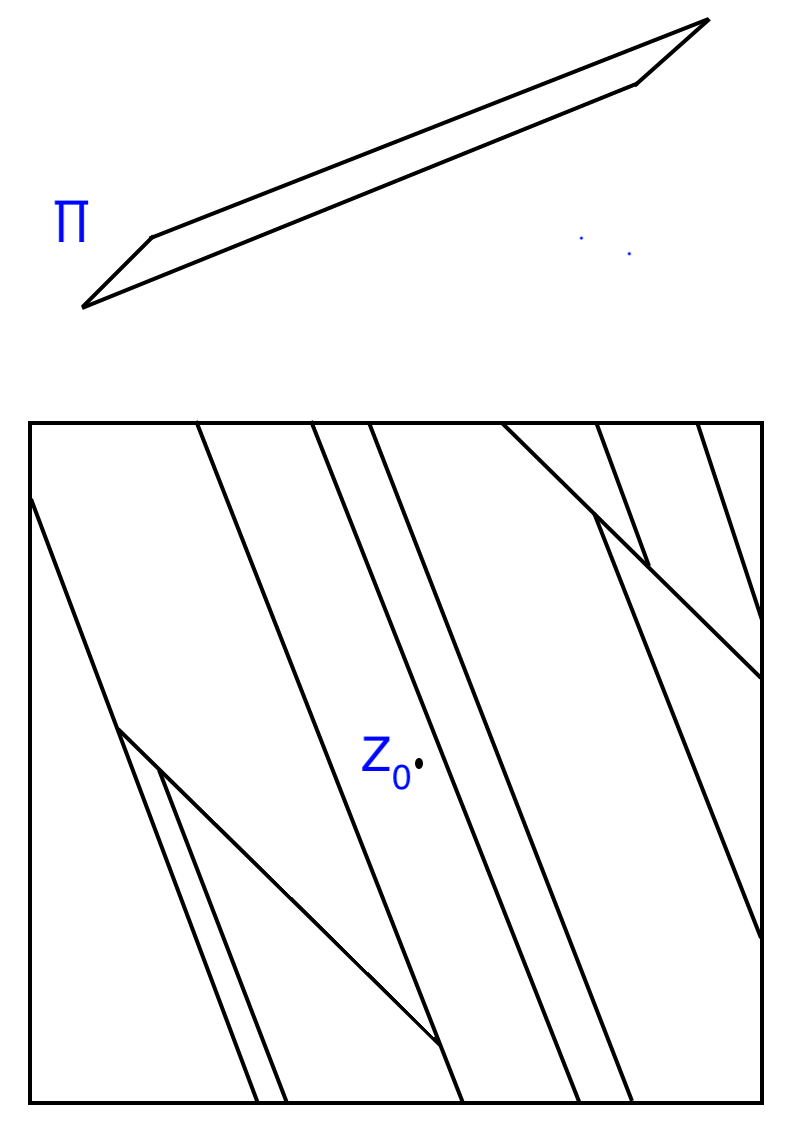}
    \caption{Oblique Mondrian}
    \end{subfigure}
    \caption{An illustration of (a) a weighted Mondrian process with its associated zonoid $\Pi$ as in Example \ref{ex:PI_Mondrian} and (b) an oblique Mondrian process and its associated zonoid $\Pi$ as in Example \ref{ex:PI_oblique_Mondrian}.}
\end{figure}

\subsection{Intrinsic Volumes and Mixed Volumes}

Steiner's formula in convex geometry gives an expression of the volume of the parallel body of a convex body $K$ at distance $\ee > 0$. That is,
\begin{align}\label{e:steiner}
\mathrm{vol}_d(K + \ee B^d) = \sum_{j=0}^d \ee^{d-j} \kappa_{d-j} V_j(K).
\end{align}
The constants $V_j(K)$ are called the \emph{intrinsic volumes} of $K$. The values of these constants only depend on $K$, not the ambient space that $K$ is embedded in. In particular, if $K$ is $\ell$-dimensional, $V_{\ell}(K) = \mathrm{vol}_{\ell}(K)$, the usual $\ell$-dimensional Lebesgue measure of $K$. $V_0(K)$ is the number of connected components of the convex body $K$, and thus $V_0(K) = 1$. The first intrinsic volume is proportional to the mean width and satisfies
\begin{align}\label{e:V1}
    V_1(K) := \frac{d\kappa_d}{\kappa_{d-1}}\int_{\mathbb{S}^{d-1}} h(K,u) \dint \sigma(u),
\end{align}
where $\sigma$ is the uniform probability measure on the unit sphere $\mathbb{S}^{d-1}$.
When $K$ is the ball of unit radius $B^d$ in $\RR^d$, for all $j=1, \ldots, d$,
\begin{align}\label{e:vk_ball}
V_j(B^d) = \binom{d}{j} \frac{\kappa_d}{\kappa_{d-j}}, 
\end{align}
and when $K$ is the unit cube $[0,1]^d$, for all $j = 1, \ldots, d$,
\begin{align}\label{e:vk_cube}
 V_j([0,1]^d) = \binom{d}{j}.   
\end{align}
For convex bodies $K_1, \ldots, K_d$ in $\RR^d$, we notate the 
\emph{mixed volume} by $V(K_1, \ldots, K_d)$. This functional is multilinear in its arguments, symmetric, positive, and monotonic in each variable with respect to inclusion. For additional background on intrinsic volumes and mixed volumes see \cite[Chapter 14]{weil}.

%%%%%%%%%%%%%%%%%%%%%%%%%%%%%%%%%%%%%%%%%%%%%
\section{Regression Setting and Random Tessellation Forests}\label{sec:setting}
%%%%%%%%%%%%%%%%%%%%%%%%%%%%%%%%%%%%%%%%%%%%%

Consider the following standard nonparametric regression setting. %Fix a compact and convex $d$-dimensional domain $W \subset \RR^d$ and S
Suppose the data set $\mathcal{D}_n := \{(X_1, Y_1), \ldots, (X_n,Y_n)\}$ consists of $n$ i.i.d. samples from a random pair $(X,Y) \in \RR^d \times \RR$ such that $\EE[Y^2] < \infty$.  Let $\mu$ denote the unknown distribution of $X$. Suppose $\mu$ has compact support and
\begin{align}\label{e:model1}
Y = f(X) + \ee,
\end{align}
for some unknown function $f: \RR^d \to \RR$ and noise $\ee$ satisfying $\EE[\ee|X] = 0$ and $\mathrm{Var}(\ee|X) = \sigma^2 < \infty$ almost surely. 
We make the additional assumption that the function $f$ is of the form
\begin{align}\label{e:multi-index_model}
f(x) = g(Bx), \quad x \in \RR^d,
\end{align}
where $g:\RR^s \to \RR$ and $B \in \RR^{s \times d}$ for $s \leq d$. This is a general dimensionality reduction model known as a \emph{multi-index model} or \emph{ridge function}, where the regression function depends only on the inputs $\langle b_1, X \rangle, \ldots, \langle b_s, X \rangle$, where $\{b_i\}_{i=1}^s$ are the rows of $B$. Let $S := \mathrm{span}(\{b_i\}_{i=1}^s)$ denote the associated \emph{relevant feature subspace}. An equivalent assumption is that
\begin{align}\label{e:ridge_fxn_assump}
f(x) = \tilde{g}(P_S x),    
\end{align}
for some $\tilde{g}: S \to \RR$ where $P_S$ is the orthogonal projection operator onto the subspace $S$. In the following, we will assume $\tilde{g}$ satisfies the following regularity condition.

\begin{definition}
    A function $f: \RR^d \to \RR$ is in $\mathcal{C}^{k,\beta}(L)$ for $L > 0$ if for all $x,y \in \RR^d$ and $\alpha \leq k$,
    %\[|f(x) - f(y)| \leq L h_{B_*}(x-y).\]
    \[\|D^{\alpha}f(x) - D^{\alpha}f(y)\| \leq L\|x - y\|^{\beta}.\]
\end{definition}

To estimate $f$, we use a random forest estimator built from a random tessellation $\mathcal{P}$ of the support of $\mu$ and the data set $\mathcal{D}_n$. A regression tree estimator based on $\mathcal{P}$ is first defined as
\begin{align}\label{e:tree}
    \hat{f}_n(x, \mathcal{P}) := \sum_{i=1}^n \frac{1_{\{X_i \in Z_x\}}}{\mathcal{N}_n(x)}Y_i,
\end{align}
where $Z_x$ is the cell of $\mathcal{P}$ that contains $x$ and $\mathcal{N}_n(x) := \sum_{i=1}^n 1_{\{X_i \in Z_x\}}$ is the number of points in $Z_x$. If $\mathcal{N}_n(x) = 0$, then it is assumed that $\hat{f}_n(x, \mathcal{P}) = 0$. The random forest estimator based on $\mathcal{P}$ is defined by averaging $M$ i.i.d. copies of the tree estimator, i.e.
\begin{align}\label{e:forest}
    \hat{f}_{n,M}(x) := \frac{1}{M} \sum_{m=1}^M \hat{f}_n(x, \mathcal{P}_m),
\end{align}
where $\mathcal{P}_1, \ldots, \mathcal{P}_M$ are $M$ i.i.d. copies of $\mathcal{P}$. 

A \emph{random tessellation forest} estimator is defined as a random forest estimator, where the random tessellation $\mathcal{P}$ is the tessellation generated by a STIT process. This class of estimators is parameterized by a lifetime $\lambda > 0$ and a directional distribution $\phi$ on the unit sphere, or equivalently, a normalized associated zonoid $\Pi$.

\subsection{Risk Bound for Ridge Functions}

 Our first main result provides an upper bound on the quadratic risk for a general random tessellation forest estimator of a ridge function. In the following, we will denote the diameter of a convex body $K$ in $\RR^d$ by $\mathrm{D}(K)$, and for a linear subspace $S$ in $\RR^d$ we will denote by $P_SK$ the orthogonal projection of $K$ onto $S$ and $P_{S^{\perp}}K$ the orthogonal projection of $K$ onto the orthogonal subspace $S^{\perp}$ to $S$. Throughout, the expectation in the risk is taken with respect to the dataset $\mathcal{D}_n$, $X$, and the random tessellations. The notation $a \vee b$ denotes the maximum of $a$ and $b$ and $a \wedge b$ denotes the minimum of $a$ and $b$.

\begin{theorem}\label{t:Lipschitz_Rate}
Assume $\mathrm{supp}(\mu) \subseteq B^d$ and $f$ satisfies \eqref{e:ridge_fxn_assump} with $\tilde{g} \in \mathcal{C}^{0,\beta}(L)$ for some $L > 0$ and subspace $S$ of dimension $s \leq d$.  Let $\hat{f}_{n} = \hat{f}_{n, M, \lambda, \Pi}$ be a random tessellation forest estimator with normalized associated zonoid $\Pi$, $M$ trees, and lifetime $\lambda > 0$. Then,
\begin{align*}
&\EE[(\hat{f}_{n}(X) - f(X))^2] \\
&\leq \frac{L^2\EE[\mathrm{D}(P_SZ_{0})^{2\beta}]}{\lambda^{2\beta}}   + \frac{5\|f\|_{\infty}^2 + 2\sigma^2}{n} \left(2s\sum_{k=1}^d \lambda^k\kappa_k  V_1(P_{S^{\perp}}\Pi)^{1 \vee (k-s)} + \sum_{k=0}^{s}\lambda^k \kappa_k V_{k}(P_S\Pi)\right). 
\end{align*}
\end{theorem}

The upper bound for the random tessellation forest above is obtained by first bounding the forest risk by the risk of a single tree estimator and then considering a standard bias-variance decomposition. The first expression in the upper bound controls the bias, or approximation error, of the tree estimator, quantifying how well a function $f$ in $\mathcal{C}^{0, \beta}(L)$ can be approximated by any function that is constant over the cells of the corresponding tessellation of the input space. For all inputs that lie in the same cell, the estimator will output the same value, and thus, given the assumption on $f$, this error is controlled by $L$ and the diameter of the projection of the cell onto the relevant feature subspace. The second expression is a bound on the variance, or the estimation error of the model. This is controlled by the amount of data and the complexity of the model, which for randomized decision trees can be quantified by the number of cells of the tessellation, or equivalently, the number of leaves of the corresponding tree. 

The dependence of the second term in the upper bound on $S$ may seem odd since the variance should not depend on the regression model. Indeed, such a bound on the variance holds for an arbitrary subspace, but we use the $S$ defined by the multi-index model to highlight how the variance can decay as the directional distribution becomes more and more concentrated on the relevant feature subspace. The first term in the parentheses comes from using a STIT process that makes splits in directions not aligned with the relevant feature subspace $S$. Note that if the associated zonoid $\Pi$ is contained in $S$, i.e., all split directions are contained in $S$, then the variance term will have order $\lambda^s/n$, which is the order of the variance for a random tessellation tree estimator with lifetime $\lambda$ of a function on $\RR^s$. Also note that if $s = d$, i.e. $S = \RR^d$, then we recover the risk upper bound for general Lipschitz functions on $\RR^d$ in \cite{OReillyTran2021minimax}.

Theorem \ref{t:Lipschitz_Rate} in particular illuminates how the risk for the random tessellation estimator of a ridge function depends on the relationship between the geometry of normalized associated zonoid of the STIT tessellation and the zero cell to the relevant feature subspace $S$. Figure \ref{fig:subspace} illustrates this relationship and how ensuring the projection of $\Pi$ onto $S^{\perp}$ is small means the relevant subspace is more efficiently subdivided for a given lifetime $\lambda$ and the projection of $Z_0$ onto $S$ can be controlled, ensuring a smaller risk.

We now observe that as in Theorem 6 of \cite{OReillyTran2021minimax}, the upper bound in Theorem \ref{t:Lipschitz_Rate} does not depend on the number of trees $M$ and thus holds for a single random tessellation tree estimator. We can assume a stronger regularity condition on the regression function, as well as stronger assumptions on the input distribution $\mu$, and obtain an upper bound that depends on the number of trees $M$ in the forest estimator. However, there are terms in this upper bound that are difficult to interpret in the general case, so we omit it here and refer to Proposition \ref{t:C2_Rate} in Section \ref{a:C2-general-bnd} in the Appendix for the result. Next we will restrict to oblique Mondrian forests, where the analysis simplifies and we can obtain closed form expressions for the risk upper bounds which constitute the main contributions of this work.

\begin{figure}
    \centering
    \includegraphics[width=.5\textwidth]{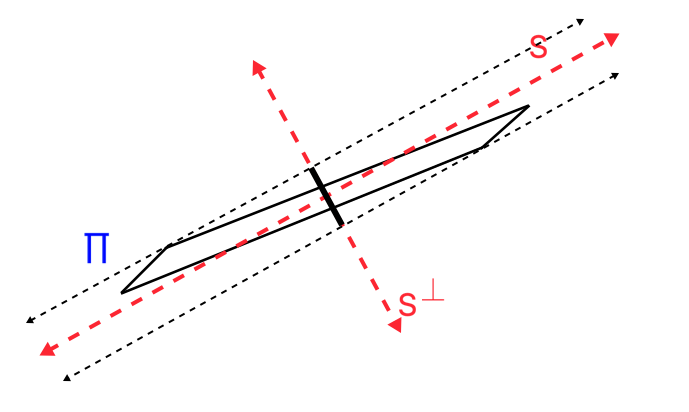}
    \includegraphics[width=.4\textwidth]{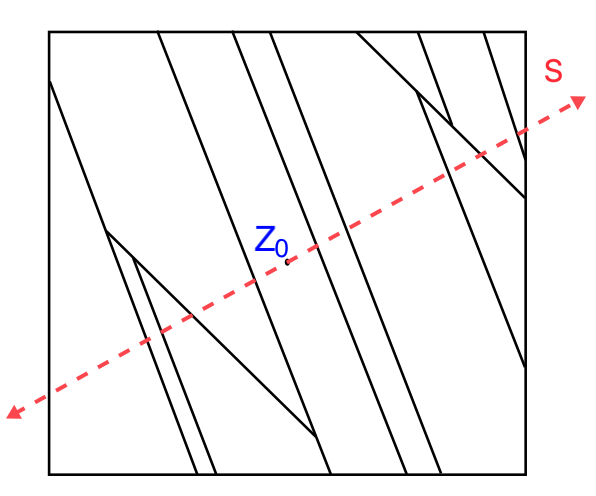}
    \caption{Illustration of an associated zonoid and corresponding STIT tessellation in relation to a relevant feature subspace $S$. If the projection of $\Pi$ onto $S^{\perp}$ is small, then $S$ is cut more frequently by the boundaries of the STIT tessellation for a given lifetime.} 
    \label{fig:subspace}
\end{figure}

%%%%%%%%%%%%%%%%%%%%%%%%%%%%%%%%%%%%%%%%%%%%%
\section{Convergence Rates for Oblique Mondrian Trees and Forests}\label{sec:main_results}
%%%%%%%%%%%%%%%%%%%%%%%%%%%%%%%%%%%%%%%%%%%%%

The risk upper bound in the previous section holds for random tessellation trees and forests with any associated directional distribution. We next would like to obtain rates of convergence for a sequence of random tessellation forest estimators built from $n$ data points as $n$ grows. The results in \cite{OReillyTran2021minimax} provide such rates when the lifetime grows with $n$ and the directional distribution is fixed for all $n$. Here, we consider the case when the directional distribution is also allowed to depend on $n$, representing an estimator that uses a data-driven choice of directional distribution to generate the STIT process. 
It is difficult in general to obtain closed form expressions for the terms in the bounds from Theorem \ref{t:Lipschitz_Rate} and Proposition \ref{t:C2_Rate} that depend on the directional distribution through the diameter of the normalized zero cell projected onto the relevant feature subspace $S$. 
Without further understanding how these terms explicitly depend on the directional distribution or the normalized associated zonoid, we cannot in general obtain the asymptotic behavior of the bias for a sequence of estimators where this parameter depends on $n$.

To overcome this, we now restrict ourselves to the subclass of STIT processes with discrete directional distributions, where the directions of the splits are sampled from a finite discrete set of vectors on the unit sphere. That is, there is a finite set of linear combinations of covariates along which the STIT process makes splits. Under this assumption, we can obtain bounds on the relevant statistics that will illuminate the asymptotic behavior of the risk upper bounds. Another reason for focusing on this subclass of STIT processes is that the partition of the data they generate can be efficiently obtained by first applying a linear transformation to the input data, and then running a Mondrian process. As mentioned in the introduction, we will thus call this subclass oblique Mondrian processes and refer to the corresponding tree and forest estimators as \emph{oblique Mondrian trees and forests}. 
 
In particular, for a matrix $A \in \RR^{d \times k}$ define the directional distribution %there exist weights $\omega_i^{(n)} \in (0,1)$ such that
\begin{align}\label{e:model_dirdist}
\phi_A = \sum_{i = 1}^k \frac{\|a_i\|_2}{2\|A\|_{2,1}}\left(\delta_{a_i/\|a_i\|_2} + \delta_{-a_i/\|a_i\|_2}\right),
\end{align}
where $\{a_i\}_{i=1}^k$ are the columns of $A$, and $\|A\|_{2,1} = \sum_{i=1}^k \|a_i\|_2$ is the norm of the matrix that sums the $\ell_2$-norms of the column vectors. We assume the columns contain $d$ linearly independent vectors in $\RR^d$, i.e. the rank of $A$ is $d \leq k$.
The partition of the data induced by a STIT tessellation with directional distribution $\phi_A$ can be efficiently obtained by applying the transformation $A^T$ to the data and then running a Mondrian process. This is proved in Section \ref{sec:Mondrian_Transform}, and is a refinement of Theorem 3.1 in \cite{OReillyTran2021}. 
In the remainder of this section, we will focus on directional distributions of the form \eqref{e:model_dirdist} for nonsingular $A \in \RR^{d \times d}$. The theory can be extended to general fixed $k \geq d$ and $A \in \RR^{d \times k}$ with rank $d$, but a larger $k$ only increases the upper bound on the bias using our proof techniques.

Our first result of this section is an upper bound on the risk of an oblique Mondrian forest for a regression function satisfying the same assumption as in Theorem \ref{t:Lipschitz_Rate}. 

\begin{theorem}\label{t:C1_Transformed_Mondrain_Bound}
Assume $\mathrm{supp}(\mu) \subseteq B^d$ and $f$ satisfies \eqref{e:ridge_fxn_assump} with $\tilde{g} \in \mathcal{C}^{0,\beta}(L)$ for some $L > 0$.  Let $\hat{f}_{n} = \hat{f}_{n,\lambda,M}$ be an oblique Mondrian forest estimator with lifetime $\lambda$ and directional distribution $\phi_A$ as in \eqref{e:model_dirdist} for some nonsingular $A \in \RR^{d \times d}$ with  $\|A\|_{2,1} = 1$.
Then,
\begin{align*}
&\EE[(\hat{f}_{\lambda, n, M}(X) - f(X))^2] \\
&\leq \frac{9L^{2}d^{2\beta}}{\lambda^{2\beta}\sigma_{s}(P_SA)^{2\beta}} + \frac{(5\|f\|^2_{\infty} + 2\sigma^2)}{n}\left(2s\sum_{k=1}^d \lambda^k\kappa_k \|P_{S^{\perp}}A\|_{2,1}^{1 \vee (k-s)} + \sum_{k=0}^{s}\lambda^k \frac{\kappa_k}{k!}\right)%\left( \sum_{k=s+1}^d c_{d,k}\lambda^k \|P_{S^{\perp}}A\|_{2,1}^{k-s}  +  \sum_{k=0}^{s} c_{d,k}\lambda^k\right),
\end{align*}
where %$c_{d,k} := \frac{\kappa_k \pi^{k/2}d^{k/2}}{k!}$ and 
$\sigma_s(P_SA)$ is the $s$-th largest singular value of the matrix $P_SA$. 
\end{theorem}

If the relevant feature subspace $S$ is known, one can project the input data onto $S$ and then generate a random tessellation forest estimator supported on this $s$-dimensional subspace. Note that the risk bound in Theorem \ref{t:C1_Transformed_Mondrain_Bound} reduces to the upper bound for an oblique Mondrian forest on $S$ when the range of $A$ is contained in $S$, and thus minimax optimal rates for functions on $\RR^s$ will be obtained with such an estimator by appropriately tuning $\lambda$ with $n$ as in \cite{OReillyTran2021minimax}. In practice, we do not know this subspace, and so instead must estimate a linear image $A$ that approximates a projection onto $S$ and build an oblique Mondrian estimator. There are many existing approaches for estimating relevant feature directions, including sufficient dimension reduction methods \cite{SIR_Li1991,dennis2000save,Xia2002MAVE} and gradient-based approaches \cite{Trivedietal2014,JMLR:v11:wu10a}. We do not study a particular method for estimating $A$ here, but rather focus on inference post-estimation of the relevant feature directions. An algorithm that generates an oblique Mondrian forest with an estimate of $A$ based on the expected gradient outer product was recently introduced \cite{Baptista_TrIM} and uses the results presented here to obtain convergence guarantees.

From the definition of $\phi_A$ in \eqref{e:model_dirdist}, the columns of $A$ determine the directions and weights of the splits used to generate each tree. When the projection of these column vectors onto $S^{\perp}$ has a small norm, then each vector is either close to the span of $S$ or has a small norm, giving the associated direction a small weight so that the oblique Mondrian process rarely makes a split in that direction. The bound in Theorem \ref{t:C1_Transformed_Mondrain_Bound} above quantifies how the risk of the corresponding oblique Mondrian estimator depends on the choice of this $A$, including the dependence on the projection of the columns of $A$ onto $S^{\perp}$ though the sum of the column norms $\|P_{S^{\perp}}A\|_{2,1}$.

We next model the results of a data-driven procedure for selecting a set of split directions with a sequence of matrices $A_n$ that will be applied to inputs of the dataset $\mathcal{D}_n$ of size $n$. The following result provides a rate of convergence of the corresponding sequence of oblique Mondrian forests depending on how well $A_n$ approximates a projection onto $S$ as $n$ grows. As long as this approximation error approaches zero in the limit, we obtain an improved rate of convergence for ridge functions over the worst-case minimax rate for $\mathcal{C}^{0,\beta}$ functions on $\RR^d$. In addition, these rates provide a sufficient condition for this approximation error such that these oblique Mondrian forests achieve the minimax optimal rate of convergence for $\mathcal{C}^{0,\beta}$ functions on $\RR^s$, where $s$ is the dimension of the relevant feature subspace.

\begin{corollary}\label{cor:C1_Transformed_Mondrain_Rate}
Consider the setting of Theorem \ref{t:C1_Transformed_Mondrain_Bound}. For each $n$, let $\hat{f}_n$ be an oblique Mondrian forest with lifetime $\lambda_n$ and directional distribution $\phi_{A_n}$ for a nonsingular $A_n \in \mathbb{R}^{d \times d}$ with $\|A_n\|_{2,1} = 1$. Assume there is an absolute constant $c > 0$ such that
\begin{itemize}
\item[(i)] $\sigma_{s}(P_SA_n) \geq c$, and
\item[(ii)] $\|P_{S^{\perp}}A_n\|_{2,1} \leq \ee_n$ for $\ee_n = o(1)$. 
\end{itemize}
Then, letting 
$$\lambda_n \asymp \begin{cases} (L^2n)^{\frac{1}{d + 2\beta}} \ee_n^{-\frac{(d-s)}{d + 2\beta}}, & \varepsilon_n \gtrsim (L^2n)^{-\frac{1}{s + 2\beta}}, \\ 
(L^2n)^{\frac{1}{s + 2\beta}}, & \varepsilon_n \lesssim (L^2n)^{-\frac{1}{s + 2\beta}}
\end{cases}$$ 
yields
\begin{align}\label{e:rate1}
\EE\left[(f(X) - \hat{f}_{n,\lambda_n, M_n}(X))^2\right] \lesssim \max\left\{L^{\frac{2d}{d + 2\beta}}n^{-\frac{2\beta}{d + 2\beta}} \ee_n^{\frac{2\beta(d-s)}{d + 2\beta}}, L^{\frac{2s}{s + 2\beta}}n^{-\frac{2\beta}{s + 2\beta}}\right\}.
\end{align}
\end{corollary}

\begin{rem}
We note that Corollary \ref{cor:C1_Transformed_Mondrain_Rate} implies that a sufficient condition for minimax optimal rates for the class of $\mathcal{C}^{0,\beta}(L)$ functions on $\RR^s$ to be achieved, i.e.
\begin{align}\label{e:rate1}
\EE\left[(f(X) - \hat{f}_{\lambda_n,n}(X))^2\right] 
&\lesssim L^{\frac{2s}{s + 2\beta}}n^{-\frac{2\beta}{s + 2\beta}},
\end{align}
is that the asymptotic error in the relevant feature matrix satisfies $\ee_n \lesssim L^{-\frac{2}{s + 2\beta}}n^{-\frac{1}{s + 2\beta}}$, and setting $\lambda_n = L^{\frac{2}{s + 2\beta}}n^{\frac{1}{s + 2\beta}}$.
\end{rem}

The above results hold for oblique Mondrian forests with any number of trees. The advantage of averaging the prediction of many trees is observed in the following results, which provide a risk bound that depends on the number of trees for an oblique Mondrian forest estimator when additional smoothness is assumed for the regression function as in Proposition \ref{t:C2_Rate}, as well as an improved rate of convergence. For a sequence of oblique Mondrian forests with directional distribution depending on $n$, it is much more difficult to obtain improved rates in this setting with transparent conditions on the linear transformation $A_n$ in general. To provide such conditions, we make the additional assumption that the columns of $A$ form an orthogonal set of vectors in $\RR^d$. If this is not initially satisfied, one could perform an eigendecomposition or other orthogonalization procedure of the given matrix to obtain a directional distribution that satisfies this assumption. %that the normal vectors to the hyperplane splits, i.e. the linear combinations of covariates used as features, either already lie in the relevant feature subspace $S$ or lie in the orthogonal subspace $S^{\perp}$.

\begin{theorem}\label{t:C2_Transformed_Mondrain_Bound}
Assume $f$ satisfies \eqref{e:ridge_fxn_assump} with $\tilde{g} \in \mathcal{C}^{1, \beta}(L)$, and assume $\mu$ has a positive and Lipschitz density on its support $B^d$.   Let $\hat{f}_{n} = \hat{f}_{n,\lambda, M}$ be the oblique Mondrian forest estimator with lifetime $\lambda$, $M$ trees, and directional distribution $\phi_A$ given by \eqref{e:model_dirdist} for a nonsingular $A \in \RR^{d \times d}$ with $\|A\|_{2,1} = 1$ and orthogonal column vectors. %such that $P_Sa_i \in \{a_i, \mathbf{0}\}$ for each $i = 1, \ldots, d$. 
%Let $r(K)$ denote the radius of the largest ball contained in $K$ and define $K_{\delta}$ as in Theorem \ref{t:C2_Rate}. 
Then, for fixed $\delta \in (0, 1)$, letting $B_{\delta} := \{x \in \RR^d : \|x\| \leq 1 - \delta\}$,
\begin{align*}
&\EE[(\hat{f}_{n}(X) - f(X))^2| X \in B_{\delta}] \\
&\leq c_{\mu}L^2\left(\frac{d^2\|P_{S^{\perp}}A\|^2_F}{\lambda^2\sigma_s(P_SA)^2} + \frac{d^{2\beta + 2}}{\lambda^{2 + 2\beta} \sigma_s(P_SA)^{2 + 2\beta}}\right) %\frac{c_{\mu} L^2\Gamma(2d + 1 + \beta)^2}{\lambda^{2 + 2\beta}\sigma_s(P_SA)^{2 + 2\beta}\Gamma(2d)^2}  
+ \frac{6L^2d^2}{\lambda^2M\sigma_s(P_SA)^{2}} \\
& +  \frac{5\|f\|_{\infty}^2 + 2\sigma^2}{n\mu(B_{\delta})} \left(2s\sum_{k=1}^d \lambda^k\kappa_k \|P_{S^{\perp}}A\|_{2,1}^{1 \vee (k-s)} + \sum_{k=0}^{s}\frac{\lambda^k\kappa_k}{k!}\right) + o\left(\frac{1}{\lambda^{2+ 2 \beta} \sigma_s(P_SA)^{2+ 2 \beta}}\right), 
\end{align*}
%where $\rho_{A,S} := \max_{i=1, \ldots, d} \min\{\|(A_n^{+}P_S)_{i,:}\|_2, \|(A_n^{+}P_{S^{\perp}})_{i,:}\|_2\}$.
and in the unconditional case $\delta = 0$,
\begin{align*}
 &\EE[(\hat{f}_{\lambda,n, M}(X) - f(X))^2]  \\&\leq c_{\mu}L^2\left(\frac{d^2\|P_{S^{\perp}}A\|^2_F}{\lambda^2\sigma_s(P_SA)^2} + \frac{d^4}{\lambda^{3 \wedge (2 + 2\beta)}\sigma_s(P_SA)^{3 \wedge (2 + 2\beta)}}\right) %\frac{c_{\mu} L^2\Gamma(2d + 1 + \beta)^2}{\lambda^{2 + 2\beta}\sigma_s(P_SA)^{2 + 2\beta}\Gamma(2d)^2}  + \frac{\tilde{c}_{\mu} L^2 d^3V_{s-1}(K_S)}{\lambda^3\sigma_s(P_SA)^3\mathrm{vol}_d(K)} 
 + \frac{6L^2d^2}{\lambda^2M\sigma_s(P_SA)^{2}}\\
 & +  \frac{5\|f\|_{\infty}^2 + 2\sigma^2}{n} \left(2s\sum_{k=1}^d \lambda^k\kappa_k \|P_{S^{\perp}}A\|_{2,1}^{1 \vee (k-s)} + \sum_{k=0}^{s}\frac{\lambda^k \kappa_k}{k!}\right)  + o\left(\frac{1}{\lambda^{3 \wedge (2+ 2 \beta)} \sigma_s(P_SA)^{3 \wedge (2+ 2 \beta)}}\right). 
\end{align*}
The constants in the little-$o$ term depend on $\delta, d, L$, and $\beta$ and the constant $c_{\mu}$ only depends on $\mu$.
\end{theorem}

Using these upper bounds, we are now able to obtain convergence rates for a sequence of oblique Mondrian forests corresponding to a sequence of linear transformations $A_n$ that depend on the approximation error between $A_n$ and a projection onto the relevant feature subspace $S$ similarly to Corollary \ref{cor:C1_Transformed_Mondrain_Rate}.

\begin{corollary}\label{cor:C2_Transformed_Mondrain_Rate}
Consider the setting of Theorem \ref{t:C2_Transformed_Mondrain_Bound}. For each $n$, let $\hat{f}_n$ be an oblique Mondrian forest estimator with lifetime $\lambda_n$, number of trees $M_n$, and directional distribution $\phi_{A_n}$ for some nonsingular $A_n \in \mathbb{R}^{d \times d}$ with $\|A_n\|_{2,1} = 1$ and orthogonal column vectors. Assume there is an absolute constant $c > 0$ such that 
\begin{itemize}
\item[(i)] $\sigma_{s}(P_SA_n) \geq c$ for all $n$, and
\item[(ii)] $\|P_{S^{\perp}}A_n\|_{2,1} \leq \ee_n$ for $\ee_n = o(1)$. %, and
%\item[(iii)] $\rho_{A_n, S} \leq \rho_n$ for $\rho_n = o(1)$.
\end{itemize}
For fixed $\delta \in (0, 1)$, letting 
\[
\lambda_n
\asymp
\begin{cases}
(nL^2)^{\frac{1}{s+2+2\beta}},
&
\varepsilon_n
\lesssim
(nL^2)^{-\frac{1}{s+2+2\beta}},
\\
(nL^2\varepsilon_n^{s-d})^{\frac{1}{d+2+2\beta}},
&
(nL^2)^{-\frac{1}{s+2+2\beta}}
\lesssim
\varepsilon_n
\lesssim
(nL^2)^{-\frac{\beta}{d(1-\beta)+\beta s+2+2\beta}},
\\
(nL^2\epsilon_n^{s-d+2})^{\frac{1}{d+2}},
&
\varepsilon_n
\gtrsim
(nL^2)^{-\frac{\beta}{d(1-\beta)+\beta s+2+2\beta}}
\end{cases}
\]
%$\lambda_n = L^{2/(d + 2+ 2 \beta)}n^{1/(d + 2 + 2 \beta)}\ee_n^{-(d-s)/(d + 2 + 2\beta)}$ 
and $M_n \gtrsim \lambda_n^{2\beta}$ yields
\begin{align}\label{e:rate1}
&\EE[(\hat{f}_{n}(X) - f(X))^2|X \in B_{\delta}] \nonumber \\
&\lesssim 
\begin{cases}
L^{\frac{2s}{s+2+2\beta}} n^{-\frac{2+2\beta}{s+2+2\beta}},
&
\varepsilon_n
\lesssim
(nL^2)^{-\frac1{s+2+2\beta}},
\\
L^{\frac{2d}{d + 2 + 2\beta}}n^{-\frac{2+2\beta}{d+2+2\beta}}
\varepsilon_n^{\frac{(d-s)(2+2\beta)}{d+2+2\beta}},
&
(nL^2)^{-\frac1{s+2+2\beta}}
\lesssim
\varepsilon_n
\lesssim
(nL^2)^{-\frac{\beta}{d(1-\beta)+\beta s+2+2\beta}},
\\
L^{\frac{2d}{d+2}} n^{-\frac{2}{d+2}}
\varepsilon_n^{\frac{2(2d-s)}{d+2}},
&
\varepsilon_n
\gtrsim
(nL^2)^{-\frac{\beta}{d(1-\beta)+\beta s+2+2\beta}}.
\end{cases}
%\max\left\{ L^{\frac{2d}{d + 2 \beta + 2}}n^{-\frac{2 + 2\beta}{d + 2\beta + 2}}\ee_n^{\frac{(d-s)(2 + 2\beta)}{d+ 2 + 2\beta}}, L^{\frac{2d}{s + 2+2 \beta}}n^{-\frac{2 + 2\beta}{s + 2 + 2\beta}} \right\}.
\end{align}
In the unconditional case $\delta = 0$, the rates above hold if $2 + 2\beta \leq 3$, and otherwise letting $M_n \gtrsim \lambda_n$ and
\[
\lambda_n
\asymp
\begin{cases}
(nL^2)^{\frac{1}{s+3}},
&
\varepsilon_n
\lesssim
(nL^2)^{-\frac{1}{s+3}},
\\
\left(nL^2\varepsilon_n^{s-d}\right)^{\frac{1}{d+3}},
&
(nL^2)^{-\frac{1}{s+3}}
\lesssim
\varepsilon_n
\lesssim
(nL^2)^{-\frac{1}{d+s+6}},
\\
\left(nL^2\varepsilon_n^{s-d+2}\right)^{\frac{1}{d+2}},
&
\varepsilon_n
\gtrsim
(nL^2)^{-\frac{1}{d+s+6}},
\end{cases}
\]
gives
\[
\EE[(\hat{f}_{\lambda,n, M}(X) - f(X))^2] \lesssim
\begin{cases}
L^{\frac{2s}{s+3}} n^{-\frac{3}{s+3}},
&
\varepsilon_n
\lesssim
(nL^2)^{-1/(s+3)},
\\
L^{\frac{2d}{d+3}} n^{-\frac{3}{d+3}}
\varepsilon_n^{\frac{3(d-s)}{d+3}},
&
(nL^2)^{-\frac{1}{s+3}}
\lesssim
\varepsilon_n
\lesssim
(nL^2)^{-\frac{1}{d+s+6}},
\\
L^{\frac{2d}{d+2}} n^{-\frac{2}{d+2}}
\varepsilon_n^{\frac{2(2d-s)}{d+2}},
&
\varepsilon_n
\gtrsim
(nL^2)^{-\frac{1}{d+s+6}}.
\end{cases}
\]
%$\lambda_n \sim L^{\frac{2}{d+3}}n^{\frac{1}{d+3}}\ee_n^{- \frac{d-s}{d + 3}}$ gives
%\begin{align*}
%&\EE[(\hat{f}_{\lambda,n, M}(X) - f(X))^2] \lesssim \max\left\{ L^{\frac{2d}{d+3}}n^{-\frac{3}{d+3}}\ee_n^{ \frac{3(d-s)}{d + 3}} , L^{\frac{2s}{s+3}}n^{-\frac{3}{s+3}}\right\}.
%\end{align*}}
\end{corollary}

\begin{rem}
    We note that Corollary \ref{cor:C2_Transformed_Mondrain_Rate} implies that a sufficient condition for the minimax rate for the class of $\mathcal{C}^{1,\beta}(L)$ functions on $\RR^s$ to be achieved, i.e.
    \begin{align}\label{e:rate1}
\EE[(f(X) - \hat{f}_{n,\lambda_n, M_n}(X))^2|X \in B_{\delta}] \lesssim L^{\frac{2d}{s + 2 \beta + 2}}n^{-\frac{2 + 2\beta}{s + 2\beta + 2}},
\end{align}
is an asymptotic error decay $\ee_n \lesssim L^{-2/(s + 2 + 2 \beta)}n^{-1/(s + 2 + 2\beta)}$ of the relevant feature matrix and letting $\lambda_n = L^{2/(s + 2 + 2 \beta)}n^{1/(s + 2 + 2\beta)}$ and $M_n \gtrsim \lambda_n^{2\beta}$. The minimax rate is not achieved for any $\varepsilon_n$ in the case $\delta = 0$ and $2 + 2\beta > 3$.
\end{rem}

\section{Risk Bounds for Weighted Mondrian Forests}\label{sec:Mondrian_forests}

Consider now the special case of weighted Mondrian forests obtained from weighted Mondrian processes as in example \ref{ex:PI_Mondrian}. We will study the ability of this subclass of oblique Mondrian forests to adapt to sparse functions, as has been studied for other variants of axis-aligned random forests. 

More specifically, consider the following setting. Assume that $\mathcal{S} \subseteq \{1, \ldots, d\}$ is a subset of size $|\mathcal{S}| = s$ that corresponds to a small subset of the covariates that the regression function varies with respect to. That is, we assume the true function $f$ is of the form
\begin{align}\label{e:sparse_f}
f(x) = g(x_S) = g(\{x_i\}_{i \in \mathcal{S}}) = g(P_Sx),
\end{align}
for $g: \RR^s \to \RR$ and the orthogonal projection operator $P_S$ onto $S = \mathrm{span}\{e_i : i \in \mathcal{S}\}$.
Assume the input $X$ is supported on $[0,1]^d$ and $Y = f(X) + \ee$ for noise $\ee$ as in section \ref{sec:setting}. Consider a weighted Mondrian forest estimator $\hat{f}_n$ built from $n$ i.i.d. samples of $(X,Y)$ with lifetime $\lambda_n$ and directional distribution
\begin{align}\label{e:phi_weighted_mondrian}
\phi = \sum_{i=1}^d \frac{\omega_i}{2} \left(\delta_{e_i} + \delta_{-e_i}\right),
\end{align}
where the weights $\{\omega_i\}_{i=1}^d$ satisfy $\sum_{i=1}^d \omega_i = 1$ and $\omega_i > 0$ for each $i$.

The following results are analogous to those presented for oblique Mondrian forests, with upper bounds on the risk followed by corollaries in the setting where the weights depend on $n$, modeling a data-driven choice of weights. A variety of feature importance scores have been developed that could be used to select the weights \cite{breiman2001random, deng2022towards}, and the approach of reweighting the split selection probabilities before generating the trees in random forest algorithms was introduced in \cite{Yuetal2018}. Here, we assume some data-driven method of estimating feature relevance has generated associated weights $\omega^{(n)}_i$ that converge to 0 as $n$ grows if dimension $i$ is not in the set of relevant features $\mathcal{S}$. In this setting, we obtain rates of convergence and conditions on this approximation error needed to obtain minimax optimal rates depending on the sparsity level $s$.
We state the results in this setting separately from the more general oblique Mondrian forests because we can obtain a simplified version of the variance bound which gives a weaker condition on the weights for improved rates than obtained from directly applying the previous results. For simplicity, we restrict to the case where $\beta = 1$ for the assumption on the regression function in the following statements.

\begin{theorem}\label{t:rate1_Mondrian}
Assume $\mathrm{supp}(\mu) \subseteq [0,1]^d$ and $f$ satisfies \eqref{e:sparse_f} where $g \in \mathcal{C}^{0, 1}(L)$ for some $L > 0$, i.e. $g$ is $L$-Lipschitz. Let $\hat{f}_{n} = \hat{f}_{\lambda, n, M}$ be the weighted Mondrian tree estimator with directional distribution \eqref{e:phi_weighted_mondrian} and lifetime $\lambda > 0$, and define $\omega_{S} := \min_{i \in \mathcal{S}} \omega_i$.
Then,
\begin{align*}
\EE[(\hat{f}_{n}(X) - f(X))^2] &\leq \frac{6L^2s}{\lambda^2\omega_{S}^2}+ \frac{(5\|f\|^2_{\infty} + 2\sigma^2)}{n}\prod_{i = 1}^d\left(1 + \lambda \omega_i\right). 
\end{align*}
\end{theorem}

\begin{corollary}\label{cor:C1_Mondrian}
Consider the setting of Theorem \ref{t:rate1_Mondrian}. For each $n$, let $\hat{f}_n$ be a weighted Mondrian forest estimator with lifetime $\lambda_n$ and directional distribution $\phi_n$ as in \eqref{e:phi_weighted_mondrian} where the weights $\{\omega_i^{(n)}\}_{i=1}^d$ depend on $n$. Assume there is an absolute constant $c > 0$ such that
\begin{itemize}
\item[(i)] $\omega^{(n)}_S \geq c$ for all $n$, and
\item[(ii)] $\max_{i \notin \mathcal{S}} \omega^{(n)}_i \leq \ee_n$ for $\ee_n = o(1)$. 
\end{itemize}
Then, the same rates as in Corollary \ref{cor:C1_Transformed_Mondrain_Rate} hold.
\end{corollary}

\begin{theorem}\label{t:C2_Rate_Mondrian}
Assume $\mathrm{supp}(\mu) = [0,1]^d$ and that $\mu$ has a positive and Lipschitz density on its support. Assume $f$ satisfies \eqref{e:sparse_f} for some $g \in \mathcal{C}^{1,\beta}(L)$ and let $\hat{f}_{n}$ be the weighted Mondrian forest estimator with directional distribution \eqref{e:phi_weighted_mondrian} and lifetime $\lambda > 0$.
Then, for $\delta \in (0,1/2)$,
\begin{align*}
&\EE[(\hat{f}_{n}(X) - f(X))^2| X \in [\delta, 1 - \delta]^d] \\
&\leq \frac{c_{\mu} s^4L^2}{\lambda^{4}\omega_S^{4}}  + \frac{6L^2s}{\lambda^2M \omega_S^2}+  \frac{5\|f\|_{\infty}^2 + 2\sigma^2}{n} \prod_{i = 1}^d\left(1 + \lambda \omega_i\right) + o(\lambda^{-4}),    
\end{align*}
where $\omega_{S} := \min_{i \in \mathcal{S}} \omega_i$. For $\delta = 0$,
\begin{align*}
&\EE[(\hat{f}_{n}(X) - f(X))^2]\\ &\leq \frac{c_{\mu} s^4L^2}{\lambda^{4}\omega_S^{4}}  + \frac{\tilde{c}_{\mu} s^4 L^2}{\lambda^3} + \frac{6L^2s}{\lambda^2M \omega_S^2} +  \frac{5\|f\|_{\infty}^2 + 2\sigma^2}{n} \prod_{i = 1}^d\left(1 + \lambda \omega_i\right) + o(\lambda^{-3}), 
\end{align*}
where $c_{\mu}$ and $\tilde{c}_{\mu}$ are constants that depend only on $\mu$.
\end{theorem}

\begin{corollary}\label{cor:C2_Mondrian}
Consider the setting of Theorem \ref{t:C2_Rate_Mondrian}. For each $n$, let $\hat{f}_n$ be a weighted Mondrian forest estimator with lifetime $\lambda_n$, number of trees $M_n$, and directional distribution $\phi_n$ as in \eqref{e:phi_weighted_mondrian} where the weights $\{\omega_i^{(n)}\}_{i=1}^d$ depend on $n$.
Assume there is an absolute constant $c > 0$ such that
\begin{itemize}
\item[(i)] $\omega^{(n)}_S \geq c$ for all $n$, and
\item[(ii)] $\max_{i \notin \mathcal{S}} \omega^{(n)}_i \leq \ee_n$ for $\ee_n = o(1)$. 
\end{itemize}
Then, the same rates as in Corollary \ref{cor:C2_Transformed_Mondrain_Rate} hold.
\end{corollary}

The proofs of the above results appear in Appendix \ref{a:weighted_Mondrian}.

\section{Suboptimality of Mondrian trees for estimating ridge functions}\label{sec:suboptimality}

The results presented in section \ref{sec:main_results} show that improved rates of convergence for ridge functions over the minimax rates for general Lipschitz and $\mathcal{C}^2$ functions in $\RR^d$ can be obtained from oblique Mondrian forests with an adaptive choice of directional distribution that has support consisting of directions that approximate directions spanning the relevant feature subspace $S$. The results also provide sufficient conditions for how well the sequence of linear transformations $A_n$ must approximate a projection onto $S$ to achieve minimax optimal convergence rates depending on the dimension $s$ of $S$. When the underlying function depends on a relevant feature that is a dense linear combination of the original set of covariates, restricting the splits to be axis-aligned (i.e. using a weighted Mondrian process) means that these conditions will not be satisfied, as the transformation matrix will be diagonal and thus will not approximate well the oblique projection. To make this precise, the next result shows that oblique splits are in fact necessary to obtain improved rates of convergence for general ridge functions over the worst-case minimax rates for functions on $\RR^d$ by obtaining a lower bound on the risk of a weighted Mondrian tree estimator when the underlying function is linear. 

\begin{theorem}\label{thm:Mondrian_suboptimal}
    Suppose $Y = \langle a, X \rangle + \varepsilon$, where $a_i \neq 0$ for each $i = 1, \ldots, d$, and assume $X \sim \mathrm{Uniform}([0,1]^d)$. Let $\hat{f}_{n} = \hat{f}_{n,\lambda}$ be a weighted Mondrian tree estimator with lifetime $\lambda$ and directional distribution
    $$\phi = \sum_{i=1}^d \frac{\omega_i}{2}\left(\delta_{e_i} + \delta_{-e_i}\right),$$
    where $\{\omega_i\}_{i=1}^d$ are weights such that $\omega_i > 0$ and $\sum_{i=1}^d \omega_i = 1$.
    Then,
    \begin{align*}
        \EE[(\hat{f}_{n}(X) - f(X))^2] &\geq \sum_{i=1}^d \frac{a^2_i}{2\lambda^2\omega_i^2} \bigg(1- \frac{2}{\lambda\omega_i} - \frac{1}{\lambda^2 \omega^2_i}\bigg) +  \sigma^2\left(\frac{n}{2^d\lambda^d\Pi_{i=1}^d \omega_i} + 1\right)^{-1}.
    \end{align*}
\end{theorem}

The proof of this result is in Appendix \ref{a:suboptimal}. 
Considering the asymptotic behavior of this lower bound when the weights are allowed to depend on $n$, note that if $(\lambda^d\Pi_{i=1}^d\omega^{(n)}_i) /n \to 0$, then the variance is on the order of $(\lambda^d\Pi_{i=1}^d\omega^{(n)}_i) /n$. Then, observe that the assumption $a_i \neq 0$ for all $i = 1, \ldots, d$ implies there is no choice of weight sequences $\omega^{(n)}_{i}$ as $n \to \infty$ that will give an improved rate of convergence over the minimax rate for general Lipschitz functions on $\RR^d$. An improved rate \emph{can} be obtained with a sequence of directional distributions with supports consisting of vectors converging in Euclidean distance to $a/\|a\|_2$ by Corollary \ref{cor:C1_Transformed_Mondrain_Rate}.

\section{Oblique Mondrian Processes}\label{sec:Mondrian_Transform}

In this section, we prove that one can generate a partition of the dataset induced by an oblique Mondrian process with directional distribution \eqref{e:model_dirdist} by applying a linear transformation to the data and then running a standard Mondrian process. We also see that under the assumption this linear transformation is nonsingular, the zero cell of the resulting oblique Mondrian tessellation has the distribution of a transformation of the zero cell of the tessellation generated by a standard Mondrian process.

\begin{prop}\label{p:Mondrian_image}
Let $A \in \RR^{d \times m}$ have rank $d \leq m$. 
Fix $\lambda > 0$. %
Let $\mathcal{Y}_A(\lambda)$ denote the union of cell boundaries of an oblique Mondrian tessellation in $\RR^d$ %$\mathrm{range}(A)$ 
with directional distribution $\phi_A$ as in \eqref{e:model_dirdist} and lifetime $\lambda$. 
Then, $A^T\left(\mathcal{Y}_A(\lambda)\right)$ has the same distribution as the union of cell boundaries of a Mondrian tessellation in $\RR^m$ with lifetime $\frac{m\lambda}{\|A\|_{2,1}}$ intersected with the $d$-dimensional subspace $\text{ran}(A^T)$.
\end{prop}

\begin{rem}\label{rem:piA}
    An oblique Mondrian process corresponding to a $d\times m$ matrix $A$ has associated zonoid $\Pi_A$ with support function given by
\begin{align*}
    h_{\Pi_A}(u) &: = \frac{1}{\|A\|_{2,1}} \sum_{i=1}^m |\langle u, A^Te_i\rangle| = \frac{1}{m}\sum_{i=1}^m \frac{m}{\|A\|_{2,1}} |\langle Au, e_i\rangle| \\
    &= h_{\Pi_M}\left( \frac{m}{\|A\|_{2,1}} Au\right) = h_{\frac{m}{\|A\|_{2,1}}A^T\Pi_M}(u),
\end{align*}
for all $u \in \RR^d$, where $\Pi_M$ is the associated zonoid of a standard Mondrian process in $\RR^m$. Thus, $\Pi_A = \frac{m}{\|A\|_{2,1}} A^T\Pi_M$.
\end{rem}

\begin{rem}
The above result highlights an important consideration when generating oblique random forests by first applying a linear transformation $A$ to the data and then running an axis-aligned random forest. The lifetime of the oblique Mondrian process, which determines the complexity of the partition, is implicitly scaled by the constant $\frac{1}{m} \sum_{i=1}^m \|a_i\|_2 = \frac{1}{m}\|A\|_{2,1}$.
Thus, to ensure the data transformation does not change the complexity of the corresponding tree estimator, we must not only apply $A$ to the input data but also scale the data by the constant $\frac{m}{\|A\|_{2,1}}$. This will cancel out the implicit scaling of the lifetime induced by $A$ and the overall lifetime will be unchanged from the lifetime of the Mondrian process that is run on the transformed data.
\end{rem}

From Proposition \ref{p:Mondrian_image} we also obtain a coupling of the zero cell of an oblique Mondrian tessellation in $\RR^d$ and standard Mondrian tessellation in $\RR^m$. In the following, $B^{+}$ denotes the Moore-Penrose pseudoinverse of a matrix $B$.

\begin{corollary}\label{cor:obliqueMondrian_zerocell}    

Let $A \in \RR^{d \times m}$ have rank $d \leq m$ and fix $\lambda > 0$. Let $\mathcal{P}_M := \mathcal{P}_M\left(\frac{m\lambda}{\|A\|_{2,1}}\right)$ be a Mondrian tessellation in $\RR^m$ with lifetime $\frac{m\lambda}{\|A\|_{2,1}}$ and $Z_{0}^{(M)}$ its zero cell. 
Then, $(A^T)^{+}(Z_0^{(M)} \cap \mathrm{ran}(A^T))$ has the same distribution as the zero cell $Z_0$ of the oblique Mondrian tessellation $\mathcal{P}_A(\lambda)$ with lifetime $\lambda$ with cell boundaries $\mathcal{Y}_A(\lambda)$ as in Proposition \ref{p:Mondrian_image}.
\end{corollary}

\section{Conclusion}\label{sec:conclusion}

In this work, we have studied a class of oblique randomized decision trees and forests that split data along features obtained by taking linear combinations of the covariates. Given this set of features, which can be chosen using domain knowledge or estimated from data, the random partition used to build the tree estimators is generated using a Mondrian process. This method is equivalent to partitioning the original data with a more general STIT process we call an oblique Mondrian process where the directional distribution is discrete, allowing us to build on the theoretical framework developed in \cite{OReillyTran2021minimax} at the intersection of random tessellation theory in stochastic geometry and statistical learning theory. 

This study sought to understand the statistical advantages of using these oblique directions in the input domain to make splits when building a random forest estimator. Our analysis makes clear and rigorous that one such advantage of these random forest variants is their ability to capture low dimensional structure in the regression function described by the class of multi-index models, also called ridge functions. These are linear dimension reduction models for which the output depends on a general low-dimensional relevant feature subspace of the input domain. We obtained convergence rates (see Corollaries \ref{cor:C1_Transformed_Mondrain_Rate} and \ref{cor:C2_Transformed_Mondrain_Rate}) for general oblique Mondrian forests that depend on a parameter controlling the error between the features and associated weights used to make splits and the true relevant features for the regression model. We also illuminated how quickly this error should decay with the amount of data to achieve minimax optimal rates for this model class. Further, we showed that without the ability to divide the data along linear combinations of covariates that approximate vectors spanning this subspace, the geometry of axis-aligned random partitions prevents the associated randomized decision trees from adapting to general ridge functions (see Theorem \ref{thm:Mondrian_suboptimal}). In particular, weighted Mondrian trees cannot achieve the improved rates of convergence that oblique Mondrian trees can for general ridge functions no matter how the distribution over the covariates for making splits is asymptotically reweighted.

Not considered in this study is an algorithm for how to choose the features, or equivalently, the linear transformation $A$, such that these theoretical rates are achieved. To obtain improved rates over the minimax rates with respect to the dimension of the ambient input space, this relevant feature subspace must be consistently estimated. Several such methods exist in the literature to do so by estimating a matrix that approximates a projection onto this subspace \cite{SIR_Li1991, dennis2000save, Xia2002MAVE, JMLR:v11:wu10a, Trivedietal2014} and a subject of future work is the study of complete algorithms for high dimensional regression that are both computationally efficient and provably achieve these improved rates of convergence.

Another future direction is to study the statistical advantage of randomized decision tree and forest variants that use both oblique splits and optimization procedures for choosing the location of the splits. Mondrian forests choose the location uniformly at random after having chosen the feature along which to split. The advantage of choosing this location in a data-driven way intuitively would be to capture local variation and feature importance, but this is not captured by the class of ridge functions studied here, which describes a low-dimensional subset of globally relevant features. Recent work \cite{Cui_DimensionReductionForests2022} has argued with numerical studies that criteria such as CART are more powerful in capturing this local or nonlinear low-dimensional structure, but more theoretical justification and interpretation is needed.

\section{Selected Proofs}\label{sec:proofs}

We collect here the proofs for some of the main results in this paper including Theorem \ref{t:Lipschitz_Rate}, Theorem \ref{t:C1_Transformed_Mondrain_Bound}, and Corollary \ref{t:rate1_Mondrian}. The proofs of the remaining results appear in the Appendix.

\subsection{Proof of Theorem \ref{t:Lipschitz_Rate}}

Let $\hat{f}_{n,\lambda}$ denote a random tree estimator of $f$ obtained from a STIT tessellation $\mathcal{P}(\lambda)$ of the input space with associated zonoid $\Pi$ and lifetime parameter $\lambda$. 
The proof of Theorem \ref{t:Lipschitz_Rate} begins with the following bias-variance decomposition of the risk of a tree estimator presented in \cite{arlot2014analysis}. First, let $Z_x^{\lambda}$ denote the cell of $\mathcal{P}(\lambda)$ that contains the vector $x \in \RR^d$, and define
\begin{align}\label{e:fbar}
\bar{f}_{\lambda}(x) := \EE_X[f(X) | X \in Z^{\lambda}_x], \quad x \in W,
\end{align}
where here and throughout the following, $\mathbb{E}_X$ denotes the expectation with respect to the input random variable $X$.
Conditioned on $\mathcal{P}(\lambda)$, this is the orthogonal projection of $f \in L^2(W, \mu)$ onto the subspace of functions that are constant within the cells of $\mathcal{P}(\lambda) \cap W$.

Conditioning on the data $\mathcal{D}_n$, $\hat{f}_{n, \lambda}$ is in this subspace of piecewise constant functions, and hence $\EE_X[(f(X) - \bar{f}_{\lambda}(X))\hat{f}_{n,\lambda}(X)] = 0$. Thus,
\begin{align*}
    \EE_X[(f(X) - \hat{f}_{\lambda,n}(X))^2] &= \EE_X[(f(X) - \bar{f}_{\lambda}(X) + \bar{f}_{\lambda}(X) - \hat{f}_{n,\lambda}(X))^2]  \\
   & = \EE_X[(f(X) - \bar{f}_{\lambda}(X))^2] + \EE_X[(\bar{f}_{\lambda}(X) - \hat{f}_{n,\lambda}(X))^2].
\end{align*}
Taking the expectation with respect to $\mathcal{P}(\lambda)$ and $\mathcal{D}_n$, we obtain
\begin{align}\label{e:bias-var}
    \EE[(f(X) - \hat{f}_{n,\lambda}(X))^2] = \EE[(f_{\lambda}(X) - \bar{f}_{\lambda}(X))^2] + \EE[(\bar{f}(X) - \hat{f}_{n,\lambda}(X))^2].
\end{align}
The first term on the right-hand side above is called the bias, or approximation error, of the estimator and the second term is the variance, or estimation error. The bound on the risk then depends on the following two lemmas, which bound each of these expressions.

\begin{lemma}\label{l:fixed_x_bias_bnd}
Let $\bar{f}_{\lambda}(x)$ be defined as in \eqref{e:fbar}. Under the assumptions on $f$ in Theorem \ref{t:Lipschitz_Rate}, for any fixed $x \in \text{supp}(\mu)$, 
\begin{align*}
\EE[(f(x) - \bar{f}_{\lambda}(x))^{2}] \leq \frac{L^2}{\lambda^{2\beta}} \EE[\mathrm{D}(P_SZ_{0})^{2\beta}].
\end{align*}
\end{lemma}
\begin{proof}
By the assumption on $f$,
\begin{align*}
    |f(x) - \bar{f}_{\lambda}(x)|
    &= \frac{1}{\mu(Z^{\lambda}_x)} \int_{\RR^d} \left|f(x) - f(z)\right| 1_{\{z \in Z_x^{\lambda}\}} \mu(dz) \\
    &\leq \frac{L}{\mu(Z^{\lambda}_x)} \int_{\RR^d}\|P_{S}(x - z)\|^{\beta} 1_{\{z \in Z_x^{\lambda}\}} \mu(dz)\\
     &\leq \frac{L\mathrm{D}(P_SZ_x^{\lambda})^{\beta}}{\mu(Z^{\lambda}_x)} \int_{\RR^d} 1_{\{z\in Z_x^{\lambda}\}} \mu(dz) = L\mathrm{D}(P_SZ_x^{\lambda})^{\beta}.
\end{align*}
By stationarity and \eqref{e:scaling}, for any fixed $x \in \RR^d$, $$Z_{x}^{\lambda} \overset{(d)}{=} \frac{1}{\lambda}Z_{0} + x.$$
Thus, taking the expectation with respect to the random tessellation $\mathcal{P}(\lambda)$ gives 
\begin{align*}
\EE[(f(x) - \bar{f}_{\lambda}(x))^{2}] \leq 
\frac{L^2}{\lambda^{2\beta}} \EE[\mathrm{D}(P_SZ_{0})^{2\beta}].
\end{align*}
\end{proof}

We next prove an upper bound on the variance that highlights the effect of choosing a directional distribution with support concentrated around a subspace $S$. In particular, the upper bound below reduces to the bound obtained from Lemma 4 and Example 3 of \cite{OReillyTran2021minimax} if $s = d$. Also note that if the support of the direction distribution is concentrated in $S$, then the associated zonoid $\Pi$ is contained in $S$ and the variance bound is that for a random tessellation tree estimator in $\RR^s$.

\begin{lemma}\label{l:general_variance}
Suppose $\mathrm{supp}(\mu) \subseteq B^d$. Then,
\begin{align*}
    &\EE\left[(\bar{f}_{\lambda}(X) - \hat{f}_{\lambda,n}(X))^2\right] \\ %\leq \frac{5\|f\|_{\infty}^2 + 2\sigma^2}{n} \left(\sum_{k=s+1}^d c_{d,k}\lambda^k \mathrm{D}(P_{S^{\perp}}\Pi)^{k-s} +  \sum_{k=0}^{s} c_{d,k} \lambda^k\right),
    &\leq \frac{5\|f\|_{\infty}^2 + 2\sigma^2}{n} \left(2s\sum_{k=1}^d \lambda^k\kappa_k  V_1(P_{S^{\perp}}\Pi)^{1 \vee (k-s)} + \sum_{k=0}^{s}\lambda^k \kappa_k V_{k}(P_S\Pi)\right).
 \end{align*}
% where $c_{d,k} := \frac{\kappa_k\pi^{k/2}d^{k/2}}{k!}$.
\end{lemma}

\begin{proof}
Let $N_{\lambda}(K)$ be the number of cells of $\mathcal{P}(\lambda)$ that have a non-empty intersection with a compact subset $K \subset \RR^d$. By Lemma 15 in \cite{OReillyTran2021minimax},
\begin{align}\label{e:var_bnd}
    \EE\left[(\bar{f}_{\lambda}(X) - \hat{f}_{\lambda,n}(X))^2\right] \leq \frac{5\|f\|_{\infty}^2 + 2\sigma^2}{n} \EE[N_{\lambda}(\mathrm{supp}(\mu))].
\end{align}
Recall that for a convex body $K$, $V_k(\Pi) = \frac{\binom{d}{k}}{\kappa_{d-k}} V(K[k], B^d[d-k])$ \cite[(14.18)]{weil}. Then, by the assumption $\mathrm{supp}(\mu) \subseteq B^d$ 
and Lemma 4 in \cite{OReillyTran2021minimax},
\begin{align*}
\EE[N_{\lambda}(\mathrm{supp}(\mu))] \leq \EE[N_{\lambda}(B^d)] &= \mathrm{vol}_d(\Pi)\sum_{k=0}^d  \binom{d}{k} \lambda^k  \EE[V(B^d[k], Z[d-k])] \\
&= \mathrm{vol}_d(\Pi)\sum_{k=0}^d \lambda^k \kappa_k \EE[V_{d-k}(Z)].  
\end{align*}
By (10.3) and Theorem 10.3.3 in \cite{weil}, $\EE V_{d-k}(Z) =  \frac{V_{k}(\Pi)}{\mathrm{vol}_d(\Pi)}$. Thus,
\begin{align}\label{e:Nlambda_sumbnd}
  \EE[N_{\lambda}(\mathrm{supp}(\mu))] \leq \sum_{k=0}^d \lambda^k \kappa_k  V_k(\Pi).
\end{align}
Note that $\Pi \subseteq P_S\Pi + P_{S^{\perp}}\Pi$ for any linear subspace $S$. By monotonicity and multilinearity of mixed volumes with respect to the Minkowski sum, we have for each $k \in \{1, \ldots, d\}$,
\begin{align*}
V(\Pi[k], B^d[d-k]) &\leq V\left((P_{S}\Pi + P_{S^{\perp}}\Pi)[k], B^d[d-k]\right) \\
&= \sum_{j=0}^{k} \binom{k}{j}V(P_{S}\Pi[j], P_{S^{\perp}}\Pi[k-j], B^d[d-k]). 
\end{align*}
Observe that if $k-j > d-s$ or $j > s$, then $V(P_{S}\Pi[k-j], P_{S^{\perp}}\Pi[j], B^d[d-k]) = 0$. 
Then by Theorem 1.3 in \cite{BoroczkyHug_ReverseMinkowksi},
\begin{align*}
 V(P_{S}\Pi[k-j], P_{S^{\perp}}\Pi[j], B^d[d-k]) \leq \frac{\kappa_{d-k}}{\binom{d}{d-k, k-j, j}}V_{j}(P_S\Pi)V_{k-j}(P_{S^{\perp}}\Pi)\mathbf{1}_{\{k - (d-s)\leq j \leq s\}}.   
\end{align*}
Then,
\begin{align*}
 \EE[N_{\lambda}(\mathrm{supp}(\mu))] &\leq \sum_{k=0}^d \frac{ \lambda^k \kappa_k }{\kappa_{d-k}}\binom{d}{k} V(\Pi[k], B_d[d-k]) \\
 & \leq \sum_{k=0}^d \frac{ \lambda^k \kappa_k }{\kappa_{d-k}}\binom{d}{k} \sum_{j=0}^{k} \binom{k}{j} \frac{\kappa_{d-k}}{\binom{d}{d-k, k-j, j}}V_{j}(P_S\Pi)V_{k-j}(P_{S^{\perp}}\Pi)\mathbf{1}_{\{k-(d-s) \leq j \leq s\}}  \\
  %&= \sum_{k=0}^d \lambda^k \kappa_k\binom{d}{k} \sum_{j=0}^{k} \frac{k!(d-k)!}{d!}V_{j}(P_S\Pi)V_{k-j}(P_{S^{\perp}}\Pi)\mathbf{1}_{\{k-(d-s)\leq j \leq s\}}  \\
    & = \sum_{k=0}^d \lambda^k \kappa_k \sum_{j=0}^{k}V_{j}(P_S\Pi)V_{k-j}(P_{S^{\perp}}\Pi)\mathbf{1}_{\{k-(d-s) \leq j \leq s\}}  \\
 & \leq \sum_{k=1}^d \lambda^k \kappa_k \sum_{j=0}^{s \wedge (k-1)} V_{j}(P_S\Pi)V_{k-j}(P_{S^{\perp}}\Pi) + \sum_{k=0}^{s}\lambda^k \kappa_k V_{k}(P_S\Pi).
\end{align*}
Now observe that for any associated zonoid $\Pi$, by \eqref{e:V1} and \eqref{e:Pi_support_fxn}, the first intrinsic volume satisfies
\begin{align}\label{e:V1_Pi}
 V_1(\Pi) = \frac{d\kappa_d}{\kappa_{d-1}}\int_{\mathbb{S}^{d-1}}h_{\Pi}(u) \dint \sigma(u) = \frac{d\kappa_d}{2\kappa_{d-1}} \int_{\mathbb{S}^{d-1}} \int_{\mathbb{S}^{d-1}} |\langle u,v \rangle | \dint \phi(u) \dint \sigma(u) = 1.   
\end{align}
Also, by Theorem 2 in \cite{McMullen1991}, all intrinsic volumes can be controlled via the first intrinsic volume, giving the upper bound
\begin{align*}
 V_{j}(P_S\Pi)V_{k-j}(P_{S^{\perp}}\Pi) &\leq \frac{1}{j!(k-j)!}V_1(P_S\Pi)^jV_1(P_{S^{\perp}}\Pi)^{k-j} %\\
 %&\leq \frac{1}{j!(k-j)!} V_1(\Pi)^jV_1(P_{S^{\perp}}\Pi)^{k-j} \\
 %&\leq \frac{(\pi s)^{j/2}}{j!(k-j)!2^j}V_1(P_{S^{\perp}}\Pi)^{k-j} \\
 \leq \frac{1}{j!(k-j)!}V_1(P_{S^{\perp}}\Pi)^{k-j}.
\end{align*}
Plugging this upper bound into the above inequality and observing $V_1(P_{S^{\perp}}\Pi) \leq V_1(\Pi) = 1$, we obtain
\begin{align*}
 \EE[N_{\lambda}(\mathrm{supp}(\mu))] 
 &\leq \sum_{k=1}^d \lambda^k \kappa_k \sum_{j=0}^{s \wedge (k-1)} \frac{1}{j!(k-j)!}V_1(P_{S^{\perp}}\Pi)^{k-j} + \sum_{k=0}^{s}\lambda^k \kappa_k V_{k}(P_S\Pi) \\
 &\leq \sum_{k=1}^d \lambda^k \kappa_k \sum_{j=0}^{s \wedge (k-1)} \frac{1}{j!(k-j)!}V_1(P_{S^{\perp}}\Pi)^{k-j} + \sum_{k=0}^{s}\lambda^k \kappa_k V_{k}(P_S\Pi) \\
 %&= \sum_{k=s+1}^d \lambda^k \kappa_k \sum_{j=0}^{s} \frac{1}{j!(k-j)!}V_1(P_{S^{\perp}}\Pi)^{k-j} + \sum_{k=1}^{s} \lambda^k \kappa_k \sum_{j=0}^{k-1} \frac{1}{j!(k-j)!}V_1(P_{S^{\perp}}\Pi)^{k-j} \\
 %& \qquad \qquad + \sum_{k=0}^{s}\lambda^k \kappa_k V_{k}(P_S\Pi) \\
 %&\leq  \sum_{k=s + 1}^d \lambda^k\kappa_k(s+1) V_1(P_{S^{\perp}}\Pi)^{k-s} + \sum_{k=1}^s\lambda^k k \kappa_k V_1(P_{S^{\perp}}\Pi) + \sum_{k=0}^{s}\lambda^k \kappa_k V_{k}(P_S\Pi) \\
 &\leq  2s\sum_{k=1}^d \lambda^k\kappa_k  V_1(P_{S^{\perp}}\Pi)^{1 \vee (k-s)} + \sum_{k=0}^{s}\lambda^k \kappa_k V_{k}(P_S\Pi).
\end{align*}
\end{proof}

\begin{proof}[Proof of Theorem \ref{t:Lipschitz_Rate}]
Combining the bias-variance decomposition \eqref{e:bias-var} with the upper bounds in Lemma \ref{l:fixed_x_bias_bnd} and Lemma \ref{l:general_variance} gives
\begin{align*}
   &\EE[(f(X) - \hat{f}_{\lambda,n}(X))^2] \\
   &= \EE[(f_{\lambda}(X) - \bar{f}_{\lambda}(X))^2] + \EE[(\bar{f}(X) - \hat{f}_{\lambda,n}(X))^2] \\
   &\leq \frac{L^2}{\lambda^2} \EE[\mathrm{D}(P_SZ_{0})^2] + \frac{5\|f\|_{\infty}^2 + 2\sigma^2}{n}\left(2s\sum_{k=1}^d \lambda^k\kappa_k  V_1(P_{S^{\perp}}\Pi)^{1 \vee (k-s)} + \sum_{k=0}^{s}\lambda^k \kappa_k V_{k}(P_S\Pi)\right).%\left(\sum_{k=s+1}^d c_{d,k}\lambda^k \mathrm{D}(P_{S^{\perp}}\Pi)^{k-s} +  \sum_{k=0}^{s} c_{d,k} \lambda^k\right),
 \end{align*}
 %where $c_{d,k} := \frac{\kappa_k\pi^{k/2}d^{k/2}}{k!}$. 
 The final result follows from the observation that the risk of a STIT forest estimator for any number of trees $M$ is bounded above by the risk of a single STIT tree estimator by Jensen's inequality.
\end{proof}

\subsection{Proof of Theorem \ref{t:C1_Transformed_Mondrain_Bound} and Corollary \ref{cor:C1_Transformed_Mondrain_Rate}}

We first need the following lemma on the diameter of the zero cell of the random tessellation generated by an oblique Mondrian process.

\begin{lemma}\label{l:Deter}
Suppose that $Z_0$ is the zero cell of a STIT tessellation with unit lifetime and directional distribution $\phi_A$ as in \eqref{e:model_dirdist} for nonsingular $A \in \RR^{d \times d}$ and $\|A\|_{2, 1} = 1$. Then, %for all $r \geq 0$ and $k > 0$,
%\begin{align*}
%\EE\left[\mathrm{D}(P_SZ_0)^k1_{\{\mathrm{D}(P_{S}Z_{0}) \geq r\}}\right] &\leq \frac{\Gamma(2d +k)}{\Gamma(2d)}\sum_{n=0}^{2d + k - 1}\frac{r^n\sigma_{s}(P_SA)^{n-k}}{n!} e^{-r \sigma_{s}(P_SA)},
%\end{align*}
%where $\sigma_s$ is the $s$-th largest singular value. In particular, 
for all $p > 0$,
\begin{align*}
    \EE[\mathrm{D}(P_SZ_0)^p] \leq \frac{\Gamma(2d + p)}{\sigma_{s}(P_SA)^p\Gamma(2d)},
\end{align*}
where $\sigma_s$ is the $s$-th largest singular value. 
\end{lemma}

\begin{proof}
The distribution of the zero cell $Z^{(M)}_{0}$ for the Mondrian tessellation in $\RR^d$ with lifetime $d$ is given by
\[Z^{(M)}_{0} \overset{(d)}{=} \left([-T_{1}^{(1)}e_1, T_{1}^{(2)}e_1] + \cdots + [-T_{d}^{(1)}e_d, T_{d}^{(2)}e_d]\right),\]
where $\{T_i^{(j)}\}$ for $i = 1, \ldots, d$ and $j = 1 , 2$ are independent and identically distributed exponential random variables with unit parameter.
By Corollary \ref{cor:obliqueMondrian_zerocell}, the zero cell $Z_{0}$ as defined in the lemma has the same distribution as $(A^{-1})^TZ^{(M)}_{0}$. Then, the support function of $Z_0$ satisfies 
\begin{align*}
h_{Z_{0}}(u) &= h_{(A^{-1})^TZ^{(M)}_{0}}(u) = h_{Z^{(M)}_{0}}(A^{-1}u) \\
&= \sum_{i=1}^d \max\{\langle A^{-1}u, -T_i^{(1)}e_i \rangle, \langle A^{-1}u, T_i^{(2)}e_i \rangle\} \\
&= \sum_{i=1}^d \max\{- T_i^{(1)}\langle A^{-1}u, e_i \rangle, T_i^{(2)}\langle A^{-1}u, e_i \rangle\},
\end{align*}
and the width function of $Z_0$ satisfies
\begin{align}\label{e:wZ0}
w_{Z_0}(u) &:= h_{Z_0}(u) + h_{Z_{0}}(-u) %\nonumber \\
%&= \sum_{i=1}^d \left(\max\{- T_i^{(1)}\langle A^{-1}u,e_i \rangle, T_i^{(2)}\langle A^{-1}u, e_i \rangle\} +  \max\{ T_i^{(1)}\langle A^{-1}u,e_i \rangle, -T_i^{(2)}\langle A^{-1}u, e_i \rangle\}\right) \nonumber \\
= \sum_{i=1}^d \left(T^{(1)}_i + T^{(2)}_i \right)|\langle A^{-1}u, e_i\rangle|.
\end{align}
Then, recalling that $w_{AK}(u) = w_{K}(A^Tu)$ for a convex body $K$ and linear image $A$, the diameter of $P_SZ_0$ satisfies
\begin{align}\label{e:diam_obliqueMondrian}
    D(P_SZ_0) &= \sup_{u \in \mathbb{S}^{d-1}} w_{P_SZ_0}(u) = \sup_{u \in \mathbb{S}^{d-1}} w_{Z_0}(P_S u) \nonumber \\
    &=  \sup_{u \in \mathbb{S}^{d-1}}  \sum_{i=1}^d \left(T^{(1)}_{i} + T^{(2)}_{i}\right) |\langle A^{-1}P_Su, e_i\rangle| \nonumber \\
    &=  \left\|\sum_{i=1}^d \left(T^{(1)}_{i} + T^{(2)}_{i}\right)P_S(A^{-1})^Te_i \right\|_2.
\end{align}
Thus, we have the upper bound
\begin{align}\label{e:diam_bnd}
    D(P_SZ_0) &\leq \|(P_SA)^{+}\|  \sum_{i=1}^d \left(T^{(1)}_{i} + T^{(2)}_{i}\right) = \frac{1}{\sigma_{s}(P_SA)}  \sum_{i=1}^d \left(T^{(1)}_{i} + T^{(2)}_{i}\right),
\end{align}
where we have used the fact that $P_S(A^{-1})^T = (A^{-1}P_S)^T = ((P_SA)^{+})^T$ and $B^+$ denotes the Moore-Penrose pseudoinverse of the matrix $B$.
Thus, the diameter of $P_SZ_0$ is controlled by the sum of independent exponential random variables, which is an Erlang random variable $$T^{(d)}:= \sum_{i=1}^d \left(T^{(1)}_{i} + T^{(2)}_{i}\right) \sim \mathrm{Erlang}\left(2d, 1\right).$$
Thus, %for $r > 0$ and $k \in \mathbb{N}$,
%\begin{align*}
%\EE\left[\mathrm{D}(P_SZ_0)^k1_{\{\mathrm{D}(P_{S}Z_{0}) \geq r\}}\right] &\leq \frac{1}{\sigma_s(P_SA)^k}\EE\left[(T^{(d)})^k \mathbf{1}_{\{T^{(d)} \geq r \sigma_{s}(P_SA)\}}\right] \\
% &= \frac{\Gamma(2d +k)}{\sigma_s(P_SA)^k\Gamma(2d)}\sum_{n=0}^{2d + k - 1}\frac{1}{n!}\left( r\sigma_{s}(P_SA) \right)^{n} e^{-r\sigma_{s}(P_SA)},
%\end{align*}
%and 
the moments of the diameter of $P_SZ_0$ for any $p > 0$ satisfy the upper bound
\begin{align*}
\EE[\mathrm{D}(P_SZ_0)^p] &\leq \frac{\EE[(T^{(d)})^p]}{\sigma_s(P_SA)^p} = \frac{\Gamma(2d + p)}{\sigma_{s}(P_SA)^p\Gamma(2d)}.
\end{align*}
\end{proof}

\begin{proof}[Proof of Theorem \ref{t:C1_Transformed_Mondrain_Bound}]
First recall the bias-variance decomposition \eqref{e:bias-var} for a STIT tessellation tree used in the proof of Theorem \ref{t:Lipschitz_Rate}.
Now let $\hat{f}_{n, \lambda}$ be an oblique Mondrian forest estimator as in the statement of Theorem \ref{t:C1_Transformed_Mondrain_Bound} for a matrix $A \in \RR^{d \times m}$ with rank $d \leq m$ and such that $\|A\|_{2,1} = 1$.

To bound the bias term, Lemmas \ref{l:fixed_x_bias_bnd} and \ref{l:Deter} imply that
\begin{align*}
    \EE\left[\left(f(X) - \bar{f}_{\lambda}(X)\right)^2\right] &\leq \frac{L^2\EE[\mathrm{D}(P_{S}Z_{0})^{2\beta}]}{\lambda^{2\beta}}  \leq \frac{L^2 \Gamma(2d+2\beta)}{\lambda^{2\beta}\sigma_s(P_SA)^{2\beta}\Gamma(2d)} \leq \frac{9L^2d^{2\beta}}{\lambda^{2\beta}\sigma_s(P_SA)^{2\beta}},
\end{align*}
where in the last inequality we used Gautschi's inequality to obtain the bound 
$$\Gamma(2d + 2\beta) \leq (2d + 1)^{2\beta - 1}(2d)\Gamma(2d)\leq 9d^{2\beta}\Gamma(2d).$$

To bound the variance term, we first observe that inserting the directional distribution \eqref{e:model_dirdist} into \eqref{e:Pi_support_fxn} implies that the associated zonoid $\Pi$ corresponding to the oblique Mondrian process used to generate $\hat{f}_{n,\lambda}$ satisfies
\begin{align}\label{e:P_Sperp_zonoid_diam}
    V_1(P_{S^{\perp}}\Pi) = \frac{d\kappa_d}{\kappa_{d-1}}\int_{\mathbb{S}^{d-1}} h_{\Pi}(P_{S^{\perp}} u) \dint \sigma(u) = \sum_{i=1}^d \frac{d\kappa_d}{\kappa_{d-1}}\int_{\mathbb{S}^{d-1}} |\langle P_{S^{\perp}}a_i,  u\rangle| \dint \sigma(u) = \|P_{S^{\perp}} A\|_{2,1}.
\end{align}
Then by Theorem 2 in \cite{McMullen1991} and \eqref{e:V1_Pi}, for all $k= 1, \ldots, s$,
\begin{align*}
    V_k(P_S\Pi) &\leq \frac{1}{k!}V_1(P_S\Pi)^k \leq \frac{1}{k!}V_1(\Pi)^k = \frac{1}{k!}.% \frac{1}{k!} \left(\frac{d\kappa_d}{\kappa_{d-1}}\int_{\mathbb{S}^{d-1}} h_{\Pi}(u) \dint \sigma(u)\right)^k \\
    %&= \frac{1}{k!}\left(\sum_{i=1}^d \frac{d\kappa_d}{\kappa_{d-1}}\int_{\mathbb{S}^{d-1}} |\langle a_i,  u\rangle| \dint \sigma(u)\right)^k = \frac{1}{k!} \|A\|_{2,1}^k = \frac{1}{k!}.
\end{align*}
%\begin{align}\label{e:P_Sperp_zonoid_diam}
%   \mathrm{D}(P_{S^{\perp}}\Pi) &\leq \sup_{u \in \mathbb{S}^{d-1} \cap S^{\perp}} \left(h_{\Pi}(u) + h_{\Pi}(-u)\right) = \sup_{u \in \mathbb{S}^{d-1} \cap S^{\perp}} \sum_{i=1}^d |\langle a_i, u \rangle| \nonumber \\
%   &=  \sup_{u \in \mathbb{S}^{d-1} \cap S^{\perp}} \sum_{i=1}^d |\langle P_{S^{\perp}}a_i, u \rangle| \leq  \sum_{i=1}^d \| P_{S^{\perp}}a_i\|_2 = \|P_{S^{\perp}}A\|_{2,1}.
% \end{align}
Thus, by Lemma \ref{l:general_variance}, 
\begin{align}\label{e:oblique_mondrian_varbnd}
\EE\left[(\bar{f}_{\lambda}(X) - \hat{f}_{\lambda,n}(X))^2\right] &\leq \frac{5\|f\|_{\infty}^2 + 2\sigma^2}{n} \left(2s\sum_{k=1}^d \lambda^k\kappa_k \|P_{S^{\perp}}A\|_{2,1}^{ 1 \vee (k-s)} + \sum_{k=0}^{s} \frac{\lambda^k\kappa_k}{k!}\right).%\left(\sum_{k=s+1}^d c_{d,k}\lambda^k \|P_{S^{\perp}}A\|_{2,1}^{k-s}+  \sum_{k=0}^{s} c_{d,k} \lambda^k\right),
\end{align}
Combining these bounds with \eqref{e:bias-var}, and again observing that by Jensen's inequality the risk of a STIT forest estimator for any number of trees $M$ is bounded above by the risk of a single STIT tree, gives the final result.
\end{proof}

\begin{proof}[Proof of Corollary \ref{cor:C1_Transformed_Mondrain_Rate}]

Under the assumptions of the Corollary, for the sequence of oblique Mondrian forest estimators $\hat{f}_n$ defined there, Theorem \ref{t:C1_Transformed_Mondrain_Bound} implies
\begin{align*}
\EE\left[(f(X) - \hat{f}_{n}(X))^2\right] 
&\leq \frac{9L^2m^{2\beta}}{c^2d^2\lambda_n^{2\beta}} + 
\frac{5\|f\|_{\infty}^2 + 2\sigma^2}{n} \left(2s\sum_{k=1}^d \lambda_n^k \kappa_k \ee_n^{1 \vee (k-s)} +  O(\lambda_n^{s})\right).%\left(\sum_{k=s}^d c_{d,k} \lambda_n^k \ee_n^{k-s} +  O(\lambda_n^{s-1})\right).
\end{align*}
Minimizing the upper bound with respect to $\lambda_n$ gives that for $\lambda_n \asymp L^{\frac{2}{d + 2\beta}}n^{\frac{1}{d + 2\beta}} \ee_n^{-\frac{(d-s)}{d + 2\beta}}$,
\begin{align*}
\EE\left[(f(X) - \hat{f}_{n}(X))^2\right] 
&\lesssim \max\left\{L^{\frac{2d}{d + 2\beta}}n^{-\frac{2\beta}{d + 2\beta}} \ee_n^{\frac{2\beta(d-s)}{d + 2\beta}}, L^{\frac{2s}{s + 2\beta}}n^{-\frac{2\beta}{s + 2\beta}}\right\}.
\end{align*}
The final claim follows from the observation that by letting $\ee_n \lesssim L^{-\frac{2}{s + 2\beta}}n^{-\frac{1}{s + 2\beta}}$ and $\lambda_n \asymp L^{\frac{2}{s + 2\beta}}n^{\frac{1}{s + 2\beta}}$, the upper bound above satisfies
\begin{align*}
\EE\left[(f(X) - \hat{f}_{n}(X))^2\right] 
&\lesssim L^{\frac{2s}{s + 2\beta}}n^{-\frac{2\beta}{s + 2\beta}}.
\end{align*}

\end{proof}

\section*{Acknowledgments}
The author would like to thank Yangxinyu Xie, Ngoc Mai Tran, and Werner Nagel for their valuable suggestions and corrections to improve this manuscript. The author is grateful for support from NSF Grant DMS-2402234.

\bibliographystyle{plain}
\bibliography{Biblio}

\begin{thebibliography}{10}

\bibitem{amit_shape_1997}
Yali Amit and Donald Geman.
\newblock Shape {Quantization} and {Recognition} with {Randomized} {Trees}.
\newblock {\em Neural Computation}, 9(7):1545--1588, October 1997.

\bibitem{arlot2014analysis}
Sylvain Arlot and Robin Genuer.
\newblock Analysis of purely random forests bias.
\newblock {\em Preprint arXiv:1407.3939}, 2014.

\bibitem{Bach2017_curse}
Francis Bach.
\newblock Breaking the curse of dimensionality with convex neural networks.
\newblock {\em Journal of Machine Learning Research}, 18(19):1--53, 2017.

\bibitem{Baptista_TrIM}
Ricardo Baptista, Eliza O'Reilly, and Yangxinyu Xie.
\newblock Tr{IM}: Transformed iterative {M}ondrian forests for gradient-based
  dimension reduction and high-dimensional regression.
\newblock {\em Preprint arXiv:2407.09964}, 2024.

\bibitem{Biau2012}
Gerard Biau.
\newblock Analysis of a random forests model.
\newblock {\em Journal of Machine Learning Research}, 13:1063--1095, 2012.

\bibitem{blaser2016random}
Rico Blaser and Piotr Fryzlewicz.
\newblock Random rotation ensembles.
\newblock {\em Journal of Machine Learning Research}, 17(1):126--151, 2016.

\bibitem{BoroczkyHug_ReverseMinkowksi}
K\'{a}roly~J. B\"{o}r\"{o}czky and Daniel Hug.
\newblock Reverse {A}lexandrov–{F}enchel inequalities for zonoids.
\newblock {\em Communications in Contemporary Mathematics}, 24(8), 2022.

\bibitem{breiman2001random}
Leo Breiman.
\newblock Random forests.
\newblock {\em Machine learning}, 45(1):5--32, 2001.

\bibitem{Klusowski_oblique_2023}
Matias~D. Cattaneo, Rajita Chandak, and Jason~M. Klusowski.
\newblock {Convergence rates of oblique regression trees for flexible function
  libraries}.
\newblock {\em The Annals of Statistics}, 52(2):466 -- 490, 2024.

\bibitem{KlusowskiMondrian2023}
Matias~Damian Cattaneo, Jason~Matthew Klusowski, and William~George Underwood.
\newblock Inference with mondrian random forests.
\newblock {\em Journal of the Royal Statistical Society Series B: Statistical
  Methodology}, page qkaf077, 2025.

\bibitem{chen2012random}
Xi~Chen and Hemant Ishwaran.
\newblock Random forests for genomic data analysis.
\newblock {\em Genomics}, 99(6):323--329, 2012.

\bibitem{ChietalnewRates2021}
Chien-Ming Chi, Patrick Vossler, Yingying Fan, and Jinchi Lv.
\newblock {Asymptotic properties of high-dimensional random forests}.
\newblock {\em The Annals of Statistics}, 50(6):3415 -- 3438, 2022.

\bibitem{deng2022towards}
Wenying Deng, Beau Coker, Rajarshi Mukherjee, Jeremiah~Zhe Liu, and Brent~A.
  Coull.
\newblock Towards a unified framework for uncertainty-aware nonlinear variable
  selection with theoretical guarantees.
\newblock In Alice~H. Oh, Alekh Agarwal, Danielle Belgrave, and Kyunghyun Cho,
  editors, {\em Advances in Neural Information Processing Systems}, 2022.

\bibitem{dennis2000save}
R~Dennis~Cook.
\newblock Save: a method for dimension reduction and graphics in regression.
\newblock {\em Communications in statistics-Theory and methods},
  29(9-10):2109--2121, 2000.

\bibitem{DurouxScornet2018}
{Duroux, Roxane} and {Scornet, Erwan}.
\newblock Impact of subsampling and tree depth on random forests.
\newblock {\em ESAIM: PS}, 22:96--128, 2018.

\bibitem{pmlr-v84-fan18b}
Xuhui Fan, Bin Li, and Scott Sisson.
\newblock The binary space partitioning-tree process.
\newblock In Amos Storkey and Fernando Perez-Cruz, editors, {\em Proceedings of
  the Twenty-First International Conference on Artificial Intelligence and
  Statistics}, volume~84 of {\em Proceedings of Machine Learning Research},
  pages 1859--1867. PMLR, 09--11 Apr 2018.

\bibitem{fernandez2014we}
Manuel Fern{\'a}ndez-Delgado, Eva Cernadas, Sen{\'e}n Barro, and Dinani Amorim.
\newblock Do we need hundreds of classifiers to solve real world classification
  problems?
\newblock {\em Journal of Machine Learning Research}, 15(1):3133--3181, 2014.

\bibitem{TehRTFs2019}
Shufei Ge, Shijia Wang, Yee~Whye Teh, Liangliang Wang, and Lloyd Elliott.
\newblock Random tessellation forests.
\newblock In {\em Advances in Neural Information Processing Systems 32}, pages
  9571--9581. 2019.

\bibitem{ho_random_1998}
Tin~Kam Ho.
\newblock The random subspace method for constructing decision forests.
\newblock {\em IEEE Transactions on Pattern Analysis and Machine Intelligence},
  20(8):832--844, August 1998.
\newblock Conference Name: IEEE Transactions on Pattern Analysis and Machine
  Intelligence.

\bibitem{Klusowski2021}
Jason Klusowski.
\newblock Sharp analysis of a simple model for random forests.
\newblock In Arindam Banerjee and Kenji Fukumizu, editors, {\em Proceedings of
  The 24th International Conference on Artificial Intelligence and Statistics},
  volume 130 of {\em Proceedings of Machine Learning Research}, pages 757--765.
  PMLR, 13--15 Apr 2021.

\bibitem{KlusowskiTian2024}
Jason~M. Klusowski and Peter~M. Tian.
\newblock Large scale prediction with decision trees.
\newblock {\em Journal of the American Statistical Association},
  119(545):525--537, 2024.

\bibitem{SIR_Li1991}
Ker-Chau Li.
\newblock Sliced inverse regression for dimension reduction.
\newblock {\em Journal of the American Statistical Association},
  86(414):316--327, 1991.

\bibitem{Cui_DimensionReductionForests2022}
Joshua~Daniel Loyal, Ruoqing Zhu, Yifan Cui, and Xin Zhang.
\newblock Dimension reduction forests: Local variable importance using
  structured random forests.
\newblock {\em Journal of Computational and Graphical Statistics},
  31(4):1104--1113, 2022.

\bibitem{McMullen1991}
Peter McMullen.
\newblock Inequalities between intrinsic volumes.
\newblock {\em Monatshefte für Mathematik}, 111(1):47--54, 1991.

\bibitem{Nagel2008}
Joseph Mecke, Werner Nagel, and Viola Weiss.
\newblock The iteration of random tessellations and a construction of a
  homogeneous process of cell divisions.
\newblock {\em Advances in Applied Probability}, 40(1):49--59, March 2008.

\bibitem{Menzeetal2011}
Bjoern~H. Menze, B.~Michael Kelm, Daniel~N. Splitthoff, Ullrich Koethe, and
  Fred~A. Hamprecht.
\newblock On oblique random forests.
\newblock In Dimitrios Gunopulos, Thomas Hofmann, Donato Malerba, and Michalis
  Vazirgiannis, editors, {\em Machine Learning and Knowledge Discovery in
  Databases}, pages 453--469, Berlin, Heidelberg, 2011. Springer Berlin
  Heidelberg.

\bibitem{MolchanovBook}
Ilya Molchanov.
\newblock {\em Theory of Random Sets}, volume~87.
\newblock Springer, 2017.

\bibitem{mourtada2020minimax}
Jaouad Mourtada, St{\'e}phane Ga{\"\i}ffas, and Erwan Scornet.
\newblock Minimax optimal rates for {M}ondrian trees and forests.
\newblock {\em Annals of Statistics}, 28(4):2253--2276, 2020.

\bibitem{Nagel2003}
Werner Nagel and Viola Weiss.
\newblock Limits of sequences of stationary planar tessellations.
\newblock {\em Advances in Applied Probability}, 35:123--138, 2003.

\bibitem{Nagel2005}
Werner Nagel and Viola Weiss.
\newblock Crack {STIT} tessellations: Characterization of stationary random
  tessellations stable with respect to iteration.
\newblock {\em Advances in Applied Probability}, 37:859--883, 2005.

\bibitem{OReillyTran2021}
Eliza O'Reilly and Ngoc~Mai Tran.
\newblock Stochastic geometry to generalize the {M}ondrian process.
\newblock {\em SIAM Journal on Mathematics of Data Science}, 4(2):531--552,
  2022.

\bibitem{OReillyTran2021minimax}
Eliza O'Reilly and Ngoc~Mai Tran.
\newblock Minimax rates for high-dimensional random tessellation forests.
\newblock {\em Journal of Machine Learning Research}, 25:1--32, 2024.

\bibitem{rainforth2015canonical}
Tom Rainforth and Frank Wood.
\newblock Canonical correlation forests.
\newblock {\em Preprint arXiv:1507.05444}, 2015.

\bibitem{roy2008mondrian}
Daniel~M Roy and Yee~Whye Teh.
\newblock The {M}ondrian process.
\newblock In {\em Proceedings of the 21st International Conference on Neural
  Information Processing Systems}, pages 1377--1384, 2008.

\bibitem{Schneider1983Zonoids}
Rolf Schneider and Wolfgang Weil.
\newblock {\em Zonoids and Related Topics}, pages 296--317.
\newblock Birkh{\"a}user Basel, Basel, 1983.

\bibitem{weil}
Rolf Schneider and Wolfgang Weil.
\newblock {\em Stochastic and Integral Geometry}.
\newblock Probability and {I}ts {A}pplications. Springer-Verlag, Berlin, 2008.

\bibitem{ScornetConsistency2015}
Erwan Scornet, G{\'e}rard Biau, and Jean-Philippe Vert.
\newblock {Consistency of random forests}.
\newblock {\em The Annals of Statistics}, 43(4):1716 -- 1741, 2015.

\bibitem{Yuetal2018}
James B.~Brown Sumanta~Basu, Karl~Kumbier and Bin Yu.
\newblock Iterative random forests to discover predictive and stable high-order
  interactions.
\newblock {\em PNAS}, 115(8):1943--1948, 2018.

\bibitem{Syrgkanis2020}
Vasilis Syrgkanis and Manolis Zampetakis.
\newblock Estimation and inference with trees and forests in high dimensions.
\newblock In Jacob Abernethy and Shivani Agarwal, editors, {\em Proceedings of
  Thirty Third Conference on Learning Theory}, volume 125 of {\em Proceedings
  of Machine Learning Research}, pages 3453--3454. PMLR, 09--12 Jul 2020.

\bibitem{Tomita2020}
Tyler~M. Tomita, James Browne, Cencheng Shen, Jaewon Chung, Jesse~L. Patsolic,
  Benjamin Falk, Carey~E. Priebe, Jason Yim, Randal Burns, Mauro Maggioni, and
  Joshua~T. Vogelstein.
\newblock Sparse projection oblique randomer forests.
\newblock {\em Journal of Machine Learning Research}, 21(104):1--39, 2020.

\bibitem{Trivedietal2014}
Shubhendu Trivedi, Jialei Wang, Samory Kpotufe, and Gregory Shakhnarovich.
\newblock A consistent estimator of the expected gradient outerproduct.
\newblock In {Nevin L.} Zhang and Jin Tian, editors, {\em Uncertainty in
  Artificial Intelligence - Proceedings of the 30th Conference, UAI 2014},
  pages 819--828. AUAI Press, 2014.

\bibitem{JMLR:v11:wu10a}
Qiang Wu, Justin Guinney, Mauro Maggioni, and Sayan Mukherjee.
\newblock Learning gradients: Predictive models that infer geometry and
  statistical dependence.
\newblock {\em Journal of Machine Learning Research}, 11(75):2175--2198, 2010.

\bibitem{Xia2002MAVE}
Y.~Xia, H.~Tong, W.~Li, and L.-X. Zhu.
\newblock An adaptive estimation of dimension reduction space.
\newblock {\em Journal of the Royal Statistical Society: Series B (Statistical
  Methodology)}, 64(3):363–410, 2002.

\bibitem{yu2015useful}
Yi~Yu, Tengyao Wang, and Richard~J Samworth.
\newblock A useful variant of the davis--kahan theorem for statisticians.
\newblock {\em Biometrika}, 102(2):315--323, 2015.

\bibitem{Zhan2022ConsistencyOO}
Haoran Zhan, Yu~Liu, and Yingcun Xia.
\newblock Consistency of oblique decision tree and its boosting and random
  forest.
\newblock {\em Preprint arXiv:2211.12653v3}, 2024.

\end{thebibliography}

\begin{appendix}
%%%%%%%%%%%%%%%%%%%%%%%%%%%%%%%%%%%%%%%%%%%%
% APPENDIX
%%%%%%%%%%%%%%%%%%%%%%%%%%%%%%%%%%%%%%%%%%%%
%\appendix

\section{Proofs}

This appendix contains the remaining proofs of the results in the main text that were not contained in Section \ref{sec:proofs} of the main text. Throughout, we will denote by $B_{\delta} := \{x \in \mathbb{R}^d : \|x\| \leq 1 - \delta\}$ the ball of radius $1 - \delta$ for $\delta \in [0,1)$, and $c_{\mu}$ will denote a constant depending only on the distribution $\mu$ of $X$ that may change from line to line.

\subsection{General risk bound for $\mathcal{C}^{1,\beta}$ functions}\label{a:C2-general-bnd}

\begin{prop}\label{t:C2_Rate}
Assume $f$ satisfies \eqref{e:ridge_fxn_assump} with $\tilde{g} \in \mathcal{C}^{1, \beta}(L)$, and assume $\mu$ has a positive and Lipschitz density on its support $B^d$. Let $\hat{f}_{n} = \hat{f}_{n, M, \lambda, \Pi}$ be the random tessellation forest estimator with normalized associated zonoid $\Pi$, $M$ trees, and lifetime $\lambda > 0$. 
Then, for fixed $\delta \in [0, 1)$,
\begin{align*}
&\EE[(\hat{f}_{n}(X) - f(X))^2| X \in B_{\delta}] \\ &\leq c_{\mu}L^2\Bigg(\frac{1}{\lambda^2}\EE\left[\mathrm{D}(P_SZ_0)\left(1 - \frac{\mathrm{vol}_d(Z_{0} \cap \lambda(B^d - X))}{\mathrm{vol}_d((P_SZ_0 +  P_{S^{\perp}}Z_0) \cap \lambda (B^d - X))}\right)\bigg|X \in B_{\delta}\right]^2 \\
&\quad + \frac{\EE\left[\mathrm{D}(P_SZ_0)^2\right]^2}{\lambda^4}  +  \frac{\EE[\mathrm{D}(P_{S}Z_0)^{1+\beta}]^2}{\lambda^{2 + 2\beta}}  + \frac{\EE\left[\mathrm{D}(P_SZ_0)^21_{\{\mathrm{D}(P_SZ_0) \geq \frac{\lambda}{2}\}}\right]}{\lambda^2}\\
&\quad + \frac{s\mathbb{E}\left[\mathrm{D}(P_SZ_0)^3 1_{\{\mathrm{D}(P_SZ_0) \geq \frac{\lambda \delta}{2}\}}\right]}{\mu(B_{\delta})\lambda^3} \Bigg) + \frac{L^2\EE[\mathrm{D}(P_SZ_{0})^2]}{\lambda^2M} \\
&\quad +  \frac{5\|f\|_{\infty}^2 + 2\sigma^2}{n \mu(B_{\delta})} \left(2s\sum_{k=1}^d \lambda^k\kappa_k  V_1(P_{S^{\perp}}\Pi)^{1 \vee (k-s)} + \sum_{k=0}^{s}\lambda^k \kappa_k V_{k}(P_S\Pi)\right).  
\end{align*}
\end{prop}

The upper bound above is a result of a bias-variance decomposition of the risk of a random tessellation forest estimator, where the last term is similar to the upper bound on the variance as in Theorem \ref{t:Lipschitz_Rate}, and the remaining terms are an upper bound on the bias for the forest estimator that exploits the additional smoothness assumption. This bias upper bound depends more delicately on the geometry of the zero cell and its relation to the relevant feature subspace $S$ than in Theorem \ref{t:Lipschitz_Rate}. This result will be used to obtain the improved rate of convergence for oblique Mondrian forests under additional assumptions in Theorem \ref{t:C2_Transformed_Mondrain_Bound}.

We first need the following lemmas before proceeding to the proof of Proposition \ref{t:C2_Rate}.

\begin{lemma}\label{l:PsZ_odd_fux}
For $\lambda > 0$ and an $s$-dimensional linear subspace $S$ of $\RR^d$, define the probability density
$$F_{\lambda,S}(y) := \EE\left[\frac{1_{\{y \in P_SZ_x^{\lambda}\}}}{\mathrm{vol}_s(P_SZ_x^{\lambda})}\right], \qquad y \in S.$$
Then,
\[\int_{S} (y - P_Sx) F_{\lambda,S}(y) \dint y  = 0.\]
\end{lemma}
\begin{proof}
By stationary of $\mathcal{P}(\lambda)$,
    \begin{align*}
        \int_{S} (y - P_Sx) F_{\lambda,S}(y) \dint y  &= \int_{S}  (y - P_Sx) \EE\left[\frac{1_{\{y \in P_SZ_x^{\lambda}\}}}{\mathrm{vol}_s(P_SZ_x^{\lambda})}\right] \dint y \\
        &= \int_{S}  (y - P_Sx) \EE\left[\frac{1_{\{y - P_Sx \in P_S(Z_x^{\lambda} - x)\}}}{\mathrm{vol}_s(P_S(Z_x^{\lambda} - x))}\right] \dint y \\
        &= \int_{S}  \omega \EE\left[\frac{1_{\{\omega \in P_SZ_0^{\lambda}\}}}{\mathrm{vol}_s(P_SZ_0^{\lambda})}\right] \dint y.
    \end{align*}
The conclusion will follow from the fact that $Z_0^{\lambda}$ has the distribution as $-Z_0^{\lambda}$. Indeed, the distribution of a random convex polytope is uniquely defined by the containment function $C_K := \PP(K \subset \cdot)$ (Theorem 1.8.9 in \cite{MolchanovBook}).
Since mixed volumes are invariant under reflections, we have that for all compact $K \subset \RR^d$ containing the origin,
\begin{align*}
\PP(K \subset -Z_0^{\lambda}) = \PP(- K \subset Z_0^{\lambda}) = e^{-2dV_1(-K, \mathcal{B}_{\lambda})} =  e^{-2dV_1(K, \mathcal{B}_{\lambda})} = \PP(K \subset Z_0^{\lambda}),
\end{align*}
where $\mathcal{B}_{\lambda}$ is the the Blaschke body of $\mathcal{P}(\lambda)$ (see \cite[p. 162]{weil}).
We thus have that 
$$\EE\left[\frac{1_{\{\omega \in P_SZ_0^{\lambda}\}}}{\mathrm{vol}_s(P_SZ_0^{\lambda})}\right]  = \EE\left[\frac{1_{\{-\omega \in P_SZ_0^{\lambda}\}}}{\mathrm{vol}_s(P_SZ_0^{\lambda})}\right],$$
which implies the integrand above is odd and the integral is zero. 
\end{proof}

\begin{lemma}\label{l:diam_orthgonal_columns}
    Suppose that $Z_0$ is the zero cell of a STIT tessellation with unit lifetime and directional distribution $\phi_A$ as in \eqref{e:model_dirdist} for nonsingular $A \in \RR^{d \times d}$ with orthogonal columns and $\|A\|_{2, 1} = 1$. Denote by $\sigma_s(B)$ the $s$-th largest singular value of a matrix $B$. Let $\hat{S}$ be the $s$-dimensional linear subspace of $\RR^d$ spanned by the eigenvectors of $A$ with the largest $s$ eigenvalues and let $S$ be an arbitrary $s$-dimensional linear subspace with $\sigma_s(P_SA) > 0$. Then, for all $t \geq 0$ and $k > 0$,
\begin{align*}
\EE\left[\mathrm{D}(P_{\hat{S}}Z_0)^k1_{\{\mathrm{D}(P_{\hat{S}}Z_{0}) \geq t\}}\right] &\leq \frac{\Gamma(2s +k)}{\Gamma(2s)}\sum_{n=0}^{2s + k - 1}\frac{t^n\sigma_{s}(P_{S}A)^{n-k}}{n!} e^{-t \sigma_{r}(P_{S}A)}.
\end{align*}
In particular, for all $p > 0$,
\begin{align*}
    \EE[\mathrm{D}(P_{\hat{S}}Z_0)^p] \leq \frac{\Gamma(2s + p)}{2^p\sigma_{s}(P_{S}A)^p\Gamma(2s)}.
\end{align*}
\end{lemma}

\begin{proof}
First observe that if $A$ has orthogonal columns, then 
\begin{align*}
    A^{-1} = \mathrm{diag}(\|a_1\|^{-2}, \ldots \|a_d\|^{-2}) \cdot A^T,
\end{align*}
and thus by the equality \eqref{e:wZ0} in the proof of Lemma \ref{l:Deter},
\begin{align*}
    \omega_{Z_0}(u) = \sum_{i=1}^d (T_i^{(1)} + T_i^{(2)}) |\langle u, \frac{a_i}{\|a_i\|_2^2} \rangle|.
\end{align*}
Now let $\hat{I}_r \subseteq \{1, \ldots, d\}$ be the index subset of size $r$ corresponding to the columns of $A$ with the $r$ largest norms (which are also the eigenvalues), and let $\hat{S}$ be the subspace spanned by those $r$ columns. The diameter of $P_{\hat{S}}A$ is then
\begin{align*}
    \sup_{u \in \mathbb{S}^{d-1}} \sum_{i=1}^d (T_i^{(1)} + T_i^{(2)}) |\langle P_{\hat{S}}u, \frac{a_i}{\|a_i\|_2^2} \rangle| &= \left\|\sum_{i \in \hat{I}_r} (T_i^{(1)} + T_i^{(2)}) \frac{a_i}{\|a_i\|_2^2} \right\|_2 \leq \sum_{i \in \hat{I}_r} \frac{(T_i^{(1)} + T_i^{(2)})}{\|a_i\|_2} \\
    &= \frac{1}{\sigma_s(P_{\hat{S}}A)}\sum_{i \in \hat{I}_r} (T_i^{(1)} + T_i^{(2)}).
\end{align*}
%Now, for an arbitrary linear subspace $W$, denote by $I_r$ the index set corresponding to the $a_i$'s such that $\|P_Wa_i\|_2$ are one of the $r$ largest. By the definition of $\hat{I}_r$, we have
% Then, we see that
% \begin{align*}
% \min_{i \in \hat{I}_r} \|a_i\|_2 = \min_{i \in I_r} \|P_{\hat{S}}a_i\|_2 \geq \sigma_s(P_{\hat{S}}A). %\geq \min_{i \in I_r} \|a_i\|_2 \geq \min_{i \in I_r} \|P_Wa_i\|_2 \geq \sigma_r(P_WA).
% \end{align*}
Note that $T^{(s)} := \sum_{i \in \hat{I}_s} (T_i^{(1)} + T_i^{(2)}) \sim \mathrm{Erlang}(2s, 1)$, and $\sigma_s(P_{\hat{S}}A) \geq \sigma_s(P_SA) > 0$ for any subspace $S$ as defined in the Lemma. Thus, for $r > 0$ and $k \in \mathbb{N}$,
\begin{align*}
\EE\left[\mathrm{D}(P_{\hat{S}}Z_0)^k1_{\{\mathrm{D}(P_{\hat{S}}Z_{0}) \geq r\}}\right] & \leq \EE\left[\left(\frac{T^{(s)}} {\sigma_s(P_{S}A)}\right)^k1_{\{T^{(s)} \geq r \sigma_s(P_{S}A)\}}\right] \\
&= %\frac{1}{\sigma_s(P_SA)^k}\EE\left[(T^{(s)})^k \mathbf{1}_{\{T^{(s)} \geq r \sigma_{s}(P_SA)\}}\right] \\
 %&= 
 \frac{\Gamma(2s +k)}{\sigma_s(P_{S}A)^k\Gamma(2s)}\sum_{n=0}^{2s + k - 1}\frac{\left( r\sigma_{s}(P_{S}A) \right)^{n}}{n!} e^{-r\sigma_{s}(P_{S}A)},
\end{align*}
and moments of the diameter of $P_{\hat{S}}Z_0$ for any $p > 0$ satisfy the upper bound
\begin{align*}
\EE[\mathrm{D}(P_{\hat{S}}Z_0)^p] &\leq %\frac{\EE[(T^{(s)})^p]}{\sigma_s(P_SA)^p} = 
\frac{\Gamma(2s + p)}{\sigma_{s}(P_{\hat{S}}A)^p\Gamma(2s)}.
\end{align*}
\end{proof}

We now proceed to the proof of Proposition \ref{t:C2_Rate}.

\begin{proof}[Proof of Proposition \ref{t:C2_Rate}]
Recall the definition \eqref{e:forest} of a random tessellation forest estimator $\hat{f}_{n, \lambda, M}$ built from $M$ random tessellation trees of lifetime $\lambda > 0$. Define for each $m$ and $x \in \RR^d$, 
\begin{align*}
\bar{f}^{(m)}_{\lambda}(x) := \EE[f(X)|X \in Z^{\lambda,(m)}_x ],
\end{align*}
where $Z^{\lambda,(m)}_x$ is the cell of the $m$-th random tessellation $\mathcal{P}_m(\lambda)$ containing $x \in \RR^d$ and define the average $\bar{f}_{\lambda,M}(x) := \frac{1}{M}\sum_{m=1}^M \bar{f}^{(m)}_{\lambda}(x)$. Also define \begin{align}\label{e:tilde-f}
\tilde{f}_{\lambda}(x) := \EE_{\mathcal{P}}[\bar{f}_{\lambda,M}(x)] = \EE_{\mathcal{P}}[\bar{f}^{(1)}_{\lambda,M}(x)].
\end{align}
As noted in \cite{mourtada2020minimax}, the bias-variance decomposition for the risk of a tree estimator can be extended to the random forest estimator as follows \cite[Equation (1)]{arlot2014analysis}:
\begin{align}\label{e:bias-var_forests}
     \EE[(\hat{f}_{\lambda,n, M}(X) - f(X))^2] =  \EE[(f(X)- \bar{f}_{\lambda,M}(X))^2] + \EE[(\bar{f}_{\lambda,M}(X) - \hat{f}_{\lambda,n,M}(X) )^2].
\end{align}
\textit{Variance term}: For the variance term in \eqref{e:bias-var_forests}, Jensen's inequality implies
\begin{align*}
  \EE[(\bar{f}_{\lambda,M}(x) - \hat{f}_{\lambda,n,M}(x))^2] \leq \EE[(\bar{f}_{\lambda}^{(1)}(x) - \hat{f}_{\lambda,n,1}(x))^2].
\end{align*}
We then use Lemma \ref{l:general_variance} to obtain the upper bound
\begin{align*}
&\EE[(\bar{f}_{\lambda}^{(1)}(X) - \hat{f}_{\lambda,n,1}(X) )^2] \\
&\leq \frac{5\|f\|_{\infty}^2 + 2\sigma^2}{n} \left(2s\sum_{k=1}^d \lambda^k\kappa_k  V_1(P_{S^{\perp}}\Pi)^{1 \vee (k-s)} + \sum_{k=0}^{s}\lambda^k \kappa_k V_{k}(P_S\Pi)\right),
\end{align*}
and the conditional variance satisfies
\begin{align}\label{e:cond_var2}
  &\EE[(\bar{f}_{\lambda}^{(1)}(X) - \hat{f}_{\lambda,n,1}(X) )^2| X \in B_{\delta}] \nonumber \\
  &\leq \PP(X \in K_{\delta})^{-1} \EE[(\bar{f}_{\lambda}^{(1)}(X) - \hat{f}_{\lambda,n,1}(X) )^2] \nonumber \\
  & \leq \frac{5\|f\|^2_{\infty} + 2\sigma^2}{n \mu(B_{\delta})} \left(2s\sum_{k=1}^d \lambda^k\kappa_k  V_1(P_{S^{\perp}}\Pi)^{1 \vee (k-s)} + \sum_{k=0}^{s}\lambda^k \kappa_k V_{k}(P_S\Pi)\right).
\end{align}
\textit{Bias term}:
For the bias term in \eqref{e:bias-var_forests}, Proposition 1 of \cite{arlot2014analysis} implies that for fixed $x \in \RR^d$,
\begin{align}\label{e:C2_bias_decomp}
  \EE_{\mathcal{P}}[(f(x)- \bar{f}_{\lambda,M}(x) )^2] = \EE_{\mathcal{P}}[(f(x) - \tilde{f}_{\lambda}(x))^2] + \frac{\mathrm{Var}_{\mathcal{P}}(\bar{f}^{(1)}_{\lambda}(x))}{M}.  
\end{align}
We then have the following upper bound on the variance of $\bar{f}^{(1)}_{\lambda}$: for $x \in \RR^d$,
\begin{align*}
 \mathrm{Var}_{\mathcal{P}}(\bar{f}^{(1)}_{\lambda}(x)) \leq \EE_{\mathcal{P}}\left[(\bar{f}^{(1)}_{\lambda}(x) - f(x))^2\right] \leq \frac{L^2}{\lambda^2}\EE[\mathrm{D}(P_SZ_{0})^2],
\end{align*}
where the last inequality follows from Lemma \ref{l:fixed_x_bias_bnd} and stationarity. 
It thus remains to control the remaining term $\EE_{\mathcal{P}}[(f(x) - \tilde{f}_{\lambda}(x))^2]$. By Taylor's theorem, for $f \in \mathcal{C}^{1,\beta}(L)$ with $\beta \in (0,1]$,
\begin{align*}
    |f(z) - f(x) - \nabla f(x)^T(z-x)| &= |g(P_{S}z) - g(P_{S}x) - \nabla g(P_{S}x)^TP_{S}(z-x)| \\
    &= \left| \int_0^1 [\nabla g(P_{S}x + tP_{S}(z-x)) - \nabla g(P_{S}x)]^TP_{S}(z-x)\dint t\right| \\
    &\leq \int_0^1 L(t\|P_{S}(z-x)\|)^{\beta}\|P_{S}(z-x)\| \dint t \leq L\|P_{S}(z-x)\|^{1 + \beta}.
\end{align*}
Then, for $x \in \RR^d$,
\begin{align}\label{e:pointwise-ftildebias-bnd}
    |\tilde{f}_{\lambda}(x) - f(x)| &= \left| \EE\left[\frac{1}{\mu(Z_x^{\lambda})}\int_{Z_x^{\lambda}} (f(z) - f(x)) \mu(\dint z) \right]\right| \nonumber \\
    &\leq \left|\EE\left[ \frac{1}{\mu(Z_x^{\lambda})}\int_{Z_x^{\lambda}} \nabla f(x)^T(z - x) \mu(\dint z)\right]\right| \\
    &\quad + \EE\left[ \frac{1}{\mu(Z_x^{\lambda})}\int_{Z_x^{\lambda}} \left|f(z) - f(x) - \nabla f(x)^T(z - x) \right| \mu(\dint z)\right] \nonumber \\
    &\leq \left|\nabla f(x)^T \int_{\RR^d} (z - x) \EE\left[\frac{1_{\{z \in Z_x^{\lambda}\}}}{\mu(Z_x^{\lambda})}\right] \mu(\dint z)\right| \nonumber \\
    &\quad +  \EE\left[\frac{L}{\mu(Z_x^{\lambda})}\int_{\RR^d} \|P_{S}(z-x)\|^{1 + \beta}1_{\{z \in Z_x^{\lambda}\}} \mu(\dint z)\right] \nonumber \\
    &\leq \left|\nabla g(P_{S}x)^T \int_{\RR^d} P_{S}(z - x) \EE\left[\frac{1_{\{z \in Z_x^{\lambda}\}}}{\mu(Z_x^{\lambda})}\right] \mu(\dint z)\right| \nonumber \\
    &\quad +  \EE\left[\frac{L\mathrm{D}(P_SZ_x^{\lambda})^{1 + \beta}}{\mu(Z_x^{\lambda})}\int_{\RR^d}1_{\{z \in Z_x^{\lambda}\}} \mu(\dint z)\right] \nonumber \\
    &\leq \|\nabla g(P_{S}x)\| \left\|\int_{\RR^d} P_{S}(z - x) \EE\left[\frac{1_{\{z \in Z_x^{\lambda}\}}}{\mu(Z_x^{\lambda})}\right] \mu(\dint z)\right\| +  L \EE\left[\mathrm{D}(P_SZ_x^{\lambda})^{1 + \beta}\right] \nonumber \\
    &\leq L\left\| \int_{\RR^d} P_{S}(z - x) \EE\left[\frac{1_{\{z \in Z_x^{\lambda}\}}}{\mu(Z_x^{\lambda})}\right]\mu(\dint z)\right\| +  \frac{L}{\lambda^{1 + \beta}}\EE[\mathrm{D}(P_{S}Z_0)^{1+\beta}].
\end{align}
By the assumptions, the density $p$ of $\mu$ has a finite Lipschitz constant $C_p > 0$  on its support $B^d := \mathrm{supp}(\mu)$ and we can define $p_0 := \min_{x \in B^d} p(x) > 0$ and $p_1 := \max_{x \in B^d} p(x) < \infty$. Also note that the integrand above is zero when $z,y \notin B^d$. %In the following we denote by $K^c := \RR^d \backslash K$ the complement of $K$.
Then, for the first term above,
\begin{align*}
\left\| \int_{\RR^d} P_{S}(z - x) \EE\left[\frac{1_{\{z \in Z_x^{\lambda}\}}}{\mu(Z_x^{\lambda})}\right]\mu(\dint z)\right\| &= \left\| \int_{\RR^d} P_{S}(z - x) \EE\left[\frac{p(z) 1_{\{z \in Z_x^{\lambda} \cap B^d\}}}{\mu(Z_x^{\lambda})}\right] \dint z\right\|.
\end{align*}
Now, define $\tilde{Z}_x^{\lambda} := P_SZ_x^{\lambda} + P_{S^{\perp}}Z_x^{\lambda}$. We first compare the density $F_{\lambda, p}(z) := \EE\left[\frac{p(z) 1_{\{z \in Z_x^{\lambda} \cap K\}}}{\mu(Z_x^{\lambda})}\right]$ with
the density
\begin{align*}
\tilde{F}_{\lambda, p, S} (z) := \EE\left[\frac{p(z)1_{\{z \in \tilde{Z}_x^{\lambda} \}}}{\mu(\tilde{Z}_x^{\lambda})}\right], \quad z \in \RR^d.  
\end{align*}

By the triangle inequality,
\begin{align}\label{e:split_term}
 &\left\|\int_{\RR^d} P_{S}(z - x) F_{\lambda, p}(z)\mu(\dint z) \right\| \nonumber \\ &\leq \underbrace{\left\|\int_{\RR^d} P_{S}(z - x) \left(F_{\lambda, p}(z) - \tilde{F}_{\lambda, p, S}(z)\right) \dint z \right\|}_{I}  +   \underbrace{\left\|\int_{\RR^d} P_{S}(z - x)\tilde{F}_{\lambda, p, S}(z) \dint z \right\|}_{II}.  
\end{align}
\textit{Bound on term $I$.} To handle the first term above, we see that
\begin{align*}
    I 
    &\leq \EE\left[\int_{\RR^d} \left\|P_{S}(z - x)\right\| \left|\frac{1_{\{z \in Z_x^{\lambda}\}}}{\mu(Z_x^{\lambda})} - \frac{1_{\{z \in \tilde{Z}_x^{\lambda}  \}}}{\mu(\tilde{Z}_x^{\lambda} )}\right|p(z) \dint z \right] \\
    &\leq \EE\bigg[\frac{\mathrm{D}(P_SZ_x^{\lambda})}{\mu(Z_x^{\lambda})\mu(\tilde{Z}_x^{\lambda})} \int_{\RR^d} \left|\mu(\tilde{Z}_x^{\lambda} )1_{\{z \in Z_x^{\lambda}\}}- \mu(Z_x^{\lambda})1_{\{z \in \tilde{Z}_x^{\lambda} \}} \right|p(z) \dint z \bigg] \\
    &\leq \EE\bigg[\frac{\mathrm{D}(P_SZ_x^{\lambda})}{\mu(Z_x^{\lambda})\mu(\tilde{Z}_x^{\lambda} )} \int_{\RR^d} \int_{\RR^d}p(z)p(y) \left|1_{\{y \in \tilde{Z}_x^{\lambda} \}}1_{\{z \in Z_x^{\lambda}\}}- 1_{\{y \in Z_x^{\lambda}\}} 1_{\{z \in \tilde{Z}_x^{\lambda} \}} \right| \dint y\dint z \bigg].
\end{align*}
Then, we see that by symmetry
\begin{align*}
  &\int_{\RR^d} \int_{\RR^d}p(z)p(y) \left|1_{\{y \in \tilde{Z}_x^{\lambda} \}}1_{\{z \in Z_x^{\lambda}\}}- 1_{\{y \in Z_x^{\lambda}\}} 1_{\{z \in \tilde{Z}_x^{\lambda} \}} \right| \dint y\dint z \\
  &\leq 2\int_{\RR^d}\int_{\RR^d} p(y)p(z) 1_{\{z \in Z_x^{\lambda}\}}1_{\{y \in \tilde{Z}_x^{\lambda} \}}1_{\{y \notin Z_x^{\lambda}\}} \dint z \dint y \\
  &\leq %2p_1\mu(Z^{\lambda}_x) \mathrm{vol}_d(\tilde{Z}_x^{\lambda}  \cap (Z_x^{\lambda})^c \cap B^d) \\
  %&= 
  2p_1\mu(Z^{\lambda}_x)\left( \mathrm{vol}_d(\tilde{Z}_x^{\lambda}  \cap B^d) - \mathrm{vol}_d(Z_x^{\lambda}  \cap B^d)\right).
\end{align*}
Also note that
$\mu(\tilde{Z}_x^{\lambda} ) = \int_{B^d} p(z) 1_{\{z \in \tilde{Z}_x^{\lambda} \}} \dint y \geq p_0 \mathrm{vol}_d(\tilde{Z}_x^{\lambda}  \cap B^d)$.
Combining the above bounds gives 
\begin{align*}
    I 
    &\leq \frac{p_1}{p_0}\EE\left[\mathrm{D}(P_SZ_x^{\lambda})\left(1 - \frac{\mathrm{vol}_d(Z_{x}^{\lambda} \cap B^d)}{\mathrm{vol}_d(\tilde{Z}_x^{\lambda}  \cap B^d)}\right) \right].
\end{align*}
By stationarity and the scaling property \eqref{e:scaling}, we thus have 
\begin{align}\label{e:termIbnd}
    I \leq %\frac{p_1}{p_0}\EE\left[\mathrm{D}(P_SZ_x^{\lambda})\left(1 - \frac{\mathrm{vol}_d(Z_{x}^{\lambda} \cap K)}{\mathrm{vol}_d(\tilde{Z}_x^{\lambda}  \cap K)}\right) \right] &= 
    \frac{p_1}{\lambda p_0} \EE\left[\mathrm{D}(P_SZ_0)\left(1 - \frac{\mathrm{vol}_d(Z_0\cap \lambda(B^d - x))}{\mathrm{vol}_d((P_SZ_0 + P_{S^{\perp}}Z_0) \cap \lambda(B^d-x))}\right) \right]. 
\end{align}

\textit{Bound on term $II$.} For the second term in \eqref{e:split_term}
we compare the marginal of $\tilde{F}_{\lambda, p, S}$  with the density
$$\tilde{F}_{\lambda, S}(y) := \EE\left[\frac{1_{\{y \in P_SZ_x^{\lambda}\}}}{\mathrm{vol}_s(P_SZ_x^{\lambda})}\right], \qquad y \in S.$$ 
By Lemma \ref{l:PsZ_odd_fux},
\begin{align*}
 \int_{S}(y - P_{S}x) F_{\lambda, S}(y) \dint y  = 0,
\end{align*}
and thus, letting $z = y + \omega$ for $y \in S$ and $\omega \in S^{\perp}$, we have
\begin{align*}
II &= \left\| \int_{S}\int_{S^{\perp}} (y - P_{S}x) \EE\left[\frac{p(y + \omega) 1_{\{(y, \omega) \in \tilde{Z}_x^{\lambda}\}}}{\mu(\tilde{Z}_x^{\lambda})}\right] \dint y\dint \omega\right\| \\
&= \left\| \int_{S}(y - P_{S}x) \EE\left[ \frac{1}{\mu( \tilde{Z}_x^{\lambda})}\int_{S^{\perp}}p(y +\omega)1_{\{y + \omega \in \tilde{Z}_x^{\lambda}\}} \dint \omega \right] \dint y\right\| \\
&= \left\| \int_{S}(y - P_{S}x) \left(\EE\left[\frac{1}{\mu(\tilde{Z}_x^{\lambda})}\int_{S^{\perp}}p(y + \omega)1_{\{y + \omega \in \tilde{Z}_x^{\lambda}\}} \dint \omega\right]   - F_{\lambda,S}(y) \right)\dint y\right\|  \\
&\leq  \EE\left[ \int_{S}  \|y - P_{S}x \|  \left|\frac{1}{\mu(\tilde{Z}_x^{\lambda})} \int_{S^{\perp}}p(y + \omega)1_{\{y + \omega \in \tilde{Z}_x^{\lambda}\}} \dint \omega  - \frac{1_{\{y \in P_SZ_{x}^{\lambda}\}}}{\mathrm{vol}_s(P_SZ_x^{\lambda}) } \right| \dint y \right] \\
&\leq  \EE\left[ \mathrm{D}(P_SZ_x^{\lambda})\int_{S} \left|\frac{1}{\mu(\tilde{Z}_x^{\lambda})} \int_{S^{\perp}}p(y + \omega)1_{\{y + \omega \in \tilde{Z}_x^{\lambda} \}} \dint \omega  - \frac{1_{\{y \in P_SZ_{x}^{\lambda}\}}}{\mathrm{vol}_s(P_SZ_x^{\lambda}) } \right| \dint y \right].
\end{align*}
Next we see that the integral inside the expectation satisfies
\begin{align*}
    &\int_S \left|\frac{1}{\mu(\tilde{Z}_x^{\lambda})} \int_{S^{\perp}}p(y + \omega)1_{\{y + \omega \in \tilde{Z}_x^{\lambda}\}} \dint \omega  - \frac{1_{\{y \in P_SZ_{x}^{\lambda}\}}}{\mathrm{vol}_s(P_SZ_x^{\lambda}) } \right| \dint y \\
    &\leq  \frac{\int_{S} \int_{S} \int_{S^{\perp}} \left| p(y, \omega) - p(z,\omega)\right|1_{\{\omega \in P_{S^{\perp}}Z_x^{\lambda}\}}1_{\{y \in P_SZ_x^{\lambda}\}} 1_{\{z\in P_SZ^{\lambda}_{x}\}}   \dint \omega \dint z \dint y}{p_0\mathrm{vol}_d((P_SZ_x^{\lambda} + P_{S^{\perp}}Z_x^{\lambda}) \cap B^d)\mathrm{vol}_s(P_SZ_x^{\lambda})},
\end{align*}
and
\begin{align*}
&\left| p(y + \omega) - p(z + \omega)\right| \\
&= \left| p(y + \omega) - p(z + \omega)\right|1_{\{y + \omega \in B^d, z + \omega \in B^d\}} + |p(y + \omega)|1_{\{y + \omega \in B^d, z + \omega \notin B^d\}} \\
&\quad + |p(z + \omega)|1_{\{y + \omega \notin B^d, z + \omega \in B^d\}} \\
&\leq C_p\|y-z\|_2 1_{\{y + \omega \in B^d, z + \omega \in B^d\}} + p_1 1_{\{y + \omega \in B^d, z + \omega \notin B^d\}}  + p_1 1_{\{y + \omega \notin B^d, z + \omega \in B^d\}}.
\end{align*}
Then the integral in the numerator above satisfies
\begin{align*}
    &\int_{S} \int_{S} \int_{S^{\perp}} \left| p(y, \omega) - p(z,\omega)\right|1_{\{\omega \in P_{S^{\perp}}Z_x^{\lambda}\}}1_{\{y \in P_SZ_x^{\lambda}\}} 1_{\{z\in P_SZ^{\lambda}_{x}\}} \dint \omega \dint z \dint y\\
    &\leq \int_{S} \int_{S} \int_{S^{\perp}} C_p\|y-z\|_2 1_{\{(y, \omega) \in (P_{S}Z_x^{\lambda} + P_{S^{\perp}}Z_x^{\lambda}) \cap B^d\}} 1_{\{(z, \omega) \in (P_{S}Z_x^{\lambda})\}} \dint \omega \dint z \dint y \\
    &+ 2p_1 \int_{S} \int_{S} \int_{S^{\perp}} 1_{\{(y, \omega) \in (P_{S}Z_x^{\lambda} + P_{S^{\perp}}Z_x^{\lambda}) \cap B^d\}} 1_{\{z\in P_SZ^{\lambda}_{x}\}} 1_{\{(z, \omega) \notin (P_{S}Z_x^{\lambda}  + P_{S^{\perp}}Z_x^{\lambda}) \cap B^d\}}  \dint \omega \dint z \dint y \\
    &\leq C_p\mathrm{D}(P_SZ_x^{\lambda})\mathrm{vol}_d((P_SZ_x^{\lambda} + P_{S^{\perp}}Z_x^{\lambda}) \cap B^d)\mathrm{vol}_s(P_SZ_x^{\lambda}) \\
    & + 2p_1 \int_{S} \int_{S} \int_{S^{\perp}} 1_{\{(y, \omega) \in (P_{S}Z_x^{\lambda} + P_{S^{\perp}}Z_x^{\lambda}) \cap B^d\}} 1_{\{z\in P_SZ^{\lambda}_{x}\}} 1_{\{(z, \omega) \notin (P_{S}Z_x^{\lambda}  + P_{S^{\perp}}Z_x^{\lambda}) \cap B^d\}}  \dint \omega \dint z \dint y.
\end{align*}
and thus, 
\begin{align*}
II  &\leq \frac{C_p}{p_0} \EE\left[\mathrm{D}(P_SZ_x^{\lambda})^2 \right] \\
    &+  \frac{2p_1}{p_0}\EE\left[\mathrm{D}(P_SZ_x^{\lambda}) \int_S\int_S \int_{S^{\perp}} \frac{1_{\{(y,\omega) \in (P_{S}Z_x^{\lambda} + P_{S^{\perp}}Z_x^{\lambda}) \cap B^d\}} 1_{\{(z, \omega) \notin K\}}1_{\{z \in P_SZ^{\lambda}_{x} \}} }{\mathrm{vol}_d((P_SZ_x^{\lambda} + P_{S^{\perp}}Z_x^{\lambda}) \cap B^d)\mathrm{vol}_s(P_SZ_x^{\lambda})} \dint \omega \dint z \dint y\right]. 
    %&=  \frac{C_p}{p_0} \EE\left[ \mathrm{D}(P_SZ_x^{\lambda})^2\right] +  \frac{2p_1}{p_0}\EE\left[ \mathrm{D}(P_SZ_x^{\lambda}) \frac{\mathrm{vol}_s(P_SZ^{\lambda}_{x} \cap K_S^c)}{\mathrm{vol}_s(P_SZ_x^{\lambda})} \right],
\end{align*}
To proceed with simplifying the second term above, we first define $x_S := P_Sx$, $B_S := P_SB^d$, and $B_{S^{\perp}} := P_{S^{\perp}}B^d$. Then, note that $(1 - \|x_S\|)B^d + x_S \subseteq B^d$, and thus the event that $z + \omega \notin B^d$ implies that $z + \omega \notin (1 - \|x_S\|) B^d  + x_S$. For any linear subspace $S$, $\frac{1}{2}B_S + \frac{1}{2}B_{S^{\perp}} \subseteq B^d$, and thus $(z, \omega) \notin B^d$ also implies $z \notin \frac{1}{2}(1 - \|x_S\|)B_S + x_S$. Thus,
\begin{align*}
&\int_S\int_S \int_{S^{\perp}} \frac{1_{\{(y,\omega) \in (P_{S}Z_x^{\lambda} + P_{S^{\perp}}Z_x^{\lambda}) \cap K\}} 1_{\{(z, \omega) \notin K\}}1_{\{z \in P_SZ^{\lambda}_{x} \}} }{\mathrm{vol}_d((P_SZ_x^{\lambda} + P_{S^{\perp}}Z_x^{\lambda}) \cap K)\mathrm{vol}_s(P_SZ_x^{\lambda})} \dint \omega \dint z \dint y \\
&\leq \int_S\int_S \int_{S^{\perp}} \frac{1_{\{(y,\omega) \in (P_{S}Z_x^{\lambda} + P_{S^{\perp}}Z_x^{\lambda}) \cap K\}} 1_{\{z \in P_SZ^{\lambda}_{x} \cap (\frac{1}{2}(1 - \|x_S\|)B_S + x_S)^c \}} }{\mathrm{vol}_d((P_SZ_x^{\lambda} + P_{S^{\perp}}Z_x^{\lambda}) \cap K)\mathrm{vol}_s(P_SZ_x^{\lambda})} \dint \omega \dint z \dint y  \\
&= \frac{\mathrm{vol}_s(P_SZ^{\lambda}_{x} \cap (\frac{1}{2}(1 - \|x_S\|)B_S + x_S)^c)}{\mathrm{vol}_s(P_SZ_x^{\lambda})} \\
&= \frac{\mathrm{vol}_s(P_SZ_0 \cap \left(\frac{\lambda}{2}(1 - \|x_S\|)B_S\right)^c)}{\mathrm{vol}_s(P_SZ_0)},
\end{align*}
where the last equality follows from the scaling property \eqref{e:scaling} and stationarity, which imply $P_SZ_x^{\lambda} = \frac{1}{\lambda}P_SZ_0 + x_S$. Finally, observe that the ratio above is zero when $\mathrm{D}(P_SZ_0) \leq \frac{\lambda}{2}(1 - \|x_S\|)$, and thus
\begin{align}\label{e:final_II_bnd}
II &\leq 
\frac{C_p}{\lambda^2p_0}\EE\left[\mathrm{D}(P_SZ_0)^2\right] + \frac{2p_1}{\lambda p_0}\EE\left[\mathrm{D}(P_SZ_0)1_{\{\mathrm{D}(P_SZ_0) \geq \frac{\lambda}{2}(1 - \|x_S\|)\}}\right].
\end{align}

\textit{Final Bound.}
Combining the upper bounds on I and II gives
\begin{align}\label{e:f_ftilde_bnd}
&|f(x) - \tilde{f}_{\lambda}(x)| \nonumber \\
&\leq L\bigg(\frac{p_1}{\lambda p_0} \EE\left[\mathrm{D}(P_SZ_0)\left(1 - \frac{\mathrm{vol}_d(Z_0\cap \lambda(B^d - x))}{\mathrm{vol}_d((P_SZ_0 + P_{S^{\perp}}Z_0) \cap \lambda(B^d-x))}\right) \right]  + \frac{C_p\EE\left[\mathrm{D}(P_SZ_0)^2\right]}{\lambda^2p_0} \nonumber \\
&\qquad \quad  + \frac{2p_1}{\lambda p_0}\EE\left[\mathrm{D}(P_SZ_0)1_{\{\mathrm{D}(P_SZ_0) \geq \frac{\lambda}{2}(1 - \|x_S\|)\}}\right] +  \frac{\EE[\mathrm{D}(P_{S}Z_0)^{1+\beta}]}{\lambda^{1 + \beta}}\bigg).
\end{align}
Taking the expectation with respect to $X$ conditioned on $X \in K_{\delta}$ and  applying the inequality $(x + y)^2 \leq 2x^2 + 2y^2$ and Jensen's inequality gives
\begin{align*}
 &\EE[(f(X) - \tilde{f}_{\lambda}(X))^2| X \in B_{\delta}] \\
 &\leq L^2\EE_X\bigg[\bigg(\frac{p_1}{\lambda p_0} \EE\left[\mathrm{D}(P_SZ_0)\left(1 - \frac{\mathrm{vol}_d(Z_0\cap \lambda(B^d - X))}{\mathrm{vol}_d((P_SZ_0 + P_{S^{\perp}}Z_0) \cap \lambda(B^d-X))}\right) \right] \\
 & \quad + \frac{C_p\EE\left[\mathrm{D}(P_SZ_0)^2\right]}{\lambda^2p_0} +  \frac{\EE[\mathrm{D}(P_{S}Z_0)^{1+\beta}]}{\lambda^{1 + \beta}} \\
&\quad + \frac{2p_1}{\lambda p_0}\EE\left[\mathrm{D}(P_SZ_0)1_{\{\mathrm{D}(P_SZ_0) \geq \frac{\lambda}{2}(1 - \|P_SX\|)\}}\right]%\frac{\eoadd{2}p_1}{\lambda p_0}\EE\left[\frac{\mathrm{D}(P_SZ_0)\mathrm{vol}_s(P_SZ_0 \cap \lambda(K_S^c - X_S))}{\mathrm{vol}_s(P_SZ_0)}\right]  
\bigg)^2 \bigg| X \in B_{\delta} \bigg] \\
&\leq 4L^2\bigg(\frac{p_1^2}{\lambda^2p_0^2}\EE\left[\mathrm{D}(P_SZ_0)\left(1 - \frac{\mathrm{vol}_d(Z_{0} \cap \lambda(B^d - X))}{\mathrm{vol}_d((P_SZ_0 +  P_{S^{\perp}}Z_0) \cap \lambda (B^d - X))}\right)\bigg|X \in B_{\delta}\right]^2 \\
&\quad + \frac{C_p^2\EE\left[\mathrm{D}(P_SZ_0)^2\right]^2}{\lambda^4p_0^2}   +  \frac{\EE[\mathrm{D}(P_{S}Z_0)^{1+\beta}]^2}{\lambda^{2 + 2\beta}}  \\
&\quad + \frac{4p_1^2\EE\left[\mathrm{D}(P_SZ_0)^2\mathbb{P}_X(\mathrm{D}(P_SZ_0) \geq \frac{\lambda}{2}(1 - \|P_SX\|) \, | \, X \in B_{\delta})\right]}{\lambda^2p_0^2}\bigg). 
\end{align*}

We will now take a closer look at the last term in the parentheses above. We first note that conditioning on $X \in B_{\delta}$, i.e. $\|X\| \leq 1 - \delta$, the event $\mathrm{D}(P_SZ_0) \geq \frac{\lambda}{2}(1 - \|P_SX\|)$ implies $\mathrm{D}(P_SZ_0) \geq \frac{\lambda \delta}{2}$. Also, conditioned on $\mathrm{D}(P_SZ_0) \geq \frac{\lambda}{2}$, the event $\mathrm{D}(P_SZ_0) \geq \frac{\lambda}{2}(1 - \|P_SX\|)$ has probability $1$ with respect to $\mu$. Also, conditioned on the event $\mathrm{D}(P_SZ_0) < \frac{\lambda}{2}$,
\begin{align}\label{e:bnd-diam-tail}
&\PP_X\left(\mathrm{D}(P_SZ_0) \geq \frac{\lambda}{2}(1 - \|P_SX\|) \, \bigg| \, X \in B_\delta \right) \nonumber\\%&= \PP_X\left(\|X\| \geq 1 - \frac{2}{\lambda}\mathrm{D}(P_SZ_0) \right) \nonumber \\ 
&= \frac{1}{\mu(B_{\delta})}\int_{B^d} p(x) 1_{\{\|P_Sx\| \geq 1 - \frac{2}{\lambda}\mathrm{D}(P_SZ_0), \|x\| \leq 1 - \delta \}} \dint x \nonumber \\
%&\leq p_1 \left(\mathrm{vol}(B^d) - \mathrm{vol}\left(\left(1 - \frac{2}{\lambda}\mathrm{D}(P_SZ_0) \right)B^d\right)\right) \nonumber \\
&\leq \frac{p_1}{\mu(B_{\delta})} 1_{\{\mathrm{D}(P_SZ_0) \geq \frac{\lambda \delta}{2}\}} \int_{B_S + B_{S^{\perp}}} 1_{\{\|P_Sx\| \geq 1 - \frac{2}{\lambda}\mathrm{D}(P_SZ_0) \}} \dint x \nonumber\\
&= \frac{p_1\kappa_{d-s} \kappa_s}{\mu(B_{\delta})} \left( 1 - \left(1 - \frac{2}{\lambda}\mathrm{D}(P_SZ_0) \right)^s\right)1_{\{\mathrm{D}(P_SZ_0) \geq \frac{\lambda \delta}{2}\}}  \nonumber \\
&\leq \frac{2s p_1}{\mu(B_{\delta})} \cdot \frac{1}{\lambda}\mathrm{D}(P_SZ_0) 1_{\{\mathrm{D}(P_SZ_0) \geq \frac{\lambda \delta}{2}\}},
%p_1\kappa_{d-s}\kappa_s \sum_{k=1}^{s} \binom{s}{k} \frac{2^{k}D(P_SZ_0)^{k}}{\lambda^{k}}.
\end{align}
where we have used the inequality $(1-x)^s \leq 1 - sx$ and the bound $\kappa_k \leq 1$ for all $k \in \mathbb{N}$.

The complete upper bound on the risk conditioned on $X \in B_{\delta}$ is then
\begin{align*}
&\EE[(\hat{f}_{\lambda,n, M}(X) - f(X))^2| X \in B_{\delta}] \\ 
&\leq c_{\mu}L^2\Bigg(\frac{1}{\lambda^2}\EE\left[\mathrm{D}(P_SZ_0)\left(1 - \frac{\mathrm{vol}_d(Z_{0} \cap \lambda(B^d - X))}{\mathrm{vol}_d((P_SZ_0 +  P_{S^{\perp}}Z_0) \cap \lambda (B^d - X))}\right)\bigg|X \in B_{\delta}\right]^2 \\
&\quad + \frac{\EE\left[\mathrm{D}(P_SZ_0)^2\right]^2}{\lambda^4}  +  \frac{\EE[\mathrm{D}(P_{S}Z_0)^{1+\beta}]^2}{\lambda^{2 + 2\beta}}  + \frac{\EE\left[\mathrm{D}(P_SZ_0)^21_{\{\mathrm{D}(P_SZ_0) \geq \frac{\lambda}{2}\}}\right]}{\lambda^2} \\
& \quad + \frac{s}{\mu(B_{\delta})}\cdot\frac{\mathbb{E}\left[\mathrm{D}(P_SZ_0)^3 1_{\{\mathrm{D}(P_SZ_0) \geq \frac{\lambda \delta}{2}\}}\right]}{\lambda^3} \Bigg)  + \frac{L^2\EE[\mathrm{D}(P_SZ_{0})^2]}{\lambda^2M} \\
& \quad +  \frac{5\|f\|_{\infty}^2 + 2\sigma^2}{n \mu(B_{\delta})} \left(2s\sum_{k=1}^d \lambda^k\kappa_k  V_1(P_{S^{\perp}}\Pi)^{1 \vee (k-s)} + \sum_{k=0}^{s}\lambda^k \kappa_k V_{k}(P_S\Pi)\right). 
\end{align*}

\end{proof}

%%%%%%%%%%%%%%%%%%%%%%%%%%%%%%%%%%%%%%%%%%%%%%%%%%%%%%%%%%%%%%%%%%%%%%%%%
\subsection{Proof of Theorem \ref{t:C2_Transformed_Mondrain_Bound} and Corollary \ref{cor:C2_Transformed_Mondrain_Rate}}

\begin{proof}[Proof of Theorem \ref{t:C2_Transformed_Mondrain_Bound}]

We proceed as in the proof of Proposition \ref{t:C2_Rate} to obtain the following upper bound on the risk using \eqref{e:bias-var_forests} and \eqref{e:f_ftilde_bnd}
\begin{align*}
    &\EE[(\hat{f}_{\lambda,n, M}(x) - f(x))^2] \\
    &=  \EE[(f(x) - \tilde{f}_{\lambda}(x))^2] + \frac{L^2}{\lambda^2M}\mathbb{E}[\mathrm{D}(P_SZ_0)^2] + \EE[(\bar{f}_{\lambda,M}(x) - \hat{f}_{\lambda,n,M}(x) )^2],
\end{align*}
where $\tilde{f}_{\lambda}$ is as defined in \eqref{e:tilde-f} in the proof of Proposition \ref{t:C2_Rate}. We will now take a more careful treatment of the first bias term $\EE_{\mathcal{P}}[(f(x) - \tilde{f}_{\lambda}(x))^2]$ under the assumption in Theorem \ref{t:C2_Transformed_Mondrain_Bound}.

Recall that $f(x) = g(P_Sx)$ for an $s$-dimensional subspace $S$ and we have assumed that the STIT tessellation used to build the estimator $\hat{f}_n$ has directional distribution \eqref{e:model_dirdist} defined by a matrix $A$ with orthogonal columns. Then, define $\hat{S}$ to be the $s$-dimensional subspace spanned by the columns of $A$ with the $s$ largest norms, i.e. the principal eigenspace corresponding to the top $s$ eigenvectors.

Recall from the proof of Theorem \ref{t:C2_Rate} the bound, for $x \in \RR^d$,
\begin{align*}
    |\tilde{f}_{\lambda}(x) - f(x)| 
    &\leq L\left\| \int_{\RR^d} P_{S}(z - x) \EE\left[\frac{1_{\{z \in Z_x^{\lambda}\}}}{\mu(Z_x^{\lambda})}\right]\mu(\dint z)\right\| +  \frac{L}{\lambda^{1 + \beta}}\EE[\mathrm{D}(P_{S}Z_0)^{1+\beta}].
\end{align*}
Here we divert the proof to divide the first term in the bound above as follows:
\begin{align*}
&\left\| \int_{\RR^d} P_{S}(z - x) \EE\left[\frac{1_{\{z \in Z_x^{\lambda}\}}}{\mu(Z_x^{\lambda})}\right]\mu(\dint z)\right\| \\
&\leq
\left\| \int_{\RR^d} (P_{S}(z - x) - P_{\hat{S}}(z-x)) \EE\left[\frac{1_{\{z \in Z_x^{\lambda}\}}}{\mu(Z_x^{\lambda})}\right]\mu(\dint z)\right\|  + \left\| \int_{\RR^d} P_{\hat{S}}(z - x) \EE\left[\frac{1_{\{z \in Z_x^{\lambda}\}}}{\mu(Z_x^{\lambda})}\right]\mu(\dint z)\right\| \\
&\leq \|P_S - P_{\hat{S}}\|\left\| \EE\int_{\RR^d} (P_{S} + P_{\hat{S}})(z - x)\frac{1_{\{z \in Z_x^{\lambda}\}}}{\mu(Z_x^{\lambda})}\mu(\dint z)\right\|+ \left\| \int_{\RR^d} P_{\hat{S}}(z - x) \EE\left[\frac{1_{\{z \in Z_x^{\lambda}\}}}{\mu(Z_x^{\lambda})}\right]\mu(\dint z)\right\| \\
&\leq \frac{\|P_S - P_{\hat{S}}\|(\EE[\mathrm{D}(P_{S}Z_0)] + \EE[\mathrm{D}(P_{\hat{S}}Z_0)])}{\lambda} + \left\| \int_{\RR^d} P_{\hat{S}}(z - x) \EE\left[\frac{1_{\{z \in Z_x^{\lambda}\}}}{\mu(Z_x^{\lambda})}\right]\mu(\dint z)\right\|.
\end{align*}
Now observe that by the definition of $\hat{S}$ and assumption that $A$ has orthogonal columns, the first term in the upper bound above is zero because $Z_0 = P_{\hat{S}}Z_0 + P_{\hat{S}^{\perp}}Z_0$. Then, following the arguments in the proof of Proposition \ref{t:C2_Rate}, term $I$ in \eqref{e:split_term} is zero, and by the bound \eqref{e:final_II_bnd} on term $II$ in \eqref{e:split_term}, we have
\begin{align*}
&\left\| \int_{\RR^d} P_{\hat{S}}(z - x) \EE\left[\frac{1_{\{z \in Z_x^{\lambda}\}}}{\mu(Z_x^{\lambda})}\right]\mu(\dint z)\right\| \\
& \leq L\left(\frac{C_p\EE\left[\mathrm{D}(P_{\hat{S}}Z_0)^2\right]}{\lambda^2p_0} + \frac{2p_1}{\lambda p_0}\EE\left[\mathrm{D}(P_{\hat{S}}Z_0)1_{\{\mathrm{D}(P_{\hat{S}}Z_0) \geq \frac{\lambda}{2}(1 - \|x_{\hat{S}}\|)\}}\right]\right).
\end{align*}
%Now define the function $h(x) := g(P_{\hat{S}}x)$ and
%%\begin{align*}
 %     \tilde{h}_{\lambda}(x) &:= \EE_{\mathcal{P}}\left[\frac{1}{\mu(Z_x^{\lambda})} \int_{Z_x^{\lambda}} h(z)\mu(dz) \right], 
%\end{align*}
%where $\mathcal{P}$ is the random STIT tessellation used to build $\hat{f}$ and $Z_x^{\lambda}$ is the cell of the tessellation at lifetime $\lambda$ containing $x \in \mathrm{supp}(\mu)$.
%Using the inequality $(x + y)^2 \leq 2x^2 + 2y^2$, we then have 
%\begin{align}
%    (f(x) - \tilde{f}_{\lambda}(x))^2 
%    &\leq 2(f(x) - h(x))^2  + 4(h(x) - \tilde{h}_{\lambda}(x))^2 + 4 (\tilde{h}_{\lambda}(x) - \tilde{f}_{\lambda}(x))^2.
%\end{align}
%\paragraph{Bound on $(h(x) - \tilde{h}_{\lambda}(x))^2$.}We can bound the middle term $(h(x) - \tilde{h}_{\lambda}(x))^2$ in the upper bound above using the same steps in the proof of Theorem \ref{t:C2_Rate} to obtain the bound \eqref{e:f_ftilde_bnd} except with $\hat{S}$ in place of $S$. 
Thus,
\begin{align*}
    |f(x) - \tilde{f}_{\lambda}(x)| &\leq \frac{L\|P_S - P_{\hat{S}}\|(\EE[\mathrm{D}(P_{S}Z_0)] + \EE[\mathrm{D}(P_{\hat{S}}Z_0)])}{\lambda} + \frac{LC_p\EE\left[\mathrm{D}(P_{\hat{S}}Z_0)^2\right]}{\lambda^2p_0} \\
    &\qquad + \frac{2Lp_1}{\lambda p_0}\EE\left[\mathrm{D}(P_{\hat{S}}Z_0)1_{\{\mathrm{D}(P_{\hat{S}}Z_0) \geq \frac{\lambda}{2}(1 - \|x_{\hat{S}}\|)\}}\right] +  \frac{L\EE[\mathrm{D}(P_{S}Z_0)^{1+\beta}]}{\lambda^{1 + \beta}},
\end{align*}
and again following the same argument as in the proof of Proposition \ref{t:C2_Rate} gives that conditioned on $X \in B_{\delta}$ for $\delta \in [0,1)$,
\begin{align*}
&\EE[(f(X) - \tilde{f}_{\lambda}(X))^2| X \in B_{\delta}] \\
&\leq c_{\mu}L^2 \bigg(\frac{\|P_S - P_{\hat{S}}\|^2(\EE[\mathrm{D}(P_{S}Z_0)] + \EE[\mathrm{D}(P_{\hat{S}}Z_0)])^2}{\lambda} + \frac{\EE\left[\mathrm{D}(P_{\hat{S}}Z_0)^2\right]^2}{\lambda^4} \\
    &\qquad + \frac{\EE\left[\mathrm{D}(P_{\hat{S}}Z_0)^21_{\{\mathrm{D}(P_{\hat{S}}Z_0) \geq \frac{\lambda}{2}\}}\right]}{\lambda^2} + \frac{s\mathbb{E}\left[\mathrm{D}(P_{\hat{S}}Z_0)^3 1_{\{\mathrm{D}(P_{\hat{S}}Z_0) \geq \frac{\lambda \delta}{2}\}}\right]}{\mu(B_{\delta})\lambda^3}  +  \frac{\EE[\mathrm{D}(P_{S}Z_0)^{1+\beta}]^2}{\lambda^{2 + 2\beta}}\bigg),
\end{align*}
%and for $\delta = 0$, 
%\begin{align*}
%\EE[(f(X) - \tilde{f}_{\lambda}(X))^2]
%&\leq\frac{L\|P_S - P_{\hat{S}}\|\EE[\mathrm{D}(P_{S + \hat{S}}Z_0)]}{\lambda} + \frac{LC_p\EE\left[\mathrm{D}(P_{\hat{S}}Z_0)^2\right]}{\lambda^2p_0} \\
%    &\qquad +  \frac{L\EE[\mathrm{D}(P_{S}Z_0)^{1+\beta}]}{\lambda^{1 + \beta}}  + \frac{\kappa_{d-s}\kappa_s \EE\left[\mathrm{D}(P_SZ_0)^2\right]}{\lambda^2} \sum_{k=1}^{s} \binom{s}{k} \frac{2^k \EE\left[D(P_SZ_0)^{k}\right]}{\lambda^{k}}\bigg).
%\end{align*}

Next, by Lemma \ref{l:Deter} and Gautschi's inequality, for $\delta \geq 0$ and $\beta \in [0, 1]$,
\begin{align*}
\EE\left[\mathrm{D}(P_{S}Z_0)^{1 + \beta}\right] &\leq \frac{\Gamma(2d + 1 + \beta)}{ \sigma_s(P_SA)^{1 + \beta}\Gamma(2d)} \leq \frac{\Gamma(2d + 1)(2d+1)^{\beta}}{ \sigma_s(P_SA)^{1 + \beta}\Gamma(2d)} \\
&= \frac{2d(2d+1)^{\beta}}{ \sigma_s(P_SA)^{1 + \beta}} \leq \frac{6d^{1 + \beta}}{\sigma_s(P_SA)^{1 + \beta}} ,
\end{align*}
and for $k \in \mathbb{N}$ and $\delta \in [0,1]$, by Lemma \ref{l:diam_orthgonal_columns},
\begin{align*}
\EE\left[\mathrm{D}(P_{\hat{S}}Z_0)^k1_{\{\mathrm{D}(P_{\hat{S}}Z_{0}) \geq \frac{\lambda \delta}{2}\}}\right] &\leq \frac{\Gamma(2s + k)}{\sigma_s(P_{S}A)^k \Gamma(2s)}\sum_{n=0}^{2s + k - 1}\frac{\lambda^n \delta^n \sigma_{s}(P_{S}A)^{n}}{2^{n} n!} e^{-\frac{\lambda \delta \sigma_{s}(P_{S}A)}{2}}.
%\frac{\Gamma(2s +k)}{\Gamma(2s)}\sum_{n=0}^{2s + k - 1}\frac{r^n\sigma_{s}(P_SA)^{n-k}}{n!} e^{-r \sigma_{s}(P_SA)}
\end{align*}
%and
%\begin{align*}
%    \EE[\mathrm{D}(P_{\hat{S}}Z_0)^k] \leq \frac{\Gamma(2s + k)}{2^k\sigma_{s}(P_SA)^k\Gamma(2s)}.
%\end{align*}
%Now note that $\sigma_s(P_{\hat{S}} A) \geq \sigma_s(P_S A) > 0$ where the second inequality is by assumption. 
By the Davis-Kahan Theorem \cite{yu2015useful},
\begin{align*}
    \|P_{\hat{S}} - P_{S}\|_{F} &= 2\|\sin \Theta(\hat{S},S)\|_F \leq \frac{\|A - P_S A\|_F}{\sigma_s(P_SA)} = \frac{\|P_{S^{\perp}}A\|_F}{\sigma_s(P_SA)}.
\end{align*}

Combining the bounds above gives
\begin{align*}
    \EE[(f(X) - \tilde{f}_{\lambda}(X))^2 |X \in B_{\delta}] &\leq c_{\mu}L^2\bigg[\frac{d^2\|P_{S^{\perp}}A\|^2_F}{\lambda^2\sigma_s(P_SA)^4} + \frac{s^4}{\lambda^4 \sigma_s(P_SA)^{4}} +  \frac{d^{2\beta + 2}}{\lambda^{2 + 2\beta} \sigma_s(P_SA)^{2 + 2\beta}}  \\
    & + \frac{s^2}{\lambda^2\sigma_s(P_SA)^2}\sum_{n=0}^{2s + 1}\frac{\lambda^n \sigma_{s}(P_SA)^{n}}{2^{n} n!} e^{-\frac{\lambda \sigma_{s}(P_SA)}{2}} \\
    &+ \frac{s^4}{\lambda^3\sigma_s(P_SA)^3\mu(B_{\delta})}\sum_{n=0}^{2s + 2}\frac{\lambda^n \delta^n \sigma_{s}(P_SA)^{n}}{2^{n} n!} e^{-\frac{\lambda \delta \sigma_{s}(P_SA)}{2}}\bigg],
    %\frac{se^{-\lambda \delta \sigma_{s}(P_SA)}}{\sigma_s(P_SA)^2\lambda^2} \left(\sum_{n=0}^{2s}\frac{\lambda^n \delta^n \sigma_{s}(P_SA)^{n}}{2^{n} n!} \right)^2\right]
\end{align*}
and for $\delta = 0$,
\begin{align*}
    \EE[(f(X) - \tilde{f}_{\lambda}(X))^2] &\leq c_{\mu}L^2\bigg[\frac{d^2\|P_{S^{\perp}}A\|^2_F}{\lambda^2\sigma_s(P_SA)^4} + \frac{s^4}{\lambda^4 \sigma_s(P_SA)^{4}} +  \frac{d^{2\beta + 2}}{\lambda^{2 + 2\beta} \sigma_s(P_SA)^{2 + 2\beta}} \\
    & + \frac{s^3}{\lambda^3\sigma_s(P_SA)^3} + \frac{s^2}{\lambda^2\sigma_s(P_SA)^2}\sum_{n=0}^{2s + 1}\frac{\lambda^n \sigma_{s}(P_SA)^{n}}{2^{n} n!} e^{-\frac{\lambda \sigma_{s}(P_SA)}{2}}\bigg].
\end{align*}
Finally, we use the same bound on the variance term \eqref{e:cond_var2} in the proof of Theorem \ref{t:C2_Rate} and recall from equation \eqref{e:P_Sperp_zonoid_diam} in the proof of Theorem \ref{t:C1_Transformed_Mondrain_Bound} that in this setting we have
\begin{align*}
    V_1(P_{S^{\perp}}\Pi) = \|P_{S^{\perp}}A\|_{2,1}.
\end{align*}
We also use the fact that $V_k(K) \leq \frac{V_1(K)^k}{k!}$ for any convex body $K$ and $\|P_{S^{\perp}}A\|_{2,1} \leq 1$ to obtain
\begin{align*}
  &\EE[(\bar{f}_{\lambda,M}(X) - \hat{f}_{\lambda,n,M}(X) )^2| X \in B_{\delta}] \\
  &\leq \frac{5\|f\|^2_{\infty} + 2\sigma^2}{n \mu(B_{\delta})} \left(2s\sum_{k=1}^d \lambda^k\kappa_k  \|P_{S^{\perp}}A\|_{2,1}^{1 \vee (k-s)} + \sum_{k=0}^{s}\frac{\lambda^k \kappa_k}{k!}\right).
\end{align*}

The final bound on the conditional and unconditional risks follows: 
\begin{align*}
    &\EE[(\hat{f}_{\lambda,n, M}(X) - f(X))^2 | X \in B_{\delta}] \\
    &\leq c_{\mu}L^2\bigg[\frac{d^2\|P_{S^{\perp}}A\|^2_F}{\lambda^2\sigma_s(P_SA)^4} + \frac{s^4}{\lambda^4 \sigma_s(P_SA)^{4}} +  \frac{d^{2\beta + 2}}{\lambda^{2 + 2\beta} \sigma_s(P_SA)^{2 + 2\beta}}  \\
    &\quad + \frac{s^2}{\lambda^2\sigma_s(P_SA)^2}\sum_{n=0}^{2s + 1}\frac{\lambda^n \sigma_{s}(P_SA)^{n}}{2^{n} n!} e^{-\frac{\lambda \sigma_{s}(P_SA)}{2}} \\
    &\quad + \frac{s^4}{\lambda^3\sigma_s(P_SA)^3\mu(B_{\delta})}\sum_{n=0}^{2s + 2}\frac{\lambda^n \delta^n \sigma_{s}(P_SA)^{n}}{2^{n} n!} e^{-\frac{\lambda \delta \sigma_{s}(P_SA)}{2}}\bigg] \\
    &\quad  + \frac{6L^2d^2}{\lambda^2M \sigma_s(P_SA)^2} + \frac{5\|f\|^2_{\infty} + 2\sigma^2}{n \mu(B_{\delta})} \left(2s\sum_{k=1}^d \lambda^k\kappa_k  \|P_{S^{\perp}}A\|_{2,1}^{1 \vee (k-s)} + \sum_{k=0}^{s}\frac{\lambda^k \kappa_k}{k!}\right).
\end{align*}

\end{proof}

\begin{proof}[Proof of Corollary \ref{cor:C2_Transformed_Mondrain_Rate}]

For the statement in Corollary \ref{cor:C2_Transformed_Mondrain_Rate}, Theorem \ref{t:C2_Transformed_Mondrain_Bound} and the assumptions imply that as $n \to \infty$,
\begin{align*}
\EE[(\hat{f}_{\lambda_n,n, M_n}(X) - f(X))^2| X \in B_{\delta}] 
&\lesssim \frac{L^2}{\lambda_n^{2 + 2\beta}}  + \frac{L^2\ee_n^2}{\lambda_n^2} + \frac{L^2}{\lambda_n^2M_n} +  \frac{\sum_{k=s}^d\lambda_n^k \ee_n^{k-s}}{n}. % + o(\lambda_n^{-2 - 2\beta}).
\end{align*}
Then additionally assuming $M_n \gtrsim \lambda_n^{2\beta}$,
we have
\begin{align*}
\EE[(\hat{f}_{\lambda_n,n, M_n}(X) - f(X))^2| X \in B_{\delta}] &\lesssim \frac{L^2}{\lambda_n^{2 + 2\beta}} + \frac{\lambda_n^s}{n} + \frac{\sum_{k=s+1}^d\lambda_n^k \ee_n^{k-s}}{n} + \frac{L^2\ee_n^2}{\lambda_n^2}.% + o(\lambda_n^{-2 - 2\beta}).
\end{align*}
Depending on the rate at which $\varepsilon_n$ goes to zero, the asymptotic behavior of this risk upper bound could correspond to three different possible regimes where different terms become dominant: $\varepsilon_n \lesssim \lambda_n^{-1}$, $\lambda_n^{-1} \ll \varepsilon_n \lesssim \lambda_n^{-\beta}$, and $\lambda_n^{-\beta} \ll \varepsilon_n$. We look at each case and find the optimal $\lambda_n$ that balances the two dominant terms to obtain the overall asymptotic behavior: \\
\emph{Case 1}:
If $\ee_n \lesssim (L^{2}n)^{-\frac{1}{s + 2\beta + 2}}$ we have that for $\lambda_n \asymp (L^{2}n)^{\frac{1}{s + 2\beta + 2}}$,
\begin{align*}
\EE[(f(X) - \hat{f}_{\lambda_n,n, M_n}(X))^2 | X \in B_{\delta}] 
&\lesssim L^{\frac{2s}{s + 2\beta + 2}}n^{-\frac{2\beta + 2}{s + 2\beta + 2}}.
\end{align*}
\emph{Case 2}: If $(L^2n)^{-\frac{1}{s + 2 + 2\beta}} \lesssim \varepsilon_n \lesssim (L^2n)^{-\frac{\beta}{d + 2 + 2\beta + \beta(s-d)}}$, then letting $\lambda_n \asymp (L^{2}n \varepsilon_n^{s-d})^{\frac{1}{d + 2\beta + 2}}$,
\begin{align*}
\EE[(f(X) - \hat{f}_{\lambda_n,n,M_n}(X))^2 | X \in B_{\delta}] %\\ 
&\lesssim %\frac{L^{2}}{\left(L^{\frac{2}{d + 2 \beta + 2}}n^{\frac{1}{d + 2\beta + 2}}\ee_n^{-\frac{d-s}{d+2\beta + 2}}\right)^{2 + 2\beta}} + \frac{1}{n}\left(\left(L^{\frac{2}{d + 2 \beta + 2}}n^{\frac{1}{d + 2\beta + 2}}\ee_n^{-\frac{d-s}{d+2\beta + 2}}\right)^{d}\ee_n^{d-s}\right)  + \frac{L^2\ee_n^2}{\left(L^{\frac{2}{d + 2 \beta + 2}}n^{\frac{1}{d + 2\beta + 2}}\ee_n^{-\frac{d-s}{d+2\beta + 2}}\right)^{2}} \\
%&= 
L^2(L^{2}n \varepsilon_n^{s-d})^{\frac{-2 + 2\beta}{d + 2\beta + 2}}
= L^{\frac{2d}{d + 2 \beta + 2}}n^{-\frac{2 + 2\beta}{d + 2\beta + 2}}\ee_n^{\frac{(d-s)(2 + 2\beta)}{d+2\beta + 2}} . % + L^{\frac{2d + 4\beta }{d + 2 + 2\beta}}n^{-\frac{2}{d + 2 + 2\beta}}\ee_n^{2 + \frac{2(d-s)}{d + 2 + 2\beta}}.
\end{align*}
\emph{Case 3}: If $\varepsilon_n \gtrsim (L^2n)^{-\frac{\beta}{d + 2 + 2\beta + \beta(s-d)}}$, then letting $\lambda_n \asymp (L^2n\varepsilon^{s - d + 2}_n)^{\frac{1}{d+2}}$ gives
\begin{align*}
 \EE[(\hat{f}_{\lambda_n,n, M_n}(X) - f(X))^2| X \in B_{\delta}] &\lesssim  %\frac{(nL^2\varepsilon^{s - d + 2}_n)^{\frac{d}{d+2}} \ee_n^{d-s}}{n} + \frac{L^2\ee_n^2}{(nL^2\varepsilon^{s - d + 2}_n)^{\frac{2}{d+2}}} = 
 L^{\frac{2d}{d+2}}n^{-\frac{2}{d+2}}\varepsilon_n^{\frac{4d - 2s}{d+2}}.   
\end{align*}

For $\delta = 0$, the upper bound satisfies
\begin{align*}
&\EE[(\hat{f}_{\lambda,n, M}(X) - f(X))^2]\lesssim  \frac{L^2\ee_n^2}{\lambda_n^2} + \frac{L^2}{\lambda_n^{2 + 2\beta}}  + \frac{L^2}{\lambda_n^3} + \frac{L^2}{\lambda_n^2M} +  \frac{\sum_{k=s}^d \lambda_n^k \ee_n^{k-s} }{n}. 
\end{align*}
If $3 \geq 2 + 2\beta$, then the same rates as above hold. If $3 < 2 + 2\beta$, then
\begin{align*}
&\EE[(\hat{f}_{\lambda,n, M}(X) - f(X))^2]\lesssim \frac{L^2\ee_n^2}{\lambda_n^2} + \frac{L^2}{\lambda_n^3} + \frac{L^2}{\lambda_n^2M} +  \frac{\sum_{k=s}^d \lambda_n^k \ee_n^{k-s}}{n}. 
\end{align*}
Additionally assuming $M_n \gtrsim \lambda_n$ gives
\begin{align*}
&\EE[(\hat{f}_{\lambda,n, M}(X) - f(X))^2]\lesssim \frac{L^2\ee_n^2}{\lambda_n^2} + \frac{L^2}{\lambda_n^3} +  \frac{\lambda^s}{n} + \frac{\sum_{k={s+1}}^d \lambda_n^k \ee_n^{k-s}}{n} . 
\end{align*}
Minimizing the upper bound with respect to $\lambda_n$ similarly gives the following asymptotic behavior divided into three regimes depending on $\varepsilon_n$:\\
\emph{Case 1}:
If $\ee_n \lesssim (L^{2}n)^{-\frac{1}{s + 3}}$ we have that for $\lambda_n \asymp (L^{2}n)^{\frac{1}{s + 3}}$,
\begin{align*}
\EE[(f(X) - \hat{f}_{\lambda_n,n, M_n}(X))^2] 
&\lesssim L^{\frac{2s}{s + 3}}n^{-\frac{3}{s + 3}}.
\end{align*}
\emph{Case 2}: If $(L^2n)^{-\frac{1}{s + 3}} \lesssim \varepsilon_n \lesssim (L^2n)^{-\frac{1}{d + s + 6}}$, then letting $\lambda_n \asymp (L^{2}n \varepsilon_n^{s-d})^{\frac{1}{d + 3}}$,
\begin{align*}
\EE[(f(X) - \hat{f}_{\lambda_n,n,M_n}(X))^2] %\\ 
&\lesssim 
L^2(L^{2}n \varepsilon_n^{s-d})^{\frac{-3}{d + 3}}
= L^{\frac{2d}{d + 3}}n^{-\frac{3}{d + 3}}\ee_n^{\frac{3(d-s)}{d+3}}.
\end{align*}
\emph{Case 3}: If $\varepsilon_n \gtrsim (L^2n)^{-\frac{1}{d + s+ 6 }}$, then letting $\lambda_n \asymp (L^2n\varepsilon^{s - d + 2}_n)^{\frac{1}{d+2}}$ gives
\begin{align*}
 \EE[(\hat{f}_{\lambda_n,n, M_n}(X) - f(X))^2| X \in B_{\delta}] &\lesssim  %\frac{(nL^2\varepsilon^{s - d + 2}_n)^{\frac{d}{d+2}} \ee_n^{d-s}}{n} + \frac{L^2\ee_n^2}{(nL^2\varepsilon^{s - d + 2}_n)^{\frac{2}{d+2}}} = 
 L^{\frac{2d}{d+2}}n^{-\frac{2}{d+2}}\varepsilon_n^{\frac{4d - 2s}{d+2}}.   
\end{align*}

\end{proof}

\subsection{Proofs of Theorem \ref{t:rate1_Mondrian} and \ref{t:C2_Rate_Mondrian}}\label{a:weighted_Mondrian}

We begin with a lemma on the diameter of the projected zero cell of the tessellation generated by an oblique Mondrian process as a special case of Lemma \ref{l:Deter}.
\begin{lemma}\label{l:diam_mondrian}
    Suppose that $Z_0$ is the zero cell of a weighted Mondrian tessellation with unit lifetime and directional distribution as defined in \eqref{e:phi_weighted_mondrian}. Then, for $r \geq 0$ and $k > 0$
\begin{align*}
\EE\left[\mathrm{D}(P_SZ_0)^k1_{\{\mathrm{D}(P_{S}Z_{0}) \geq r\}}\right] &\leq \frac{\Gamma(2s +k)}{\Gamma(2s)}\sum_{n=0}^{2s + k - 1}\frac{r^n \omega_S^n}{n!}e^{-r \omega_S},
\end{align*}
where $\omega_S := \min_{i \in S} \omega_i$. In particular,
\begin{align*}
    \EE[\mathrm{D}(P_SZ_0)^k] \leq \frac{\Gamma(2s + k)}{\omega_S^k\Gamma(2s)}.
\end{align*}
\end{lemma}

\begin{proof}
Recall that $Z_0$ has the same distribution as the Minkowksi sum of the line segments 
$$\omega_i^{-1}[-T_1^{(i)} e_i, T_2^{(i)} e_i], \text{ for }i = 1, \ldots, d,$$
where $T^{(i)}_j$ are i.i.d. exponential random variables with unit parameter. The diameter of $P_SZ_0$ then has the following upper bound:
\begin{align*}
    \mathrm{D}(P_SZ_0) &= \left(\sum_{i\in S} \omega_i^{-2}\left(T^{(1)}_{i} + T^{(2)}_{i}\right)^2\right)^{1/2} \leq \sum_{i\in S} \omega_i^{-1}\left(T^{(1)}_{i} + T^{(2)}_{i}\right) \leq \omega_S^{-1}\sum_{i\in S}\left(T^{(1)}_{i} + T^{(2)}_{i}\right),
\end{align*}
where $\omega_S := \min_{i \in S} \omega_i$. That is, the diameter of $P_SZ_0$ is controlled by the sum of exponential random variables, which is an Erlang distributed random variable $$T^{S}:= \sum_{i\in S} \left(T^{(1)}_{i} + T^{(2)}_{i}\right) \sim \mathrm{Erlang}\left(2s, 1\right).$$
Thus, for $r \geq 0$ and $k > 0$,
\begin{align*}
\EE\left[\mathrm{D}(P_SZ_0)^k1_{\{\mathrm{D}(P_{S}Z_{0}) \geq r\}}\right] &\leq \omega_S^{-k}\EE\left[(T^{S})^k \mathbf{1}_{\{T^S \geq r \omega_S \}}\right] = \frac{\Gamma(2s +k)}{\Gamma(2s)}\sum_{n=0}^{2s + k - 1}\frac{r^n \omega_S^n}{n!}e^{-r \omega_S},
\end{align*}
and moments of the diameter satisfy
\begin{align*}
\EE[\mathrm{D}(P_SZ_0)^k] &\leq \frac{\EE[(T^{S})^k]}{\omega_S^k} = \frac{\Gamma(2s + k)}{\omega_S^k\Gamma(2s)}.
\end{align*}

\end{proof}

\begin{proof}[Proof of Theorem \ref{t:rate1_Mondrian}]
Under the assumptions of the theorem, by the bias-variance decomposition \eqref{e:bias-var}, Lemma \ref{l:fixed_x_bias_bnd} and Lemma 20 in \cite{OReillyTran2021minimax}, we have the following upper bound on the risk of the weighted Mondrian tree estimator $\hat{f}_n$:
\begin{align*}
    \EE[(f(X) - \hat{f}_{n}(X))^2] &= \EE[(f_{\lambda}(X) - \bar{f}_{\lambda}(X))^2] + \EE[(\bar{f}(X) - \hat{f}_{\lambda,n}(X))^2] \\
    &\leq \frac{L^2}{\lambda^2}\mathbb{E}[\mathrm{D}(P_SZ_0)^2] +  \frac{5\|f\|_{\infty}^2 + 2 \sigma^2}{n}\EE[N_{\lambda}([0,1]^d)].
\end{align*}
By Lemma \ref{l:diam_mondrian}, we also have the upper bound
\begin{align*}
\EE[\mathrm{D}(P_SZ_0)^2] \leq \frac{6}{\omega_S^2}.
\end{align*}

We next bound the expectation in the variance upper bound. Let $Z_{\lambda}$ be the typical cell of a STIT with directional distribution \eqref{e:phi_weighted_mondrian} and lifetime $\lambda$ as defined in \eqref{e:campbell}. 
Then, the support function of the typical cell $Z := Z_1$ is given by 
\[h(Z, u) = \frac{1}{2}\sum_{i=1}^d T_i |\langle u, e_i\rangle|,\] where $T_1, \ldots T_d$ are independent and $T_i \sim \exp(\omega_i)$. By the formula for mixed volumes of a zonoid from \cite[p.~614]{weil},
\[V(W[k], Z[d-k]) \stackrel{d}{=} \frac{1}{\binom{d}{d-k}}\sum_{i_1, \ldots, i_{d-1}}^{\neq} \prod_{j=1}^{d-k} T_{i_j},\]
and $\EE[V(W[k], Z[d-k])] = \frac{1}{\binom{d}{d-k}}\sum_{i_1, \ldots, i_{d-k}}^{\neq} \prod_{j=1}^{d-k} \frac{1}{\omega_{i_j}}$. 
Thus, by Lemma 6 in \cite{OReillyTran2021minimax},
\begin{align*}
N_{\lambda}([0,1]^d) &= \mathrm{vol}_d(\Pi_n) \sum_{k=0}^d\lambda^k \sum_{i_1, \ldots, i_{d-k}}^{\neq} \prod_{j=1}^{d-k} \frac{1}{\omega_{i_j}} =\mathrm{vol}_d(\Pi_n) \sum_{k=0}^d\lambda^d \sum_{i_1, \ldots, i_{d-k}}^{\neq} \prod_{j=1}^{d-k} \frac{1}{\lambda\omega_{i_j}} \\
&= \mathrm{vol}_d(\Pi)\lambda^d\prod_{i=1}^d \left(\frac{1}{\lambda\omega_i} + 1\right).
\end{align*}
Using the fact that the associated zonoid for the weighted Mondrian is the hyperrectangle \begin{align}\label{e:model_zonoid_mondrian}
    \Pi = \oplus_{i=1}^d \frac{\omega_i}{2} [-1,1],
\end{align}
we see that $\mathrm{vol}_d(\Pi) = \prod_{i=1}^d \omega_i$,
and thus,
\begin{align*}
 N_{\lambda}([0,1]^d) &= \prod_{i=1}^d \lambda \omega_i \prod_{i=1}^d \left(\frac{1}{\lambda\omega_i} + 1\right) = \prod_{i =1}^d\left(1 + \lambda \omega_i\right).  
\end{align*}
Combining the above observations gives the final bound
\begin{align*}
    \EE[(f(X) - \hat{f}_{n}(X))^2] &\leq \frac{6L^2}{\lambda^2\omega_S^2} +  \frac{5\|f\|_{\infty}^2 + 2 \sigma^2}{n}\prod_{i =1}^d\left(1 + \lambda \omega_i\right).
\end{align*}
\end{proof}

\begin{proof}[Proof of Theorem \ref{t:C2_Rate_Mondrian}]

Note that under the definition of the directional distribution for a weighted Mondrian, the associated zonoid is the hyperrectangle \eqref{e:model_zonoid_mondrian}, and thus we are in the setting where $\Pi = \Pi_S + \Pi_{S^{\perp}}$ for $\Pi_S \subset S$ and $\Pi_{S^{\perp}}\subset S^{\perp}$. Then, from the proof of Theorem \ref{t:C2_Rate}, we have the following upper bound on the risk for a weighted Mondrian forest $\hat{f}_{\lambda, n, M}$ with $M$ trees, lifetime $\lambda$, and directional distribution 
\eqref{e:phi_weighted_mondrian}:

\begin{align*}
&\EE[(\hat{f}_{n}(X) - f(X))^2| X \in [\delta, 1 - \delta]^d] \leq \bigg(\frac{LC_p\EE\left[\mathrm{D}(P_SZ_0)^2\right]}{\lambda^2p_0} +  \frac{L\EE[\mathrm{D}(P_{S}Z_0)^{2}]}{\lambda^{2}}\bigg)^2  \\
&+ \frac{9L^2p^3_1}{\lambda^2 p^3_0(1 - 2\delta)^d}\sum_{j=0}^{s-1} \frac{\kappa_{s-j} V_j([0, 1]^d)}{\lambda^{s-j}}\EE\left[\mathrm{D}(P_SZ_{0})^{s - j + 2}1_{\{\mathrm{D}(P_SZ_{0}) \geq \lambda \delta\}}\right] \\
&+ \bigg(\frac{6L^2C_pp_1\EE\left[\mathrm{D}(P_SZ_0)^2\right]}{\lambda^4p^2_0} +  \frac{6L^2p_1\EE[\mathrm{D}(P_{S}Z_0)^{1+\beta}]}{\lambda^{3 + \beta}p_0} \bigg)\\
& \qquad \cdot \frac{p_1}{p_0(1 - 2\delta)^d}\sum_{j=0}^{s-1} \frac{\kappa_{s-j} V_j([0,1]^s)}{\lambda^{s-j}}\EE\left[\mathrm{D}(P_SZ_{0})^{s - j + 1}1_{\{\mathrm{D}(P_SZ_{0}) \geq \lambda \delta\}}\right] \\
 &+ \frac{6L^2s}{\lambda^2M \omega_S^2} +  \frac{5\|f\|_{\infty}^2 + 2\sigma^2}{n} \EE[N_{\lambda}([0,1]^d)]. 
\end{align*}
By Lemma \ref{l:diam_mondrian} and \eqref{e:vk_cube},
\begin{align*}
&\EE[(\hat{f}_{\lambda,n, M}(X) - f(X))^2| X \in [\delta, 1 - \delta]^d] \leq \bigg(\frac{LC_p\Gamma(2s + 2)}{\lambda^2p_0\omega_S^2\Gamma(2s)} +  \frac{L\Gamma(2s + 1 + \beta)}{\lambda^{1 + \beta}\omega_S^{1 + \beta}\Gamma(2s)}\bigg)^2  \\
&+ \frac{9L^2p^3_1}{\lambda^2 p^3_0(1 - 2\delta)^d}\sum_{j=0}^{s-1} \binom{s}{j}\frac{\kappa_{s-j}\Gamma(2s + s - j + 2)}{\lambda^{s-j}\Gamma(2s)}\sum_{\ell=0}^{2s + (s - j + 2) - 1}\frac{\lambda^{\ell} \ee^{\ell} \omega_S^{\ell}}{\ell!}e^{-\lambda \ee \omega_S} \\
&+ \bigg(\frac{6L^2C_pp_1\Gamma(2s + 2)}{\lambda^4p^2_0\omega_S^2\Gamma(2s)} +  \frac{6L^2p_1\Gamma(2s + 1 + \beta)}{\lambda^{3 + \beta}p_0\omega_S^{1 + \beta}\Gamma(2s)} \bigg)\\
& \qquad \cdot \frac{p_1}{p_0(1 - 2\delta)^d}\sum_{j=0}^{s-1}\binom{s}{j} \frac{\kappa_{s-j}\Gamma(2s + s - j + 1)}{\lambda^{s-j}\Gamma(2s)}\sum_{\ell=0}^{2s + (s - j + 1) - 1}\frac{\lambda^{\ell} \ee^{\ell} \omega_S^{\ell}}{\ell!}e^{-\lambda \ee \omega_S}  \\
 &+ \frac{6L^2s}{\lambda^2M \omega_S^2} +  \frac{5\|f\|_{\infty}^2 + 2\sigma^2}{n} \prod_{i = 1}^d\left(1 + \lambda \omega_i\right). 
\end{align*}
Thus, for $\delta > 0$,
\begin{align*}
\EE[(\hat{f}_{n}(X) - f(X))^2| X \in [\delta, 1 - \delta]^d] &\leq 4s^2(2s + 1)^2\bigg(\frac{C_p}{p_0} + 1 \bigg)^2\frac{L^2}{\lambda^{4}\omega_S^{4}}  \\
 & + \frac{6L^2s}{\lambda^2M \omega_S^2} +  \frac{5\|f\|_{\infty}^2 + 2\sigma^2}{n} \prod_{i = 1}^d\left(1 + \lambda \omega_i\right) + o(\lambda^{-4}).    
\end{align*}
and for $\delta = 0$,
\begin{align*}
\EE[(\hat{f}_{\lambda,n, M}(X) - f(X))^2] &\leq 4s^2(2s + 1)^2\bigg(\frac{C_p}{p_0} + 1 \bigg)^2\frac{L^2}{\lambda^{4}\omega_S^{4}}  + \frac{18L^2p^3_1s\Gamma(2s + 3)}{\lambda^3 p^3_0\Gamma(2s)} \\
&+ \frac{6L^2s}{\lambda^2M \omega_S^2} +  \frac{5\|f\|_{\infty}^2 + 2\sigma^2}{n} \prod_{i = 1}^d\left(1 + \lambda \omega_i\right) + o(\lambda^{-3}). 
\end{align*}
\end{proof}

\subsection{Proof of Theorem \ref{thm:Mondrian_suboptimal}}\label{a:suboptimal}

\begin{proof}
Recall the bias-variance decomposition \eqref{e:bias-var} of a weighted Mondrian tree estimator $\hat{f}_{n}$ with lifetime $\lambda$:
\begin{align*}
    \EE[(f(X) - \hat{f}_{n}(X))^2] &= \EE[(f_{\lambda}(X) - \bar{f}_{\lambda}(X))^2] + \EE[(\bar{f}(X) - \hat{f}_{\lambda,n}(X))^2].
\end{align*}
First we obtain a lower bound on the bias. Recall that the distribution of the cell $Z_x^{\lambda}$ of a weighted Mondrian tessellation with lifetime $\lambda$ and directional distribution \eqref{e:phi_weighted_mondrian} containing $x \in \RR^d$ is the hyperrectangle
\[ \prod_{i=1}^d \left[x_i - T_i^{(1)}, x_i + T_i^{(2)}\right],\]
where for each $i = 1, \ldots, d$, $T^{(1)}_{i}$ and $T_i^{(2)}$ are independent exponential random variables with parameter $\lambda \omega_i$. Then, under the assumptions in the theorem,
\begin{align*}
\bar{f}_{\lambda}(x) - f(x) &= \frac{1}{\mu(Z_x^{\lambda})}\int_{\RR^d} f(y) - f(x) \dint \mu(y) \\
&= \frac{1}{\mathrm{vol}_d(Z_x^{\lambda} \cap [0,1]^d)}\int_{Z_x^{\lambda} \cap [0,1]^d} \langle a, y - x \rangle \dint y  \\
&= \sum_{i=1}^d \frac{a_i }{|[x_i - T_i^{(1)}, x_i + T_i^{(2)}] \cap [0,1]|}\int_{[x_i - T_i^{(1)}, x_i + T_i^{(2)}] \cap [0,1]} (y_i - x_i) \dint y_i \\
&\overset{(d)}{=} \sum_{i=1}^d\frac{a_i}{|[- T_i^{(1)}, T_i^{(2)}] \cap [-x_i, 1 -x_i ]|}\int_{ [- T_i^{(1)}, T_i^{(2)}] \cap [-x_i,1 - x_i]} t \dint t  \\
&= \sum_{i=1}^d \frac{a_i}{2}\left(\min\{1 - x_i, T^{(i)}_2\} -\min\{x_i, T^{(i)}_1\}\right).
\end{align*}
Squaring the above expression, taking the expectation with respect to the random tessellation, and applying Jensen's inequality gives
\begin{align*}
&\EE_{\mathcal{P}}[(\bar{f}_{\lambda}(x) - f(x))^2] \geq \EE\left[\left(\sum_{i=1}^d  \frac{a_i}{2}\left(\min\{1 - x_i, T^{(i)}_2\} -\min\{x_i, T^{(i)}_1\}\right)\right)^2\right] \\
&= \sum_{i=1}^d \frac{a^2_i}{4} \EE\left[\left(\min\{1 - x_i, T^{(i)}_2\} -\min\{x_i, T^{(i)}_1\}\right)^2\right] \\
&+ \sum_{i,j=1: i \neq j}^d \frac{a_ia_j}{4} \EE\left[\min\{1 - x_i, T^{(i)}_2\} -\min\{x_i, T^{(i)}_1\}\right]\EE\left[\min\{1 - x_j, T^{(j)}_2\} -\min\{x_j, T^{(j)}_1\}\right] \\
&= \sum_{i=1}^d \frac{a^2_i}{4} \left(\EE\left[\min\{1 - x_i, T^{(i)}_2\}^2 - 2\min\{1 - x_i, T^{(i)}_2\}\min\{x_j, T^{(j)}_1\} + \min\{x_i, T^{(i)}_1\}^2\right]\right) \\
&+ \sum_{i,j=1: i \neq j}^d \frac{a_ia_j}{4} \EE\left[\min\{1 - x_i, T^{(i)}_2\} -\min\{x_i, T^{(i)}_1\}\right]\EE\left[\min\{1 - x_j, T^{(j)}_2\} -\min\{x_j, T^{(j)}_1\}\right].
\end{align*}
For the terms in the sum above, we have for any $t \in [0,1]$ and $T \sim \mathrm{Exponential}(\lambda \omega_i)$,
\begin{align*}
    \EE[\min\{t, T\}] &= \int_0^{\infty}  \PP(\min\{t, T\} \geq r)  \dint r = \int_0^{\infty}\PP(T \geq r)1_{\{t \geq r\}} \dint r \\
    &=\int_0^{t}e^{-\lambda \omega_i r}\dint r = \frac{1}{\lambda \omega_i}\left(1 - e^{-\lambda \omega_i t}\right).
\end{align*}
Also,
\begin{align*}
 \EE[\min\{t, T\}^2] &= 2\int_0^{\infty} r\PP(\min\{t, T\} \geq r)  \dint r = 2\int_0^{\infty} r\PP(T \geq r)1_{\{t \geq r\}} \dint r \\
    &= 2\int_0^{t} re^{-\lambda \omega_i r}\dint r = \frac{2}{\lambda^2\omega_i^2} - \frac{2}{\lambda^2\omega_i^2}e^{-\lambda \omega_i t} - \frac{2t}{\lambda \omega_i}e^{-t\lambda \omega_i}.   
\end{align*}
Plugging these moments into the above bound and taking the expectation with respect to $X$ gives
\begin{align*}
&\EE\left[(\bar{f}_{\lambda}(X) - f(X))^2\right] \\
&\geq \sum_{i=1}^d \frac{a^2_i}{4} \bigg(\frac{4}{\lambda^2\omega_i^2} - \frac{2}{\lambda^2\omega_i^2}\EE\left[e^{-\lambda \omega_i(1 - X_i)}\right] - \frac{2}{\lambda \omega_i}\EE\left[(1-X_i)e^{-(1-X_i)\lambda \omega_i}\right] \\
& - \frac{2}{\lambda^2 \omega^2_i}\EE\left[1 - e^{-\lambda \omega_i (1-X_i)}\right]\EE\left[1 - e^{-\lambda \omega_i X_i}\right] - \frac{2}{\lambda^2\omega_i^2}\EE\left[e^{-\lambda \omega_i X_i}\right] - \EE\left[\frac{2X_i}{\lambda \omega_i}e^{-X_i\lambda \omega_i}\right]\bigg) \\
% &= \sum_{i=1}^d \frac{a^2_i}{4} \bigg[\frac{2}{\lambda^2\omega_i^2} - \frac{2}{\lambda^3\omega_i^3}\left(1 - e^{-\lambda\omega_i}\right) - \left(\frac{2}{\lambda^3 \omega^3_i}  - \frac{2}{\lambda^3\omega_i^3}e^{-\lambda \omega_i} - \frac{2}{\lambda^2 \omega^2_i}e^{-\lambda \omega_i}\right) \\
% & \qquad - \frac{2}{\lambda^2 \omega^2_i}\left(1 - \frac{1}{\lambda \omega_i}\left(1 - e^{-\lambda\omega_i}\right)\right)^2 \\
% & \qquad +\frac{2}{\lambda^2\omega_i^2} - \frac{2}{\lambda^3\omega_i^3}\left(1 - e^{-\lambda\omega_i}\right) - \left(\frac{2}{\lambda^3 \omega^3_i}  - \frac{2}{\lambda^3\omega_i^3}e^{-\lambda \omega_i} - \frac{2}{\lambda^2 \omega^2_i}e^{-\lambda \omega_i}\right)\bigg] \\
&= \sum_{i=1}^d \frac{a^2_i}{4} \bigg[\frac{4}{\lambda^2\omega_i^2} - \frac{8}{\lambda^3\omega_i^3} + \frac{8}{\lambda^3\omega_i^3}e^{-\lambda \omega_i} + \frac{4}{\lambda^2 \omega^2_i}e^{-\lambda \omega_i} \\
& - \frac{2}{\lambda^2 \omega^2_i} + \frac{4}{\lambda^3 \omega^3_i}\left(1 - e^{-\lambda\omega_i}\right) - \frac{2}{\lambda^4 \omega^4_i}\left(1 - e^{-\lambda\omega_i}\right)^2\bigg] \\
&= \sum_{i=1}^d \frac{a^2_i}{4} \bigg[\frac{2}{\lambda^2\omega_i^2} - \frac{4}{\lambda^3\omega_i^3} + \frac{4}{\lambda^3\omega_i^3}e^{-\lambda \omega_i} + \frac{4}{\lambda^2 \omega^2_i}e^{-\lambda \omega_i}  - \frac{2}{\lambda^4 \omega^4_i}\left(1 - e^{-\lambda\omega_i}\right)^2\bigg] \\
&\geq \sum_{i=1}^d \frac{a^2_i}{2\lambda^2\omega_i^2} \bigg(1- \frac{2}{\lambda\omega_i} - \frac{1}{\lambda^2 \omega^2_i}\bigg),
\end{align*}
where we have used the independence of the $X_i$'s and the following intergal evaluations:
\begin{align*}
 \int_{0}^1 e^{-\lambda \omega_i t} \dint t =  \int_{0}^1 e^{-\lambda\omega_i(1-t)} \dint t  = \frac{1}{\lambda \omega_i}\left(1 - e^{-\lambda\omega_i}\right),
\end{align*}
and
\begin{align*}
\int_{0}^1 te^{-\lambda \omega_i t} \dint t =  \int_{0}^1 (1-t)e^{-\lambda\omega_i(1-t)} \dint t  = \frac{1}{\lambda^2 \omega^2_i}  - \frac{1}{\lambda^2\omega_i^2}e^{-\lambda \omega_i} - \frac{1}{\lambda \omega_i}e^{-\lambda \omega_i}.  
\end{align*}
%\textcolor{red}{NEED TO FINISH}

Next we obtain a lower bound for the variance term. Recall that if no inputs $\{X_1, \ldots, X_n\}$ fall in $Z_x^{\lambda}$, then we assume the estimator $\hat{f}_{n}(x) = 0$. For each $C \in \mathcal{P}(\lambda)$, let $\mathcal{N}_n(C) = \sum_{i=1}^n 1_{\{X_i \in C\}}$ be the number of covariates inside $C$ and let $p_{\lambda,C} := \PP_X(X \in C)$. Then,
\begin{align*}
&\EE_{\mathcal{D}_n}\left[(\bar{f}_{\lambda}(x) - \hat{f}_{n}(x))^2\right] \\ &= \int_{\RR^d} \sum_{C \in \mathcal{P}(\lambda)} 1_{\{x \in C\}} \EE_{\mathcal{D}_n} \left[ \left(\EE_X[f(X)| X \in C] - \frac{\sum_{i=1}^n Y_i1_{\{X_i \in C\}}}{\mathcal{N}_n(C)}\right)^2\right] \dint \mu(x) \\
     &= \sum_{\substack{C \in \mathcal{P}(\lambda): \\ C \cap \mathrm{supp}(\mu) \neq \emptyset}}1_{\{x \in C\}}  \EE_{\mathcal{D}_n}\left[\left(\EE_X[f(X)| X \in C] - \frac{\sum_{i=1}^n Y_i1_{\{X_i \in C\}}}{\mathcal{N}_n(C)}\right)^2\right].
\end{align*}
For the expectation in the sum, we have
\begin{align*}
&\EE_{\mathcal{D}_n}\left[\left(\EE_X[f(X)| X \in C] - \frac{\sum_{i=1}^n Y_i1_{\{X_i \in C\}}}{\mathcal{N}_n(C)}\right)^2\right] \\
    &= \EE_{\mathcal{D}_n}\left[\EE_X[f(X)| X \in C]^2 - 2\EE_X[f(X)| X \in C]\frac{\sum_{i=1}^n Y_i1_{\{X_i \in C\}}}{\mathcal{N}_n(C)} + \left(\frac{\sum_{i=1}^n Y_i1_{\{X_i \in C\}}}{\mathcal{N}_n(C)}\right)^2\right]
\end{align*}
As in the proof of Lemma 15 in \cite{OReillyTran2021minimax},
\begin{align*}
&\EE_{\mathcal{D}_n}\left[\left(\EE_X[f(X)| X \in C] - \frac{\sum_{i=1}^n Y_i1_{\{X_i \in C\}}}{\mathcal{N}_n(C)}\right)^2\right] \\
    &= \sum_{k=1}^n\PP(\mathcal{N}_n(C) = k)k^{-1} \left(\EE_{X}[f(X)^2 | X \in C] - \EE_X[f(X)| X \in C]^2) + \sigma^2\right) \\
    & \qquad + \PP(\mathcal{N}_n(C) = 0) \EE_X[f(X)| X \in C]^2.
\end{align*}
Now, define the random variables $\tilde{\mathcal{N}}_n(C) := \mathcal{N}_n(C) + 1_{\{\mathcal{N}_n(C) = 0\}}$. Then, by Jensen's inequality,
\begin{align*}
&\EE_{\mathcal{D}_n}\left[\left(\EE_X[f(X)| X \in C] - \frac{\sum_{i=1}^n Y_i1_{\{X_i \in C\}}}{\mathcal{N}_n(C)}\right)^2\right] \\
& \geq \sigma^2 \left(\sum_{k=1}^n\PP(\mathcal{N}_n(C) = k)k^{-1} + \PP(\mathcal{N}_n(C)= 0)\right)  \\
    &= \sigma^2 \EE[\tilde{\mathcal{N}}_n(C)^{-1}] \geq \sigma^2 \EE[\tilde{\mathcal{N}}_n(C)]^{-1} \\
    &= \sigma^2 \left(np_{\lambda, C} + (1 - p_{\lambda, C})^n\right)^{-1} \geq \sigma^2 \left(np_{\lambda, C} + 1\right)^{-1}. 
\end{align*}
Thus, taking the expectation with respect to the random tessellation $\mathcal{P}$ gives the lower bound
\begin{align*}  
\EE_{\mathcal{P}, \mathcal{D}_n}\left[(\bar{f}_{\lambda}(x) - \hat{f}_{n}(x))^2\right] & \geq \sigma^2\EE_{\mathcal{P}}\left[\sum_{\substack{C \in \mathcal{P}(\lambda): \\ C \cap \mathrm{supp}(\mu) \neq \emptyset}}1_{\{x \in C\}}  \left(np_{\lambda, C} + 1\right)^{-1}\right] \\
&= \sigma^2\EE_{\mathcal{P}}\left[ \left(n\PP_X(X \in Z_x^{\lambda}) + 1\right)^{-1}\right] \\
&\geq \sigma^2\EE\left[ \left(n\mathrm{vol}_d(Z_x^{\lambda} \cap [0,1]^d) + 1\right)^{-1}\right],
\end{align*}
and then By Jensen's inequality,
\begin{align*}
\EE_{\mathcal{P}, \mathcal{D}_n}\left[(\bar{f}_{\lambda}(x) - \hat{f}_{n}(x))^2\right] &\geq \sigma^2\left(n\EE\left[ \mathrm{vol}_d(Z_x^{\lambda} \cap [0,1]^d)\right] + 1\right)^{-1} \\
&\geq \sigma^2\left(n\EE\left[ \mathrm{vol}_d(Z_x^{\lambda})\right] + 1\right)^{-1} \\
&= \sigma^2\left(\frac{n}{2^d\lambda^d\Pi_{i\in [d]}\omega_i} + 1\right)^{-1}. % \\
%&= \frac{\sigma^2\lambda^d}{2n \Pi_{i=1}^d \omega_i} \left(1 + \lambda^{d}/\left(2n \Pi_{i=1}^d \omega_i\right)\right)^{-1}.
\end{align*}
Combining the lower bounds on the bias and the variance with \eqref{e:bias-var} gives the final result.

\end{proof}

\subsection{Proofs of Proposition \ref{p:Mondrian_image} and Corollary \ref{cor:obliqueMondrian_zerocell}}

\begin{proof}[Proof of Proposition \ref{p:Mondrian_image}]
In \cite{Nagel2005}, Lemma 4 and Corollary 1 show that the capacity functional for the cell boundaries of a STIT tessellation is determined by an associated intensity measure on the space of hyperplanes $\mathcal{H}^{d}$. Note that $\mathcal{Y}_A(\lambda)$ has associated intensity measure
\begin{align}\label{e:LambdaU}
\lambda\Lambda_A(\cdot) = \lambda\sum_{i=1}^m \frac{\|a_i\|_2}{\|A\|_{2,1}} \int_{\RR}  \mathbf{1}\{H_d\left(\frac{a_i}{\|a_i\|_2}, t\right) \in \cdot\} \dint t,
\end{align}
where $H_d(u,t):= \{x \in \RR^d : \langle x, u \rangle = t\}$.
The space $\mathcal{H}^{d}$ is equipped with the hit-miss topology, which is generated by sets of the following form: for Borel sets $B \subset \RR^d$, %define
\[[B] := \{H \in \mathcal{H}^d : H \cap B \neq \emptyset\}.\]
By Lemma 4 in \cite{Nagel2005}, it suffices to define $\Lambda$ on sets of the form $[C]$ for convex bodies $C \subset \RR^d$. Thus, it is sufficient to show that for any convex body set $C \subset \RR^d$,
\[\lambda\Lambda_{A}([C]) = \frac{m\lambda}{\|A\|_{2,1}}\Lambda_{M}([A^T(C)]),\]
where $\Lambda_M$ is the intensity measure on $\mathcal{H}^d$ associated to the Mondrian tessellation with unit lifetime.

Let $\{e_i\}_{i=1}^m$ denote the standard basis in $\RR^m$ and $C$ a convex body in $\RR^d$. First, note that $H_m(e_i, t) \cap A^T(C) \neq \emptyset$ if and only if
\begin{align*}
h_{A^T(C)}(-e_i) \leq t \leq h_{A^T(C)}(e_i).
\end{align*}
Then, noting that $$h_{A^T(C)}( \pm \, e_i) = h_{C}(\pm \, Ae_i) = \|Ae_i\|_2 h_{C}(\pm \, Ae_i /\|Ae_i\|_2) = \|a_i\|_2 h_{C}(\pm \, a_i /\|a_i\|_2),$$ the above inequality is equivalent to the inequality
\begin{align*}
h_{C}(- a_i /\|a_i\|_2) \leq \frac{t}{\|a_i\|_2} \leq h_{C}( a_i /\|a_i\|_2),
\end{align*}
These inequalities hold if and only if $H_d(a_i /\|a_i\|_2, t/\|a_i\|_2) \cap C \neq \emptyset$. Thus,
\begin{align*}
\frac{m\lambda}{\|A\|_{2,1}}\Lambda_{M}([A^T(C)]) &= \frac{\lambda}{\|A\|_{2,1}} \sum_{i=1}^m \int_{\RR} 1_{\{H_m(e_i, t) \cap A^T(C) \neq \emptyset\}} \dint t \\
&= \frac{\lambda}{\|A\|_{2,1}} \sum_{i=1}^m \int_{\RR} 1_{\{ H_d(a_i /\|a_i\|_2, t/\|a_i\|_2) \cap C \neq \emptyset\}} \dint t \\
&= \lambda \sum_{i=1}^m \frac{\|a_i\|_2}{\|A\|_{2,1}}  \int_{\RR} 1_{\{ H_d(a_i /\|a_i\|_2, r) \cap C \neq \emptyset\}} \dint r = \lambda\Lambda_{A}([C]). 
\end{align*}
\end{proof}

\begin{proof}[Proof of Corollary \ref{cor:obliqueMondrian_zerocell}]
Recall that the distribution of a random convex body containing the origin is determined by the set of containment probabilities $\PP(K \subseteq Z)$ for all convex bodies $K$ containing the origin. For the zero cell of a STIT tessellation with associated intensity measure $\Lambda$,
\[\PP(K \subseteq Z_0) = \PP(Y \cap K = \emptyset) = e^{-\Lambda([K])}.\]
Thus the statement follows from the fact we showed above that for any convex body $C \subset \RR^d$,
\[\Lambda_A([C]) = \frac{d}{\|A\|_{2,1}}\Lambda_{M}([A^T(C)]),\]
where $\Lambda_M$ is the intensity measure on $\mathcal{H}^d$ associated to $Y_M$, since this implies
\begin{align*}
\PP(K \subseteq Z_0) &= e^{-\Lambda_A([K])} = e^{-\frac{d}{\|A\|_{2,1}}\Lambda_M([A^T(K)])} \\
&= \PP\left(A^T(K) \subseteq Z_0^{(M)}\right) = \PP\left(A^T(K) \subseteq Z_0^{(M)} \cap \mathrm{ran}(A^T)\right) \\
&= \PP\left(K \subseteq (A^T)^{+}(Z_0^{(M)} \cap \mathrm{ran}(A^T))\right) = \PP\left(K \subseteq (A^{+})^T(Z_0^{(M)} \cap \mathrm{ran}(A^T))\right),
\end{align*}
where $A^+$ is the Moore-Penrose pseudoinverse of $A^T$.
\end{proof}
\end{appendix}

\end{document}